\documentclass[12pt,letterpaper,reqno]{amsart}
\usepackage{amssymb}
\usepackage{cases}
\usepackage{amsmath}
\usepackage{mathrsfs}
\usepackage{amssymb}
\usepackage{amsfonts}
\usepackage{graphicx}
\usepackage{wrapfig}
\usepackage{stmaryrd}
\usepackage{abstract}
\usepackage{mathrsfs}
\allowdisplaybreaks

\newcommand{\pp}{\ensuremath{\mathbb{P}}}
\newcommand\be{\begin{equation}}
\newcommand\ee{\end{equation}}
\newcommand\bea{\begin{eqnarray}}
\newcommand\eea{\end{eqnarray}}
\newcommand\bi{\begin{itemize}}
\newcommand\ei{\end{itemize}}
\newcommand\ben{\begin{enumerate}}
\newcommand\een{\end{enumerate}}

\newcommand{\twocase}[5]{#1 \begin{cases} #2 & \text{{\rm #3}}\\ #4 &\text{{\rm #5}} \end{cases}   }
\newcommand{\ncr}[2]{{#1 \choose #2}}

\newtheorem{thm}{Theorem}[section]

\newtheorem{lem}[thm]{Lemma}
\newtheorem{rek}[thm]{Remark}

\newtheorem{defi}[thm]{Definition}
\newtheorem{theorem}{Theorem}[section]
\newtheorem{proposition}[theorem]{Proposition}
\newtheorem{lemma}[theorem]{Lemma}
\newtheorem{corollary}[theorem]{Corollary}

\newtheorem{claim}[theorem]{Claim}
\newtheorem{definition}[theorem]{Definition}

\newtheorem*{cond1}{Condition 1}
\newtheorem*{cond2}{Condition 2}
\newtheorem*{case1}{Case 1}
\newtheorem*{case2}{Case 2}

\numberwithin{equation}{section}
\newtheorem{remark}{Remark}
\newtheorem*{remark*}{Remark}
\numberwithin{remark}{section}
\numberwithin{subsubcase}{subcase}
\numberwithin{subsubsection}{subsection}

\begin{document}

\title{From Fibonacci Numbers to Central Limit Type Theorems}

\date{\today}

\author{Steven J. Miller}\email{Steven.J.Miller@williams.edu}
\address{Department of Mathematics and Statistics, Williams College,
Williamstown, MA 01267}

\author{Yinghui Wang}\email{yinghui@mit.edu}
\address{Department of Mathematics, MIT, Cambridge, MA 02139}



\thanks{The first named author was partially supported by NSF grant DMS0970067 and the second named author was partially supported by NSF grant DMS0850577, Williams College and the MIT Mathematics Department. It is a pleasure to thank our colleagues from the Williams College 2010 SMALL REU program for many helpful conversations, especially Ed Burger, David Clyde, Cory Colbert, Carlos Dominguez, Gene Kopp, Murat Kolo$\breve{{\rm g}}$lu, Gea Shin and Nancy Wang, and Ed Scheinerman for useful comments on an earlier draft. We thank the referees for many helpful suggestions and comments.}

\maketitle

\begin{abstract}
A beautiful theorem of Zeckendorf states that every integer can be written uniquely as a sum of non-consecutive Fibonacci numbers $\{F_n\}_{n=1}^{\infty}$. Lekkerkerker proved that the average number of summands for integers in $[F_n, F_{n+1})$ is $n/(\varphi^2 + 1)$, with $\varphi$ the golden mean. This has been generalized to the following: given nonnegative integers $c_1,c_2,\dots,c_L$ with $c_1,c_L>0$ and recursive sequence $\{H_n\}_{n=1}^{\infty}$ with $H_1=1$, $H_{n+1} =c_1H_n+c_2H_{n-1}+\cdots +c_nH_1+1$  $(1\le n< L)$ and $H_{n+1}=c_1H_n+c_2H_{n-1}+\cdots +c_LH_{n+1-L}$ $(n\geq L)$, every positive integer can be written uniquely as $\sum a_iH_i$ under natural constraints on the $a_i$'s, the mean and the variance of the numbers of summands for integers in $[H_{n}, H_{n+1})$ are of size $n$, and the distribution of the numbers of summands converges to a Gaussian as $n$ goes to the infinity. Previous approaches used number theory or ergodic theory. We convert the problem to a combinatorial one. In addition to re-deriving these results, our method generalizes to a multitude of other problems (in the sequel paper \cite{BM} we show how this perspective allows us to determine the distribution of gaps between summands in decompositions). For example, it is known that every integer can be written uniquely as a sum of the $\pm F_n$'s, such that every two terms of the same (opposite) sign differ in index by at least 4 (3). The presence of negative summands introduces complications and features not seen in previous problems. We prove that the distribution of the numbers of positive and negative summands converges to a bivariate normal with computable, negative correlation, namely $-(21-2\varphi)/(29+2\varphi) \approx -0.551058$.
\end{abstract}

\setcounter{tocdepth}{1}

\tableofcontents


\section{Introduction}

\subsection{History}\label{history}

The Fibonacci numbers have intrigued mathematicians for hundreds of years. One of their most interesting properties is the Zeckendorf decomposition. Zeckendorf \cite{Ze} proved that every positive integer can be written uniquely as a sum of non-consecutive Fibonacci numbers (called the \textit{Zeckendorf decomposition})\label{Zeckendorfdec}, where the Fibonacci numbers\footnote{If we used the standard counting, then 1 would appear twice and numerous numbers would not have a unique decomposition.} are $F_1=1$, $F_2=2$, $F_3=3$, $F_4=5$, $\dots$\label{F_n}. Lekkerkerker extended this result and proved that the average number of summands needed to represent an integer in $[F_n, F_{n+1})$ is $\frac{n}{\varphi^2+1}+O(1)\approx 0.276n$\label{varphi}, where $\varphi=\frac{\sqrt{5}+1}{2}$ is the golden mean.
There is a related question: \emph{how are the number of summands distributed about the mean for integers in $[F_n, F_{n+1})$?} This is a very natural question to ask. Both the question and the answer are reminiscent of the Erd\H{o}s-Kac Theorem \cite{EK}, which states that as $n\to\infty$ the number of distinct prime divisors of integers on the order of size $n$ tends to a Gaussian with mean $\log\log n$ and standard deviation $\sqrt{\log\log n}$.


We first set some notation before describing the previous results.

\begin{defi}\label{defn:goodrecurrencereldef}\label{def:goodrecurrence} We say a sequence $\{H_n\}_{n=1}^\infty$ of positive integers is a \textbf{Positive Linear Recurrence Sequence (PLRS)} if the following properties hold:

\ben
\item \emph{Recurrence relation:} There are non-negative integers $L, c_1, \dots, c_L$\label{c_i} such that $$H_{n+1} \ = \ c_1 H_n + \cdots + c_L H_{n+1-L},$$ with $L, c_1$ and $c_L$ positive.
\item \emph{Initial conditions:} $H_1 = 1$, and for $1 \le n < L$ we have
$$H_{n+1} \ =\
c_1 H_n + c_2 H_{n-1} + \cdots + c_n H_{1}+1.$$
\een

We call a decomposition $\sum_{i=1}^{m} {a_i H_{m+1-i}}$\label{a_i} of a positive integer $N$ (and the sequence $\{a_i\}_{i=1}^{m}$) \textbf{legal}\label{legal} if $a_1>0$, the other $a_i \ge 0$, and one of the following two conditions holds:

\begin{cond1}\label{legalcond1}
We have $m<L$ and $a_i=c_i$ for $1\le i\le m$.
\end{cond1}

\begin{cond2}
There exists $s\in\{1,\dots, L\}$ such that
\begin{equation}\label{eq:legalcondition2}
a_1\ = \ c_1,\ a_2\ = \ c_2,\ \cdots,\ a_{s-1}\ = \ c_{s-1}\ {\rm{and}}\ a_s<c_s,
\end{equation}
$a_{s+1}, \dots, a_{s+\ell} \ = \  0$ for some $\ell \ge 0$,
and $\{b_i\}_{i=1}^{m-s-\ell}$ (with $b_i = a_{s+\ell+i}$) is legal.
\end{cond2}

If $\sum_{i=1}^{m} {a_i H_{m+1-i}}$ is a legal decomposition of $N$, we define the \textbf{number of summands}\label{summands} (of this decomposition of $N$) to be $a_1 + \cdots + a_m$.
\end{defi}

Informally, a legal decomposition is one where we cannot use the recurrence relation to replace a linear combination of summands with another summand, and the coefficient of each summand is appropriately bounded; other authors \cite{DG,Ste1} use the phrase $G$-ary decomposition for a legal decomposition, and sum-of-digits or summatory function for the number of summands. For example, if $H_{n+1} = 2 H_n + 3 H_{n-1} + H_{n-2}$, then $H_5 + 2 H_4 + 3 H_3 + H_1$ is legal, while $H_5 + 2 H_4 + 3 H_3 + H_2$ is not (we can replace $2 H_4 + 3 H_3 + H_2$ with $H_5$), nor is $7H_5 + 2H_2$ (as the coefficient of $H_5$ is too large). Note the Fibonacci numbers are just the special case of $L=2$ and $c_1 = c_2 = 1$.

The following probabilistic language will be convenient for stating some of the results.

\begin{defi}[Associated Probability Space to a Positive Linear Recurrence Sequence]\label{def:assocprobspace}
Let $\{H_n\}$ be a Positive Linear Recurrence Sequence. For each $n$, consider the discrete outcome space \be \Omega_n \ = \ \{H_n,\ H_n+1,\ H_n + 2,\  \dots,\ H_{n+1}-1\} \ee with probability measure \be \pp_n(A) \ = \ \sum_{\omega \in A \atop \omega \in \Omega_n} \frac1{H_{n+1}-H_n}, \ \ \  A \subset \Omega_n; \ee in other words, each of the $H_{n+1}-H_n$ numbers is weighted equally. We define the random variable $K_n$\label{K_n} by setting $K_n(\omega)$ equal to the number of summands of $\omega \in \Omega_n$ in its legal decomposition. Implicit in this definition is that each integer has a unique legal decomposition; we prove this in Theorem \ref{thm:genZeckendorf}, and thus $K_n$ is well-defined.

We denote the cardinality of $\Omega_n$\label{index:Deltan}  by \be \Delta_n\ =\ H_{n+1}-H_n,\ee and we set $p_{n,k}$\label{pnk} equal to the number of elements in $[H_n, H_{n+1})$ whose generalized Zeckendorf decomposition has exactly $k$ summands; thus \be p_{n,k}\ =\ \Delta_n \cdot {\rm Prob}(K_n=k).\ee
\end{defi}

We first review previous results and methods, and then describe our new perspective and extensions. See \cite{Bu,Ha,Ho,Ke,Len} for more on generalized Zeckendorf decompositions, \cite{GT} for a proof of Theorems \ref{thm:genZeckendorf} and \ref{thm:genlekkerkerker}, and \cite{DG,FGNPT,GTNP,LT,Ste1} for a proof and some generalizations of Theorem \ref{thm:Gaussian}.

\begin{theorem}[Generalized Zeckendorf's Theorem for PLRS]\label{thm:genZeckendorf} Let $\{H_n\}_{n=1}^\infty$ be a \emph{Positive Linear Recurrence Sequence}. Then

{\rm{(a)}} There is a unique legal decomposition for each positive integer $N\ge 0$.

{\rm{(b)}} There is a bijection between the set $\mathcal{S}_n$\label{mathcalS_n} of integers in $[H_n, H_{n+1})$ and the set $\mathcal{D}_n$\label{mathcalD_n} of legal decompositions $\sum_{i=1}^{n} {a_i H_{n+1-i}}$.
\end{theorem}


\begin{theorem}[Generalized Lekkerkerker's Theorem for PLRS]\label{thm:genlekkerkerker} Let $\{H_n\}_{n=1}^\infty$ be a \emph{Positive Linear Recurrence Sequence}, let $K_n$ be the random variable of Definition \ref{def:assocprobspace} and denote its mean by $\mu_n$. Then there exist constants $C>0$, $d$ and $\gamma_1\in (0,1)$ depending only on $L$ and the $c_i$'s in the recurrence relation of the $H_n$'s such that
\begin{equation}\label{eq:genlekkerkerker}
\mu_n\ =\ Cn+d+o(\gamma_1^n).
\end{equation}
\end{theorem}


\begin{theorem}[Gaussian Behavior for PLRS]\label{thm:Gaussian} Let $\{H_n\}_{n=1}^\infty$ be a \emph{Positive Linear Recurrence Sequence} and let $K_n$ be the random variable of Definition \ref{def:assocprobspace}. The mean $\mu_n$ and variance $\sigma_n^2$ of $K_n$ grow linearly in $n$, and $(K_n-\mu_n)/\sigma_n$ converges weakly to the standard normal $N(0,1)$ as $n\rightarrow \infty$.
\end{theorem}

While the proof of Theorem \ref{thm:Gaussian} becomes very technical in general, the special case $L=1$ is straightforward, and suggests why the result should hold. When $L=1$, $H_n=c^{n-1}_1$. Thus our PLRS is just the geometric series $1, c_1, c_1^2, \dots$, and a legal decomposition of $N$ is just its base $c_1$ expansion. Hence every positive integer has a unique legal decomposition. Further, the distribution of the number of summands converges to a Gaussian by the Central Limit Theorem, as we essentially have the sum of $n-1$ independent, identically distributed discrete uniform random variables.\footnote{Writing $N = a_1 c_1^n + \cdots + a_{n+1} 1$, we are interested in the large $n$ behavior of $a_1 + \cdots + a_{n+1}$ as we vary over $N$ in $[c_1^n, c_1^{n+1})$. Note for large $n$ the contribution of $a_1$ is immaterial, and the remaining $a_i$'s can be understood by considering the sum of $n$ independent, identically distributed discrete uniform random variables on $\{0, \dots, B-1\}$ (which have mean $\frac{B-1}2$ and standard deviation $\sqrt{(c_1^2-1)/12}$). Denoting these by $A_i$\label{A_i}, by the Central Limit Theorem $A_2 + \cdots + A_{n+1}$ converges to being normally distributed with mean $\frac{c_1-1}2 n$ and standard deviation $n\sqrt{(c_1^2-1)/12}$.}\label{footnote:L=1}

Previous approaches to this problem used number theory or ergodic theory, often requiring the analysis of certain exponential sums. We recast this as a combinatorial problem, deriving formulas for the cardinality of numbers in our interval with exactly a given number of summands. We are able to re-derive the above results from this perspective. As Our method generalizes to a multitude of other problems (in a sequel paper we use the combinatorial vantage to determine the distribution of gaps between summands). For the main part of this paper, we concentrate on one particularly interesting situation where features not present in previous works arise.

\begin{definition}\label{far}
We call a sum of the $\pm F_n$'s a {\textit{\textbf{far-difference representation}}}\label{far-difference representation} if every two terms of the same sign differ in index by at least 4, and every two terms of opposite sign differ in index by at least 3.
\end{definition}

Recently Alpert \cite{Al} proved the analogue of Zeckendorf's Theorem for the far-difference representation. It is convenient to set\label{S_n} \be \twocase{S_n \ = \ }{\sum_{0<n-4i\le n} F_{n-4i} \ = \ F_n + F_{n-4} + F_{n-8} + \cdots}{if $n > 0$}{0}{otherwise.} \ee

\begin{theorem}[Generalized Zeckendorf's Theorem for Far-Difference Representations]\label{thmhan}
Every integer has a unique far-difference representation. For each $N\in (S_{n-1}=F_n-S_{n-3}-1,S_n]$, the first term in its far-difference representation is $F_n$, and the unique far-difference representation of 0 is the empty representation.
\end{theorem}

Most results in the literature concern only one quantity, the number of summands. An exception is \cite{Ste2}, where the standard Zeckendorf expansion (called the greedy expansion) and the lazy expansion (which uses as many summands as possible) are simultaneously considered. Steiner proves that their joint distribution converges to a bivariate Gaussian with a correlation of $9 - 5\varphi \approx .90983$. Unlike the Zeckendorf expansions, the far-difference representations have both positive and negative summands, which opens up the fascinating question of how the number of each are related. In the result below we find a non-zero correlation between the two types of summands.

\begin{theorem}[Generalized Lekkerkerker's Theorem and Gaussian Behavior for Far-Difference Representations]\label{thm:lekgaussfardiff}
Let $\mathcal{K}_n$\label{mathcalK_n} and $\mathcal{L}_n$ be the corresponding random variables denoting the number of positive summands and the number of negative summands in the far-difference representation for integers in $(S_{n-1},S_n]$. As $n$ tends to infinity, $\mathbb{E}[\mathcal{K}_n] = \frac{1}{10}n+\frac{371-113\sqrt{5}}{40} + o(1)$, and is $\frac{\sqrt{5}+1}{4}=\frac{\phi}{2}$ greater than $\mathbb{E}[\mathcal{L}_n]$; the variance of both is of size $\frac{15+21\sqrt{5}}{1000}n$; the standardized joint density of $\mathcal{K}_n$ and $\mathcal{L}_n$ converges weakly to a bivariate Gaussian with negative correlation $\frac{10\sqrt{5}-121}{179}=-\frac{21-2\varphi}{29+2\varphi}\approx -0.551$; and $\mathcal{K}_n+\mathcal{L}_n$ and $\mathcal{K}_n-\mathcal{L}_n$ are independent.
\end{theorem}

\subsection{Sketch of Proofs}

By recasting the problem as a combinatorial one and using generating functions, we are able to re-derive and extend the previous results in the literature. The key techniques in our proof are generating functions, partial fractional expansions, differentiating identities and the method of moments. Unfortunately, in order to be able to handle a general Positive Linear Recurrence Sequence, the arguments become quite technical due to the fact that we cannot exploit any special properties of the coefficients of the recurrence relations, but rather must prove certain technical lemmas for \emph{any} choice of the $c_i$'s. We therefore quickly look at the special case of the Fibonacci numbers, as this highlights the main ideas of the method without many of the technicalities.\footnote{\label{footnote:prooffib} Actually, the proof can be simplified further for the Fibonacci numbers, as the key quantity $p_{n,k}$ equals $\ncr{n-k}{k-1}/F_{n-1}$, which by Stirling's formula tends to a random variable being normally distributed; see \cite{KKMW} for details. Unfortunately this approach does not generalize, as the formulas for $p_{n,k}$ become far more involved.}

Our method begins with a derivation of a recurrence relation for the $p_{n,k}$'s, which in this case is the number of integers in $[F_n,F_{n+1})$ with precisely $k$ summands in their legal decomposition (see Definition \ref{def:assocprobspace}). We find $p_{n+1,k+1}=p_{n,k+1}+p_{n,k}$. Multiplying both sides of this equation by $x^ky^n$, summing over $n,k>0$, and calculating the initial values of the $p_{n,k}$'s, namely $p_{1,1}$, $p_{2,1}$ and $p_{2,2}$, we obtain a formula for the generating function $\sum_{n,k>0} p_{n,k}x^ky^n$:
\begin{equation}\label{scrGfib}
\mathscr{G}(x,y)\ :=\
 \sum_{n,k>0} p_{n,k}x^ky^n\ = \ \frac{xy}{1-y-xy^2}.
\end{equation}
By partial fraction expansion, we write the right-hand side as\label{yix}
$$-\frac{y}{y_1(x)-y_2(x)}\left(\frac{1}{y-y_1(x)}-\frac{1}{y-y_2(x)}\right),$$
where $y_1(x)$ and $y_2(x)$\label{yixfib} are the roots of $1-y-xy^2=0$. Rewriting $\frac{1}{y-y_i(x)}$ as $-(1-\frac{y}{y_i(x)})^{-1}$ and using a power series expansion, we are able to compare the coefficients of $y^n$ of both sides of (\ref{scrGfib}). This gives an explicit formula for $g(x)=\sum_{k>0} p_{n,k}x^k$\label{gx}.

Note that \be g(1)\ = \ \sum_{k>0} p_{n,k}, \ee which is $F_{n+1} -F_n$ by definition. Further, we have  \be g'(1)\ = \ \sum_{k>0} kp_{n,k}\ = \ \mathbb{E}[K_n](F_{n+1} -F_n)\ = \ \mathbb{E}[K_n]g(1).\ee Therefore, once we determine $g(1)$ and $g'(1)$, we know $\mathbb{E}[K_n]$.

Letting $\mu_n = \mathbb{E}[K_n]$, we define the random variable $K'_n=K_n-\mu_n$.  We immediately obtain an explicit, closed form expression for $h_n(x) = g(x) - \mu_n$. Arguing as above we find $h_n(1)=F_{n+1}-F_n$ and $h_n'(1)=\mathbb{E}[K'_n]h_n(1)$. Furthermore, we get \be\label{eq:momentexpansionsfib} \left(xh_n'(x)\right)'\ = \ \mathbb{E}[{K'_n}^2]h_n(1), \ \ \ \left(x\left(xh_n'(x)\right)'\right)'\ = \ \mathbb{E}[{K'_n}^3]h_n(1), \ \ \ \dots,\ee which allows us to compute the moments of $K'_n$.

Let $\sigma_n$ denote the variance of $K_n$ (which is of course also the variance of $K_n'$), and recall that the $2m$\textsuperscript{th} moment of the standard normal is $(2m-1)!!=(2m-1)(2m-3)\cdots 1$\label{!!}. To show that $K_n$ converges to being normally distributed with mean $\mu_n$ and variance $\sigma_n$, it suffices to show that the $2m$\textsuperscript{th} moment of $K_n'/\sigma_n$ converges to $(2m-1)!!$ and the odd moments converge to 0. We are able to prove this through \eqref{eq:momentexpansionsfib}, which are repeated applications of differentiating identities to our partial fraction expansion of the generating function.

\vspace{0.1in}

We first generalize Zeckendorf's Theorem in Section \ref{sec:genzeck}. In Section \ref{sec:genfnprobden} we derive the formula for the generating function of the probability density, and then prove the generalized Lekerkerker's Theorem in Section \ref{sec:genlek}. We prove the Gaussian behavior for Positive Linear Recurrence Sequences in Section \ref{sec:GaussianBehavior},  and for the far-difference representation in Section \ref{hannah}. We conclude with some natural problems to consider. \\

\emph{For the convenience of the reader, we list the main notation and terminology of the paper in Appendix \ref{sec:notationdef}, along with the page number of the first occurrence or the definition.}


\section{Proof of Theorem \ref{thm:genZeckendorf} (Generalized Zeckendorf)}\label{sec:genzeck}

We need the following lemma about the legal decompositions in our proof.

\begin{lemma}\label{lemma1}
For $m\ge 1$, if $N=\sum_{i=1}^{m} {a_{m+1-i} H_i}$ is legal, then $N<H_{m+1}$.
\end{lemma}

\begin{proof}
We proceed by induction on $m$. The case of $m=1$ is trivial, as this implies $N=a_1 H_1=a_1\le c_1<H_2$. A similar argument proves the claim when $m < L$ (and we are in the Condition 1 case).

Suppose the lemma holds for any $m'<m$ ($m\ge 2$). From Definition \ref{def:goodrecurrence}, we see that there exists $1\le j\le L$ such that $a_{j}<c_{j}$. Let $j$ be the smallest number such that $a_{j}<c_{j}$. Since $\sum_{i=1}^{m-j-\ell+1} {a_{m+1-i} H_{i}}$ is legal for some $\ell>0$, by the induction hypothesis
  \begin{equation*}
  \sum_{i=1}^{m-j} {a_{m+1-i} H_{i}}\ = \   \sum_{i=1}^{m-j-\ell+1} {a_{m+1-i} H_{i}}\ < \ H_{m+1-j}.
  \end{equation*}
Therefore
\begin{eqnarray*}
\nonumber \sum_{i=1}^{m} {a_{m+1-i} H_i} &\ = \ & {\sum_{i=1}^{m-j} {a_{m+1-i} H_{i}}}+{\sum_{i=m-j+1}^{m} {a_{m+1-i} H_{i}}}\\
\nonumber &=& \sum_{i=1}^{m-j} {a_{m+1-i} H_{i}}+ a_{j}H_{m+1-j} + \sum_{i=1}^{j-1} {c_{i} H_{m+1-i}}\\
\nonumber &<& H_{m+1-j}+(c_{j}-1)H_{m+1-j}+ \sum_{i=1}^{j-1} {c_{i} H_{m+1-i}}\\
&=& \sum_{i=1}^{j} {c_{i} H_{m+1-i}}\ \le\ \sum_{i=1}^{L} {c_{i} H_{m+1-i}} \ =\ H_{m+1},
\end{eqnarray*}
where the last equality comes from Definition \ref{def:goodrecurrence}.
\end{proof}

The following result immediately follows from Lemma \ref{lemma1}.

\begin{corollary}
If $N\in [H_n,H_{n+1})$, then any legal decomposition of $N$ must be of the form $\sum a_iH_{n+1-i}$ with $a_1>0$.
\end{corollary}

We now prove Theorem \ref{thm:genZeckendorf}. The proof is a mostly straightforward (and somewhat tedious) induction on $n$.

\begin{proof}[Proof of Theorem \ref{thm:genZeckendorf}]

The case of $L=1$ is clearly true, since the legal decomposition is just the base $c_1$ decomposition. Assume now that $L\ge 2$. By defining $H_i = 0$ for $i<1$, for $1\le n< L$ we have
\begin{equation*}
H_{n+1} \ = \  c_1 H_n + c_2 H_{n-1} + \cdots + c_L H_{n-L+1} +1.
\end{equation*}
By Definition \ref{def:goodrecurrence}, for any $n\ge 1$  we have
\begin{equation}\label{rem1}
c_1 H_n + c_2 H_{n-1} + \cdots + c_L H_{n-L+1} \le H_{n+1} \le c_1 H_n + c_2 H_{n-1} + \cdots + c_L H_{n-L+1} +1.
\end{equation}

We call a legal decomposition {\textit{Type 1}} if it satisfies Condition 1 in Definition \ref{def:goodrecurrence} and {\textit{Type 2}} if it satisfies Condition 2. Note that Conditions 1 and 2 cannot hold at the same time. Further, if $N=0$ then it has a unique decomposition by the definition, so we may assume $N > 0$. To prove Theorem \ref{thm:genZeckendorf}(a), it suffices to show that there is a unique legal decomposition for every integer $N\in [H_n,H_{n+1})$ for all $n$. We proceed by induction on $n$.

For $n=1$, recall that $H_1=1$ and $H_2=1+c_1$. For any $N\in [H_1,H_2) = [1,1+c_1)$,
\begin{equation}\label{n10}
N\ = \ N\cdot 1\ = \  N\cdot H_1.
\end{equation}
Since $0<N\le c_1$, (\ref{n10}) is a legal decomposition of $N$. On the other hand, since $N<H_2$, (\ref{n10}) is the only legal decomposition of $N$. Therefore, there is a unique legal decomposition for every integer $N\in [H_1,H_2)$.

Assume that the statement holds for any $n'<n$ ($n\ge 2$). We first prove the existence of a decomposition for $N\in [H_n,H_{n+1})$.
If $n\ge L$, then $N<H_{n+1}=c_1H_n+c_2H_{n-1}+\cdots +c_LH_{n-L+1}$. Thus there exists a unique $s\in\{0,\dots, L-1\}$ such that
\begin{equation}\label{ss}
c_1H_n+c_2H_{n-1}+\cdots +c_sH_{n-s+1}\le N< c_1H_n+c_2H_{n-1}+\cdots +c_{s+1}H_{n-s}
\end{equation}
(if $s=0$ then the left-hand side is zero). Let $a_{s+1}$ be the unique integer such that
\begin{equation*}
a_{s+1}H_{n-s}\le N- \sum_{i=1}^s c_iH_{n-i+1}< (a_{s+1}+1)H_{n-s}.
\end{equation*}
Then $a_{s+1}<c_{s+1}$ and
\begin{equation*}
N'\ :=\ N- \sum_{i=1}^s c_iH_{n-i+1}-a_{s+1}H_{n-s}
 \ < \ H_{n-s}.
\end{equation*}
By the induction hypothesis, there exists a unique legal decomposition $\sum_{i=1}^m b_i H_{m+1-i}$ $(m<n-s)$ of $N'$. Hence
\begin{equation*}
\sum_{i=1}^s c_iH_{n-i+1}+a_{s+1}H_{n-s}+\sum_{i=1}^m b_i H_{m+1-i}
\end{equation*}
is a legal decomposition of $N$. The case when $n<L$ follows similarly.\footnote{If $n<L$ and there exists $s$ satisfying (\ref{ss}), then we can prove existence in the same way. If there does not exist such an $s$, then since $N<H_{n+1}=c_1H_n+c_2H_{n-1}+\cdots +c_nH_1+1$, i.e., $N\le c_1H_n+c_2H_{n-1}+\cdots +c_nH_1$, the equality must be achieved. Thus $\sum_{i=1}^n c_iH_{n-i+1}$ is a legal decomposition of $N$ as $n<L$.} This completes the proof of existence.

We prove uniqueness by contradiction. Assume there exist two distinct legal decompositions of $N$: $\sum_{i=1}^m a_i H_{m+1-i}$ and $\sum_{i=1}^{m'} a'_i H_{m'+1-i}$. First, since $0<H_n\le N<H_{n+1}$, we have $m,m'\le n$. On the other hand, by Lemma \ref{lemma1} we have $m,m'\ge n$. Hence $m=m'=n$. We have three cases in terms of the types of the above two decompositions.

\vspace{0.1in}

\noindent{\bf{Case 1.}}
If both decompositions are of Type 1, i.e., satisfy Condition 1, then they are the same since $m=m'$.

\vspace{0.1in}

\noindent{\bf{Case 2.}}
If both decompositions are of Type 2, let $s$ and $s'$ be the corresponding integers that satisfy Condition 2. We want to show that $s=s'$. Otherwise, we assume $s>s'$ without loss of generality (so $s' \le s-1$). Thus $a_i=c_i$ $(1\le i< s)$, $a_{s'}<c_{s'}$, $a'_i=c_i$ $(1\le i< s')$, $\sum_{i=s+\ell}^{n} a_i H_{n+1-i}$ and $\sum_{i=s'+\ell'}^{n} a'_i H_{n+1-i}$ are legal for some positive $\ell$ and $\ell'$. By Lemma \ref{lemma1}, we have $\sum_{i=s'+1}^{n} a'_i H_{n+1-i}=\sum_{i=s'+\ell'}^{n} a'_i H_{n+1-i}<H_{n-s'+1}$, thus
\begin{eqnarray}\label{ineq}
\nonumber \sum_{i=1}^{s-1} c_i H_{n+1-i}&\le& \sum_{i=1}^{n} a_i H_{n+1-i}\ = \ N\ = \ \sum_{i=1}^n a'_i H_{n+1-i}\\
\nonumber &\le& \sum_{i=1}^{s'-1} c_i H_{n+1-i}+(c_{s'}-1)H_{n-s'+1}+\sum_{i=s'+1}^{n} a'_i H_{n+1-i}\\
\nonumber &<& \sum_{i=1}^{s'-1} c_i H_{n+1-i}+(c_{s'}-1)H_{n-s'+1}+H_{n-s'+1}\\
&\ = \ &\sum_{i=1}^{s'} c_i H_{n+1-i}\ \le\ \sum_{i=1}^{s-1} c_i H_{n+1-i},
\end{eqnarray}
contradiction. Hence $s=s'$. As a result, $a_i=c_i=a'_i$ $(1\le i<s)$. Thus
\begin{equation}\label{N''}
a_sH_{n-s+1}+\sum_{i=s+\ell}^{n} a_i H_{n+1-i}\ = \ a'_sH_{n-s+1}+\sum_{i=s+\ell'}^{n} a'_i H_{n+1-i}.
\end{equation}
Since $\sum_{i=s+\ell}^{n} a_i H_{n+1-i}$ and $\sum_{i=s+\ell'}^{n} a'_i H_{n+1-i}$ are legal, they are less than $H_{n-s+1}$ by Lemma \ref{lemma1}. Let $N''$ be the value of both sides of (\ref{N''}), then there exist unique integers $q\ge 0$ and $r\in [0,H_{n-s+1})$, such that $N''=qH_{n-s+1}+r$. Therefore $a_s=q=a'_s$ and
\begin{equation*}
\sum_{i=s+\ell}^{n} a_i H_{n+1-i}\ = \ r\ = \ \sum_{i=s+\ell'}^{n} a'_i H_{n+1-i}.
\end{equation*}
Since $r<H_{n-s+1}$, there is, by induction, a unique legal decomposition of $r$. Hence $a_i=a'_i$ $(s+1\le i\le n)$. Thus we have $a_i=a'_i$ for any $i$, which leads to a contradiction that the two decompositions of $N$ are different.

\vspace{0.1in}

\noindent{\bf{Case 3.}}
If one of the decompositions is of Type 1 and the other one is of Type 2, without loss of generality we can assume that $\sum_{i=1}^n a'_i H_{n+1-i}$ is of Type 1 and $\sum_{i=1}^{n} a_i H_{n+1-i}$ is of Type 2 with the corresponding $s$ satisfying (\ref{eq:legalcondition2}). From (\ref{ineq}), we see that
\begin{equation*}
\sum_{i=1}^{n} a_i H_{n+1-i}\ < \ \sum_{i=1}^{s} c_i H_{n+1-i}\ \le\ \sum_{i=1}^{n} c_i H_{n+1-i}=N,
\end{equation*}
contradiction. This completes the proof of (a).

For (b), in the proof of (a) we showed that each $N$ has a unique legal decomposition of the form $\sum_{i=1}^{n} {a_i H_{n+1-i}}$, which induces an injective map $\sigma$ from $\mathcal{S}_n$ to $\mathcal{D}_n$.
On the other hand, by Lemma \ref{lemma1}, $H_n\le \sum_{i=1}^{n} {a_i H_{n+1-i}}<H_{n+1}$, therefore $|{\mathcal{D}_n}|\le H_{n+1}-H_n = |\mathcal{S}_n|$. Hence $\sigma $ is a bijective map.
\end{proof}


\section{Generating Function of the Probability Density}\label{sec:genfnprobden}

By Theorem \ref{thm:genZeckendorf}(b), $p_{n,k}$ is the number of legal decompositions of the form $\sum_{i=1}^{n} {a_{i} H_{n+1-i}}$ with $k=a_1+a_2+\cdots +a_n$ and $a_1>0$. In this section, we derive a recurrence relation for the $p_{n,k}$'s, and show that their generating function is $\mathscr{G}(x,y)=\sum_{n,k>0} p_{n,k} x^k y^n$. Unlike previous approaches to Lekkerkerker's theorem, our result is based on an analysis of how often there are exactly $k$ summands, and thus the results in this section are the starting point for our analysis (as well as the reason why we can prove Gaussian behavior).

\begin{proposition}\label{propG} Define
\begin{equation}\label{defs_m}
s_0=0,\ s'_0=1\ {\rm{and}}\ s'_m\ = \ s_m\ = \ c_1+c_2+\cdots +c_m,\ 1\le m\le L.
\end{equation}
The generating function $\mathscr{G}(x,y)=\sum_{n,k>0} p_{n,k} x^k y^n$ equals
\begin{equation*}
\mathscr{G}(x,y)\ = \ \frac{\mathscr{B}(x,y)}{\mathscr{A}(x,y)},
\end{equation*}
where
\begin{equation}\label{A}
\mathscr{A}(x,y)\ = \ 1-\sum_{m=0}^{L-1}\sum_{j=s_m}^{s_{m+1}-1} x^j y^{m+1}
\end{equation}
and
\begin{equation}\label{B}
\mathscr{B}(x,y)\ = \ \sum_{n\le L, k \ge 1}p_{n,k}x^k y^n-\sum_{m=0}^{L-1}\sum_{j=s_m}^{s_{m+1}-1} x^j y^{m+1} \sum_{n<L-m, k \ge 1}p_{n,k}x^k y^n.
\end{equation}
\end{proposition}

\begin{proof}
As the initial values of $p_{n,k}$'s, namely those with $n<L$, can be calculated directly, we assume $n\ge L$. For notational convenience, we say $N$ has a \emph{$k$ summand}\label{page:ksummand} decomposition if it has exactly $k$ summands in its legal decomposition.

\vspace{0.1in}

\noindent{\bf{Case 1.}} If $a_1<c_1$, let $i_2$ be the smallest integer greater than 1 such that $a_{i_2}>0$, then $H_n\le \sum_{i=1}^{n} {a_i H_{n+1-i}}$ is legal if and only if $\sum_{i=i_2}^{n} {a_i H_{n+1-i}}$ is. Since the number of legal $(k-a_1)$ summand decompositions of the form $\sum_{i=i_2}^{n} {a_{i} H_{n+1-i}}$ is $p_{n+1-i_2,k-a_1}$, the number of legal $k$ summand decompositions of the form $\sum_{i=1}^{n} {a_{i} H_{n+1-i}}$ with $a_1<c_1$ is
 \begin{equation*}
\sum_{a_1=1}^{c_1-1}\sum_{i_2=2}^{n} p_{n+1-i_2,k-a_1}\ = \ \sum_{j=1}^{c_1-1}\sum_{i=1}^{n-1} p_{i,k-j},
 \end{equation*}
where $p_{n,k}=0$ if $k\le 0$.

If instead $a_1=c_1$, then $a_2\le c_2$ by Definition \ref{def:goodrecurrence}.

\vspace{0.1in}

\noindent{\bf{Case 2.}} If $a_1=c_1$ and $a_2<c_2$, let $i_3$ be the smallest integer greater than 2 such that $a_{i_3}>0$, then $\sum_{i=1}^{n} {a_i H_{n+1-i}}$ is legal if and only if $\sum_{i=i_3}^{n} {a_i H_{n+1-i}}$ is. Note that $a_1=c_1$ and $a_2<c_2$. Since the number of legal $(k-c_1-a_2)$ summand decompositions of the form $\sum_{i=i_3}^{n} {a_{i} H_{n+1-i}}$ is $p_{n+1-i_3,k-c_1-a_2}$, the number of legal $k$ summand decompositions of the form $\sum_{i=1}^{n} {a_{i} H_{n+1-i}}$ with $a_1=c_1$ and $a_2<c_2$ is
 \begin{equation*}
\sum_{a_2=0}^{c_2-1}\sum_{i_3=3}^{n} p_{n+1-i_3,k-c_1-a_2}\ = \ \sum_{j=c_1}^{c_1+c_2-1}\sum_{i=1}^{n-2} p_{i,k-j}.
 \end{equation*}
If instead $a_i=c_i$ for $1\le i\le m<L$, we can repeat the above procedure. By Definition \ref{def:goodrecurrence}, we have $a_{m+1}\le c_{m+1}$.

\vspace{0.1in}

\noindent{\bf{Case {\boldmath $m+1\ (m\ge 1)$.}}}
If $a_i=c_i$ for $1\le i\le m<L$ and $a_{m+1}< c_{m+1}$, let $i_{m+2}$ be the smallest integer greater than $m+1$ such that $a_{i_{m+2}}>0$, then $\sum_{i=1}^{n} {a_i H_{n+1-i}}$ is legal if and only if $\sum_{i=i_{m+2}}^{n} {a_i H_{n+1-i}}$ is. Note that $a_i=c_i$ for $1\le i\le m<L$. Since the number of legal $(k-s_m-a_{m+1})$ summand decompositions of the form $\sum_{i=i_{m+2}}^{n} {a_{i} H_{n+1-i}}$ is $p_{n+1-i_{m+2},k-s_m-a_{m+1}}$, the number of legal $k$ summand decompositions of the form $\sum_{i=1}^{n} {a_{i} H_{n+1-i}}$ with $a_i=c_i$ for $1\le i\le m<L$ and $a_{m+1}< c_{m+1}$ is
 \begin{equation*}
\sum_{a_{m+1}=0}^{c_{m+1}-1}\sum_{i_3=3}^{n} p_{n+1-i_{m+2},k-s_m-a_{m+1}}=\sum_{j=s_m}^{s_{m+1}-1} \sum_{i=1}^{n-m-1} p_{i,k-j}.
 \end{equation*}

Every legal decomposition belongs to exactly one of Cases $1, 2, \dots, L$ by Definition \ref{def:goodrecurrence}, hence for $n\ge L$,
 \begin{equation}\label{n}
p_{n,k}\ = \ \sum_{j=1}^{c_1-1}\sum_{i=1}^{n-1} p_{i,k-j}+\sum_{m=1}^{L-1}\sum_{j=s_m}^{s_{m+1}-1}\sum_{i=1}^{n-m-1} p_{i,k-j}\ = \ \sum_{m=0}^{L-1}\sum_{j=s'_m}^{s'_{m+1}-1}\sum_{i=1}^{n-m-1} p_{i,k-j}.
 \end{equation}
Replacing $n$ with $n+1$ yields
  \begin{equation}\label{n1}
p_{n+1,k}\ = \ \sum_{m=0}^{L-1}\sum_{j=s'_m}^{s'_{m+1}-1}\sum_{i=1}^{n-m} p_{i,k-j}.
 \end{equation}
Subtracting (\ref{n}) from (\ref{n1}), we get
   \begin{equation*}
p_{n+1,k}-p_{n,k}\ = \ \sum_{m=0}^{L-1}\sum_{j=s'_m}^{s'_{m+1}-1}p_{n-m,k-j},
 \end{equation*}
which yields the recurrence relation for the $p_{n,k}$'s:
\begin{equation}\label{recurrencep_nk}
p_{n+1,k}\ = \ p_{n,k}+\sum_{m=0}^{L-1}\sum_{j=s'_m}^{s'_{m+1}-1}p_{n-m,k-j}\ = \ \sum_{m=0}^{L-1}\sum_{j=s_m}^{s_{m+1}-1}p_{n-m,k-j}.
 \end{equation}
Multiplying both sides of (\ref{recurrencep_nk}) by $x^k y^{n+1}$ gives
\begin{equation}\label{n3}
p_{n+1,k}x^k y^{n+1}\ = \ \sum_{m=0}^{L-1}\sum_{j=s_m}^{s_{m+1}-1} x^j y^{m+1} p_{n-m,k-j}x^{k-j} y^{n-m}.
 \end{equation}
Summing both sides of (\ref{n3}) for $n\ge L$ and $k\ge M:=s_L$\label{M} $=$ $c_1$ $+\ c_2$ $+\ \cdots$ $+\ c_L$, we get
\begin{equation}\label{n4}
\sum_{
\renewcommand\arraystretch{0.65}
\begin{array}{cc}
\mbox{\tiny
$n > L$}\\
\mbox{\tiny $k \ge M$}
\end{array}
}
p_{n,k}x^k y^n\ = \ \sum_{m=0}^{L-1}\sum_{j=s_m}^{s_{m+1}-1} x^j y^{m+1} \sum_{
\renewcommand\arraystretch{0.65}
\begin{array}{c}
\mbox{\tiny
$n \ge L-m$}\\
\mbox{\tiny $k \ge M-j$}
\end{array}
}
p_{n,k}x^k y^n.
 \end{equation}
Using the definition $\mathscr{G}(x,y)=\sum_{n,k>0} p_{n,k} x^n y^k$, we can write (\ref{n4}) in the following form (where $n$ and $k$ are always positive):
 \begin{equation}\label{n5}
\mathscr{G}(x,y) - \sum_{
\renewcommand\arraystretch{0.65}
\begin{array}{cc}
\mbox{\tiny
$n\le L$}\\
\mbox{\tiny ${\rm{or}}\ k<M$}
\end{array}
}
p_{n,k}x^k y^n
\ = \ \sum_{m\ = \ 0}^{L-1}\sum_{j=s_m}^{s_{m+1}-1} x^j y^{m+1}  \left[\mathscr{G}(x,y) - \sum_{
\renewcommand\arraystretch{0.6}
\begin{array}{cc}
\mbox{\tiny
$n< L-m$}\\
\mbox{\tiny ${\rm{or}}\ k<M-j$}
\end{array}
}
p_{n,k}x^k y^n\right].
 \end{equation}
Rearranging the terms of (\ref{n5}), we get
\begin{align}\label{n6}
\nonumber &\ \mathscr{G}(x,y) \left(1-\sum_{m=0}^{L-1} \sum_{j=s_m}^{s_{m+1}-1} x^j y^{m+1}\right)\\
\nonumber =&\sum_{
\renewcommand\arraystretch{0.6}
\begin{array}{cc}
\mbox{\tiny
$n\le L$}\\
\mbox{\tiny ${\rm{or}}\ k<M$}
\end{array}
}
p_{n,k}x^k y^n
-\sum_{m=0}^{L-1}\sum_{j=s_m}^{s_{m+1}-1} x^j y^{m+1} \sum_{
\renewcommand\arraystretch{0.6}
\begin{array}{cc}
\mbox{\tiny
$n<L-m$}\\
\mbox{\tiny ${\rm{or}}\ k<M-j$}
\end{array}
}
p_{n,k}x^k y^n\\
\nonumber =&\ \sum_{n\le L}p_{n,k}x^k y^n -\sum_{m=0}^{L-1}\sum_{j=s_m}^{s_{m+1}-1} x^j y^{m+1} \sum_{n<L-m}p_{n,k}x^k y^n\\
&+\left[\sum_{
\renewcommand\arraystretch{0.6}
\begin{array}{cc}
\mbox{\tiny
$n> L$}\\
\mbox{\tiny $k<M$}
\end{array}
}
p_{n,k}x^k y^n
-\sum_{m=0}^{L-1}\sum_{j=s_m}^{s_{m+1}-1} x^j y^{m+1} \sum_{
\renewcommand\arraystretch{0.6}
\begin{array}{cc}
\mbox{\tiny
$n\ge L-m$}\\
\mbox{\tiny $k<M-j$}
\end{array}
}
p_{n,k}x^k y^n\right].
 \end{align}

Let $D(L,M)$\label{DLM} be the parenthesized part in (\ref{n6}). Then
\begin{eqnarray*}
\nonumber D(L,M)&=&\sum_{
\renewcommand\arraystretch{0.6}
\begin{array}{cc}
\mbox{\tiny
$n> L$}\\
\mbox{\tiny $k<M$}
\end{array}
}
p_{n,k}x^k y^n
-\sum_{m=0}^{L-1}\sum_{j=s_m}^{s_{m+1}-1} \sum_{
\renewcommand\arraystretch{0.6}
\begin{array}{cc}
\mbox{\tiny
$n>L$}\\
\mbox{\tiny $k<M$}
\end{array}
}
p_{n-m-1,k-j}x^k y^n\\
\nonumber &=&\sum_{
\renewcommand\arraystretch{0.6}
\begin{array}{cc}
\mbox{\tiny
$n> L$}\\
\mbox{\tiny $k<M$}
\end{array}
}
x^k y^n\left(p_{n,k}-\sum_{m=0}^{L-1}\sum_{j=s_m}^{s_{m+1}-1}p_{n-m-1,k-j}\right)\\
&=&0,
 \end{eqnarray*}
where the last equality follows by (\ref{recurrencep_nk}) with $n$ replaced by $n-1$.

As $D(L,M)=0$, we can simplify the right-hand side of (\ref{n6}) to
\begin{equation}\label{B}
\mathscr{B}(x,y)=\sum_{n\le L}p_{n,k}x^k y^n-\sum_{m=0}^{L-1}\sum_{j=s_m}^{s_{m+1}-1} x^j y^{m+1} \sum_{n<L-m}p_{n,k}x^k y^n,
\end{equation}
which completes the proof with (\ref{n6}).
\end{proof}

\begin{remark}
Since $H_n\ge 1$, $p_{n,k}=0$ if $k>n$. Therefore, to find the explicit expression for $\mathscr{B}(x,y)$ of a given sequence ${H_n}$, we only need to find the initial values of the $p_{n,k}$'s, namely those with $0<k\le n\le L$, which is tractable.
\end{remark}


\section{Proof of Theorem \ref{thm:genlekkerkerker} (Generalized Lekkerkerker)}\label{sec:genlek}
Before giving the proof, we sketch the argument and prove some needed preliminary results and notation. Let $A(y)$ and $B(y)$ be the polynomials of (\ref{A}) and (\ref{B}) regarded as polynomials in $y$ with coefficients in $\mathbb{Z}[x]$. Define
\begin{equation}\label{G}
G(y)=\frac{B(y)}{A(y)}.
\end{equation}
Since $B$ is of degree at most $L$ according to Definition (\ref{B}), we can write
\begin{equation}\label{By}
B(y)=\sum_{m=1}^{L}b_m(x) y^m,
\end{equation}
where the $b_i(x)$'s are polynomials of $x$. If $C(x_1,\dots,x_\ell)$ is a polynomial in $\ell$ variables, let $\langle x_i^m\rangle C(x_1,\dots,x_\ell)$\label{angle} denote the coefficient of the $x_i^m$ term when we view $C(x_1,\dots,x_\ell)$ as a polynomial in $x_i$ with coefficients in $\mathbb{Z}[x_1,\dots,x_{m-1},x_{m+1},\dots,x_\ell]$.

Letting $g(x)$ be the coefficient of $y^n$ in $G(y)$, denoted by $\langle y^n\rangle G(y)$, we see that
\begin{equation}\label{g}
g(x)=\sum_{k>0} p_{n,k} x^k.
\end{equation}
For a fixed $n$, taking $x=1$ in (\ref{g}) gives us the sum of the $p_{n,k}$'s, which by definition equals $H_{n+1}-H_n=\Delta_n$, i.e.,
\begin{equation}\label{g1}
g(1)=\sum_{k>0} p_{n,k}=\Delta_n.
\end{equation}
Moreover, taking the derivative of both sides of (\ref{g}) gives
\begin{equation*}
g'(1)=\sum_{k>0} kp_{n,k}= \Delta_n \sum_{k>0} k{\rm{Prob}}(n,k)=\Delta_n \mu_n,
\end{equation*}
therefore
\begin{equation}\label{mu1}
\mu_n=\frac{g'(1)}{g(1)}.
\end{equation}
Thus the proof of Theorem \ref{thm:genlekkerkerker} reduces to finding $g
(1)$ and $g'(1)$.

Recall that $A(y)$ is the polynomial of $y$ with coefficients in $\mathbb{Z}[x]$ defined in (\ref{A}), i.e.,
\begin{equation}\label{Ayy}
A(y)=1-\sum_{m=0}^{L-1} \sum_{j=s_m}^{s_{m+1}-1} x^j y^{m+1}.
\end{equation}

Let $y_1(x),y_2(x),\dots,y_L(x)$\label{yix} be the roots of $A(y)$ (i.e., regarding $A$ as function of $y$). We want to write $\frac{1}{A(y)}$ as a linear combination of the $\frac{1}{y-y_i(x)}$'s, i.e., the partial fraction expansion, as we can use power series expansion to find the coefficient of $y^n$ in $\frac{B(y)}{A(y)}$.

To achieve this goal, we need to show that the $y_i(x)$'s are pairwise distinct, specifically, $A(y)$ has no multiple roots for $x$ in some neighborhood of 1 excluding 1, i.e., $I_{\varepsilon}:=(1-\varepsilon,1+\varepsilon)\backslash \{1\}$.\label{varepsilon} This result is formally stated in Theorem \ref{proppar}(a) and proved in Appendix \ref{amult}; we sketch the argument.

If $x>0$ and $L=1$, then $A(y)=1-\sum_{j=0}^{c_1-1} x^j y$ has a unique root $y_1(x)=\left(\sum_{j=0}^{c_1-1} x^j \right)^{-1}$ and $y_1(x)\in (0,1)$ since $c_1>1$ (see the assumption of Theorem \ref{thm:genZeckendorf}). Note that if $x>0$, then $y_1(x)$ is continuous and $\ell$-times differentiable for all $\ell>0$. Thus in this case, $\varepsilon$ can be 1.

For $L\ge 2$, there is an easy proof for non-increasing $c_i$'s (see Appendix C of \cite{MW}), but the proof for general cases (see Appendix \ref{amult}) is more complicated, involving continuity and the range of the $|y_i(x)|$'s. The main idea is to first show that there exists $x>0$ such that $A(y)$ has no multiple roots and then prove that there are only finitely many $x>0$ such that $A(y)$ has multiple roots.

\vspace{0.1in}

In the proofs in this section, we repeatedly use the continuity of the $y_i(x)$'s, which follows from the fact that the roots of a polynomial with continuous coefficients are continuous (for completeness, see \cite{US} or Appendix A of \cite{MW}. for the formal statement and the proof. Since for any $x>0$ the coefficients of $A(y)$ are continuous functions of $x$ and the leading coefficient is nonzero, the roots of $A(y)$ are continuous at $x$.

The following proposition asserts that $A(y)$ has no multiple roots for $x\in I_{\varepsilon}$ for some $\varepsilon$, and then gives the partial fraction expansion for $1/A(y)$ in terms of the roots. This is a key ingredient in extracting information from the generating function.

\begin{proposition}\label{proppar}
There exists $\varepsilon \in (0,1)$ with the following properties.

{\rm{(a)}} For any $x\in I_{\varepsilon}$, $A(y)$ as polynomial of $y$ has no multiple roots, i.e.,
\begin{equation}\label{multi}
A'(y_i(x))=-\sum_{m=0}^{L-1} \sum_{j=s_m}^{s_{m+1}-1} (m+1)x^j  y_i^m(x)\neq 0,
\end{equation}
where $A'(y)$ is the derivative with respect to $y$.

{\rm{(b)}} If $x=1$, then $A(y)$ has a unique positive real root. Letting it be $y_1(1)$ without loss of generality, then $0<y_1(1)<1$ and $|y_i(1)|>y_1(1)$ for $i>1$ and $|y_i(1)|>y_1(1)$ for $i>1$.

{\rm{(c)}} For any $x\in I_{\varepsilon}$, $A(y)$ has a unique positive real root. Letting it be $y_1(x)$ without loss of generality, then $0<y_1(x)<1$ and $|y_i(x)/y_1(x)|>\sqrt{|y_i(1)/y_1(1)|}>1$ for $i>1$.
If $\varepsilon$ satisfies the above properties, then for any $x\in I_{\varepsilon}$, we have
\begin{equation}\label{par}
\frac{1}{A(y)}= -\frac{1}{\sum_{j=s_{L-1}}^{s_L-1}x^j} \sum_{i=1}^{L}\frac{1}{(y-y_i(x))\prod_{j\neq i}\left(y_j(x)-y_i(x)\right)}.
\end{equation}
\end{proposition}

\begin{proof}
We prove in Appendix \ref{amult} that there exists $\epsilon \in (0,1)$\label{epsilon} such that for any $x\in I_{\epsilon}$, $A(y)$ has no multiple roots.

For (b), when $x=1$, $A(y)$ is strictly decreasing on $(0,\infty)$ and $A(0)=1>0>A(1)$. Thus $A(y)$ has a unique positive root $y_1(1)$ and $y_1(1)\in (0,1)$. Since $A'(y_1(1))<0$, $y_1(1)$ is not a multiple root of $A(y)$.

For any other root $y_i(1)$ $(i>1)$, if $|y_i(x)|\le y_1(x)$, then
\begin{align*}
0& \ = \ |A(y_i(1))|\ = \ \left|1-\sum_{m=0}^{L-1} \sum_{j=s_m}^{s_{m+1}-1}y_i^{m+1}(1)\right|\ \ge\ 1-\sum_{m=0}^{L-1} \sum_{j=s_m}^{s_{m+1}-1}\left|y_i^{m+1}(1)\right|\\
&\ \ge\ 1-\sum_{m=0}^{L-1} \sum_{j=s_m}^{s_{m+1}-1}\left|y_1^{m+1}(1)\right|\ = \ 0.
\end{align*}
Hence the equalities hold. Thus each $y_i^{m+1}(1)$ is nonnegative, i.e., $y_i(1)$ is nonnegative. Since $A(0)\neq 0$, $y_i(1)\neq 0$, thus $y_i(1)>0$; however, $A(y)$ only has one positive root $y_1(1)$ and it is not a multiple root, contradiction.

For (c), denote $\lambda=\min_{i>1}\{\sqrt{|y_i(1)/y_1(1)|}\}>1$. By the continuity of the $y_i(x)$'s, there exists $\varepsilon\in (0,\epsilon)$ such that for all $x\in I_{\varepsilon}$,
\begin{equation*}
y_1(x)<(1+\kappa)y_1(1)\ {\rm{and}}\ y_i(x)>(1-\kappa)y_i(1)\ {\rm{for}}\ 1<i\le L,
\end{equation*}
where $\kappa=(\lambda-1)/2(1+\lambda)\in (0,1).$ Thus
\begin{equation*}
\frac{y_i(x)}{y_1(x)}>\frac{1-\kappa}{1+\kappa}\frac{y_i(1)}{y_1(1)}=\frac{3+\lambda}{1+3\lambda}\frac{y_i(1)}{y_1(1)}>\frac{3+\lambda}{\lambda^2+3\lambda}\frac{y_i(1)}{y_1(1)}=\frac{1}{\lambda}\frac{y_i(1)}{y_1(1)}.
\end{equation*}
Since $\lambda=\min_{i>1}\{\sqrt{|y_i(1)/y_1(1)|}\}\le \sqrt{|y_i(1)/y_1(1)|}$,
\begin{equation*}
\frac{y_i(x)}{y_1(x)}>\frac{1}{\lambda}\frac{y_i(1)}{y_1(1)}\ge \sqrt{\frac{y_i(1)}{y_1(1)}},
\end{equation*}
as desired.

Now suppose $\varepsilon$ satisfies (a), (b) and (c). Since the leading coefficient of $A(y)$ is $-\sum_{j=s_{L-1}}^{s_L-1}x^j$ and the roots of $A(y)$ are $y_1(x),y_2(x),\dots,y_L(x)$,
\begin{equation}\label{Ay}
A(y)=-\sum_{j=s_{L-1}}^{s_L-1}x^j\prod_{i=1}^{L}\left(y-y_i(x)\right).
\end{equation}

For any $x\in I_{\varepsilon}$, the $y_i(x)$'s are distinct, thus we can interpolate the Lagrange polynomial of $\mathscr{L}(y)=1$ at $y_1(x)$, $y_2(x)$, $\dots$, $y_L(x)$:
\begin{equation*}
\sum_{i=1}^{L}\frac{\prod_{j\neq i}\left(y-y_i(x)\right)}{(y-y_i(x))\prod_{j\neq i}\left(y_j(x)-y_i(x)\right)}=1.
\end{equation*}
Dividing both sides by $\prod_{i=1}^{L}\left(y-y_i(x)\right)$ and combining with (\ref{Ay}) yields (\ref{par}).
\end{proof}

\begin{proposition}\label{diff}
For any $x>0$, if $y_i(x)$ is not a multiple root of $A(y)$, then $y_i(x)$ is $\ell$-times differentiable for any $\ell \ge 1$. In particular, given $\varepsilon$ as in Proposition \ref{proppar}, for any $x\in I_{\varepsilon}$ and each $1\le i\le L$, we have $y_i(x)$ is $\ell$-times differentiable for any $\ell \ge 1$. Additionally, note that $y_1(x)$ is not a multiple root of $A(y)$ when $x=1$ since $A'(y_1(1))<0$, thus $y_1(x)$ is $\ell$-times differentiable at 1 for any $\ell \ge 1$. If $y_i(x)$ is differentiable at $x$, then its derivative is
\begin{equation}\label{yi'}
y'_i(x)=-\frac{\sum_{m=0}^{L-1} \sum_{j=s'_m}^{s'_{m+1}-1}jy_i^{m+1}(x)x^{j-1}}{\sum_{m=0}^{L-1} \sum_{j=s_m}^{s_{m+1}-1} (m+1)x^j  y_i^m(x)}.
\end{equation}
\end{proposition}

\begin{proof}[Sketch of the proof] We prove the differentiability by induction on $\ell$. For the derivative, we differentiate $A(y)$ at $y_i(x)$ to get (\ref{yi'}). See Appendix \ref{adiff} for the details.
\end{proof}

Let us return to finding $g$ (with $L\ge 1$). From now on, we assume that $x\in I_{\varepsilon}$. Plugging (\ref{By}) and (\ref{par}) into (\ref{G}), we get
\begin{eqnarray}
\nonumber \sum_{j=s_{L-1}}^{s_L-1}x^j G(y)
&= &-\sum_{m=1}^{L} b_m(x) y^m \sum_{i=1}^{L}\frac{1}{(y-y_i(x))\prod_{j\neq i}\left(y_j(x)-y_i(x)\right)}\\
\nonumber &=& {\sum_{m=1}^{L}b_m(x) y^m}\sum_{i=1}^{L}\frac{1}{(1-\frac{y}{y_i(x)})y_i(x)\prod_{j\neq i}\left(y_j(x)-y_i(x)\right)}\\
\nonumber &=& {\sum_{m=1}^{L}b_m(x) y^m}\sum_{i=1}^{L}\frac{1}{y_i(x)\prod_{j\neq i}\left(y_j(x)-y_i(x)\right)}\sum_{l\ge 0}\left(\frac{y}{y_i(x)}\right)^l.
\end{eqnarray}
Thus for $n\ge L$, by looking at the coefficient of $y^n$ (which we are denoting $g(x)$), we obtain
\begin{equation*}
g(x)=\frac{1}{\sum_{j=s_{L-1}}^{s_L-1}x^j}\sum_{i=1}^{L}\frac{1}{y_i(x)\prod_{j\neq i}\left(y_j(x)-y_i(x)\right)} \sum_{m=1}^{L}\frac{b_m(x)}{y^{n-m}_i(x)}.
\end{equation*}

Define
\begin{equation}\label{qi}
q_i(x)=\frac{\sum_{m=1}^{L} b_m(x) y^{m}_i(x)}{\sum_{j=s_{L-1}+1}^{s_L}x^jy_i(x)\prod_{j\neq i}\left(y_j(x)-y_i(x)\right)} ,
\end{equation}
then
\begin{equation}\label{gg}
g(x)=\sum_{i=1}^{L} xq_i(x) y_i^{-n}(x).
\end{equation}
Note that the $q_i(x)$'s are independent of $n$.

Define
\begin{equation}\label{calA}
\mathcal{A}(y)=y^L A\left(\frac{1}{y}\right)=y^L-\sum_{m=0}^{L-1} \sum_{j=s_m}^{s_{m+1}-1} x^j y^{L-1-m}.
\end{equation}
Since $A(0)\neq 0$, the roots of $\mathcal{A}(y)$ are $\alpha_i(x):=\left(y_i(x)\right)^{-1}$\label{alphaix}. Therefore, by Proposition \ref{proppar}, $\alpha_1(x)$ is real, and
\begin{equation}\label{alpha}
\alpha_1(x)>1,\ {\rm{and}}\ |\alpha_i(x)/\alpha_1(x)|<\sqrt{|\alpha_i(1)/\alpha_1(1)|}<1 \ {\rm{for}}\ i>1.
\end{equation}
Plugging $\alpha_i(x)=\left(y_i(x)\right)^{-1}$ into (\ref{gg}), we get
\begin{equation}\label{gsum}
g(x)=\sum_{i=1}^{L} xq_i(x) \alpha_i^n(x).
\end{equation}
Since $g(x)$ is a polynomial of $x$, we have
\begin{equation}\label{go}
g^{(\ell)}(1)= \lim_{x\rightarrow 1} g^{(\ell)}(x)=\lim_{x\rightarrow 1}\left[\sum_{i=1}^{L} xq_i(x) \alpha_i^n(x)\right]^{(\ell)},\ \forall\ \ell\ge 0.
\end{equation}

We want the main term of $g^{(\ell)}(x)$ to be $\left[xq_1(x) \alpha_1^n(x)\right]^{(\ell)}$ for $x\in (x-\varepsilon,x+\varepsilon$. Since $g(x)$ is $\ell$-times differentiable at 1, by (\ref{go}) it suffices to prove the following two claims.

\begin{claim}\label{cla1}
For any $\ell \ge 1$ and any $i\in \{1,2,\dots, L\}$, we have $\alpha_i(x)$ and $q_i(x)$ are $\ell$-times differentiable at $x\in I_{\varepsilon}$ and $\alpha_1(x)$ and $q_1(x)$ are $\ell$-times differentiable at 1.
\end{claim}

\begin{claim}\label{claimo1}
For any $x\in I_{\varepsilon}$ and $\ell\ge 0$, we have
\begin{equation}\label{der}
\frac{d^{\ell}}{dx^{\ell}}\sum_{i=2}^{L}xq_i(x) \alpha_i^n(x)=o(\gamma^n_{\ell})\alpha_1^n(x),
\end{equation}
for some $\gamma_{\ell}\in (0,1)$.

We use this result for fixed $\ell$ as $n$ goes to infinity. With the result and (\ref{go}), we see that
\begin{equation}\label{der}
g^{(\ell)}(1)=\left[q_1(1) \alpha_1^n(1)\right]^{(\ell)}+o(\gamma^n_{\ell})\alpha_1^n(1),
\end{equation}
\end{claim}

We now prove the second claim; see Appendix \ref{acla1} for a proof of the first claim.

\begin{proof} There is an easy proof if $A(y)$ has no multiple roots when $x=1$. In this case, all $y_i(x)$'s, $\alpha_i(x)$'s and $q_i(x)$'s are $\ell$-times differentiable for all $\ell$ at $x=1$. Therefore Claim \ref{claimo1} follows immediately by Proposition \ref{proppar} and Proposition \ref{proppar} follows directly from the continuity of the $y_i(x)$'s.

Though the situation becomes completely different and harder if $A(y)$ has multiple roots when $x=1$, the claim is still true. See Appendix \ref{amain} for the proof.
\end{proof}

\begin{proof}[Proof of Theorem \ref{thm:genlekkerkerker}]
We combine our results above to complete the proof of the Generalized Lekkerkerker Theorem. Recall from (\ref{g1}) that $g(1)=\Delta_n=H_{n+1}-H_n$, thus by Claim \ref{claimo1} with $\ell=0$, we get
\begin{equation}\label{Deltan}
\Delta_n=g(1)=(q_1(1)+o(\gamma^n_0)) \alpha_1^n(1).
\end{equation}
Since $\Delta_n$ is positive and unbounded, we have $q_1(1)>0$.

We can also see that (\ref{Deltan}) is true for some positive constant $q_1(1)$ by looking at the formula for general $H_n$. Since the characteristic roots of the recurrence relation of $H_n$ are the $\alpha_i(1)$'s, each $H_n$ is of the form $\sum_i h_i(n)\alpha^n_i(1)$ where the $h_i(n)$'s are polynomials of $n$ with degree less than the multiplicity of $\alpha_i(1)$ and hence less than $L$. Thus it follows from (\ref{alpha}) that $\sum_i h_i(n)\alpha^n_i(1)$ is of the form $(q+o(\gamma'^n_0)) \alpha_1^n(1)$ for some constant $q$ and $\gamma'_0$.

\vspace{0.1in}

Define $g_i(x)=xq_i(x) \alpha_i^n(x)$\label{gix}. According to (\ref{gsum}) we have
$g(x)=\sum_{i=1}^{L} g_i(x)$. Applying Claim \ref{claimo1} with $\ell=1$ yields
\begin{eqnarray*}
\nonumber g'(x)=g'_1(x)+o(\gamma^n_1)\alpha_1^n(x)=nxq_1(x) \alpha'_1(x)\alpha_1^{n-1}(x)+(xq_1(x))' \alpha_1^n(x)+o(\gamma^n_1)\alpha_1^n(x).
\end{eqnarray*}
Letting $x\rightarrow 1$ and using (\ref{Deltan}), we obtain
\begin{eqnarray*}
\nonumber \frac{g'(1)}{g(1)}&=&\frac{nq_1(1) \alpha'_1(1)\alpha_1^{n-1}(1)+(q_1(1)+q'_1(1)) \alpha_1^n(1)+o(\gamma^n_1)\alpha_1^n(1)}{q_1(1)\alpha^n_1(1)+o(\gamma^n_0)\alpha_1^n(1)}\\
\nonumber &=&\frac{nq_1(1) \alpha'_1(1)(\alpha_1(1))^{-1}+(q_1(1)+q'_1(1))+o(\gamma^n_1)}{q_1(1)+o(\gamma^n_0)}\\
&=&\frac{\alpha'_1(1)}{\alpha_1(1)}n+\frac{q_1(1)+q'_1(1)}{q_1(1)}+o(\gamma^n_1).
\end{eqnarray*}
Therefore, by (\ref{mu1}) $\mu_n$ is of the form (\ref{eq:genlekkerkerker}): $\mu_n=Cn+d+o(\gamma^n_1)$, with
\begin{equation}\label{Cd}
C=\frac{\alpha'_1(1)}{\alpha_1(1)}\ {\rm{and}}\ d=1+\frac{q'_1(1)}{q_1(1)}.
\end{equation}
which completes the proof of the Generalized Lekkerkerker Theorem.
\end{proof}

\begin{remark}
We provide some information about the value of the constant $C$.

{\rm{(a)}} A formula for $C$:

Note that $C$ can be computed as follows:
\begin{equation}\label{formulaC}
C=\left.\frac{\alpha'_1(x)}{\alpha_1(x)}\right|_{x=1}=\left.\frac{\left((y_1(x))^{-1}\right)'}{(y_1(x))^{-1}}\right|_{x=1}=-\left.\frac{y'_1(x)}{y_1(x)}\right|_{x=1}=-\frac{y'_1(1)}{y_1(1)},
\end{equation}
where $y'_1(1)$ is given by (\ref{yi'}). We find
\begin{eqnarray}
\label{formC} C&=&-\frac{y'_1(1)}{y_1(1)}=\frac{\sum_{m=0}^{L-1}\sum_{j=s_m}^{s_{m+1}-1}jy^m_1(1)}{\sum_{m=0}^{L-1}\sum_{j=s_m}^{s_{m+1}-1}(m+1)y^m_1(1)}\\
\label{eqC} &=&\frac{\sum_{m=0}^{L-1}\frac{1}{2}(s_m+s_{m+1}-1)(s_{m+1}-s_m)y^m_1(1)}{\sum_{m=0}^{L-1}(m+1)(s_{m+1}-s_m)y^m_1(1)}.
\end{eqnarray}

{\rm{(b)}} Upper and lower bounds for $C$.

Applying (\ref{eqC}) with some approximations, we get
\begin{equation*}
 \min \left\{\frac{c_1-1}{2},\  \frac{c_1-2}{L}+1\right\}\ \le\ C \ \le\ \frac{(2L-1)c_1-1}{2L}<c_1
 \end{equation*}
(see Appendix \ref{aCbd} for the detailed proof).
\end{remark}


\section{Gaussian Behavior}\label{sec:GaussianBehavior}
In this section, we prove Theorem \ref{thm:Gaussian}, namely the distribution of $K_n$ converges to a Gaussian. Let $\sigma_n$\label{sigma_n} be the standard deviation of $K_n$. First we centralize and normalize $K_n$ to $K^{(c)}_n=(K_n-\mu_n)/\sigma_n$. Thus it suffices to show that $K^{(c)}_n$ converges to the standard normal. According to Markov's Method of Moments, we only need to show that each moment of $K^{(c)}_n$ tends to that of the standard normal distribution, which is equivalent to the following.

\begin{theorem}\label{Gauss}
Let $\mu_n(m)$ be the $m$\textsuperscript{${\rm th}$} moment of $K_n-\mu_n$\label{munm}, then for any integer $u\ge 1$, we have
\begin{equation}\label{gauss}
\frac{\mu_n(2u-1)}{\sigma_n^{2u-1}}\rightarrow 0\ {\rm{and}}\ \frac{\mu_n(2u)}{\sigma_n^{2u}}\rightarrow (2u-1)!!,\ {\rm{as}}\ u\rightarrow \infty.
\end{equation}
\end{theorem}

The proof for the case of Fibonacci numbers is significantly easier as we have a tractable, explicit formula for the number of integers with exactly $k$ summands: $p_{n,k}=\ncr{n-1-k}{k}$. The Gaussian behavior follows by using Stirling's formula to analyze the limiting behavior of $p_{n,k}$; see \cite{KKMW} for the details. Unfortunately, this argument does not work in general as the resulting expressions for $p_{n,k}$ are not as amenable to analysis, and we must resort to analyzing the generating function expansion.

In the proof for the general case, we first point out that it suffices to prove the same result for $K_n-(Cn+d)$ with $C$ and $d$ defined in (\ref{Cd}). Then we show that the $m$\textsuperscript{${\rm th}$} moment $\tilde \mu_n(m)$\label{tildemunm} of $K_n-(Cn+d)$ equals $\tilde g_{m}(1)/\Delta_n$ for polynomials $\tilde g_{m}(x)$  with
\begin{equation}\label{gmu}
\tilde g_{0}(x)=\sum_k p_{n,k}x^{k-\tilde \mu_n-1}= \frac{g(x)}{x^{\tilde \mu_n+1}},\ \ \
\tilde g_{j+1}(x)=(x \tilde g_{j}(x))',\ j\ge 1.
\end{equation}

By Definition \ref{g} and (\ref{gmu}), we prove by induction that the main term of $\tilde g_{m}(1)$ is of the form $\alpha_1^n(x)x^{-\tilde \mu_n}\sum _{i=0}^m f_{i,m}(x)n^i$\label{fimx} for some functions $f_{i,m}(x)$'s and thus conclude that $\tilde \mu_n(m)=\frac{1}{q_1(1)}\sum _{i=0}^m f_{i,m}(1)n^i+o(\tau_m^n)$ for some $\tau_m \in(0,1)$\label{taum}. Finally, we evaluate the $f_{i,m}(1)$'s to obtain (\ref{gauss}).

\vspace{0.1in}

We now give the proof. In the course of our analysis we will interrupt the proof to state and prove some simple, needed propositions. Noting that $\mu_n=\tilde \mu_n+o(\gamma^n_1)$, by some simple approximations (see Appendix \ref{amun}), we see that
\begin{equation}\label{estimatemunm}
\mu_n(m)=\tilde \mu_n(m)+o(\beta_m^n)
\end{equation}
 or some $\beta_m \in(0,1)$\label{betam}. In the special case of $m=2$, we have $\sigma_n^2=\mu_n(2)=\tilde \mu_n(2)+o(\tau_m^n)$, therefore (\ref{gauss}) is equivalent to
\begin{equation}\label{gauss2}
\frac{\tilde \mu_n(2u-1)}{\tilde \mu^{u-\frac{1}{2}}_n(2)}\rightarrow 0\ {\rm{and}}\ \frac{\tilde \mu_n(2u)}{\tilde \mu^u_n(2)}\rightarrow (2u-1)!!,\ {\rm{as}}\ u\rightarrow \infty.
\end{equation}

We calculate the moments $\tilde \mu_n(m)$'s by applying the method of differentiating identities to $g$. Setting $x=1$ in (\ref{gmu}), we get
\begin{equation*}
\tilde g_{0}(1)=\sum_{k} p_{n,k}=\Delta_n=\tilde \mu_n(0)\Delta_n.
\end{equation*}
When $m=1$, by Definition (\ref{gmu}) we get
\begin{equation}\label{gmu1}
\tilde g_{1}(x)=(x \tilde g_{0}(x))'=\left(\sum_{k} p_{n,k}x^{k-\tilde \mu_n}\right)'=\sum_{k} p_{n,k}(k-\tilde \mu_n)x^{k-\tilde \mu_n-1}.
\end{equation}
When $m=2$, by (\ref{gmu}) and (\ref{gmu1}), we get
\begin{equation*}
\tilde g_{2}(x)=(x \tilde g_{1}(x))'=\sum_{k} p_{n,k}(k-\tilde \mu_n)^2 x^{k-\tilde \mu_n-1}.
\end{equation*}
Setting $x=1$, we get
\begin{equation*}
 \tilde g_{2}(1)=\sum_{k} p_{n,k}(k-\tilde \mu_n)^2=\tilde \mu_n(2)\Delta_n.
\end{equation*}
By induction on $m$, we can prove the following.

\begin{proposition}\label{propmom}
For any $m\ge 0$, we have
\begin{equation}\label{gmum}
 \tilde g_{m}(x)=\sum_{k} p_{n,k}(k-\tilde \mu_n)^m x^{k-\tilde \mu_n-1}\ {\rm{and}}\ \tilde g_{m}(1)=\tilde \mu_n(m)\Delta_n.
\end{equation}
\end{proposition}

\begin{proof}
We have proved the statement for $m=0,1,2$. If (\ref{gmum}) holds for $m$, then the recurrence relation (\ref{gmu}) gives
\begin{eqnarray*}
 \tilde g_{m+1}(x) =
 (x \tilde g_{m}(x))'=\left(\sum_{k} p_{n,k}(k-\tilde \mu_n)^m x^{k-\tilde \mu_n}\right)'
=
\sum_{k} p_{n,k}(k-\tilde \mu_n)^{m+1} x^{k-\tilde \mu_n-1}.
\end{eqnarray*}
Setting $x=1$ gives $\tilde g_{m+1}(1)=\tilde \mu_n(m+1)\Delta_n$.
Thus the statement holds for $m+1$ and hence for any $m\ge 0$.
\end{proof}

Returning to the proof of Theorem \ref{Gauss}, denote
\begin{equation}\label{gji}
\tilde g_{0,i}(x)=\frac{q_i(x) \alpha_i^n(x)}{x^{\tilde \mu_n}},
\ {\rm{and}}\ \tilde g_{j+1,i}(x)=(x \tilde g_{j,i}(x))'
\end{equation}
for $x\in I_{\varepsilon}$ if $1<i\le L$ and for $x\in I_{\varepsilon}\cup \{1\}$ if $i=1$. By Definition (\ref{gji}) and using the same approach as in Lemma \ref{claimo1}, we can prove that
\begin{equation}
\forall x\in I_{\varepsilon}: \ \sum_{i=2}^{L} \tilde{g}_{j,i}(x)\ =  \ o(\tau^n_j)\alpha^n_1(x),
\end{equation}
for some $\tau_j\in (0,1)$. Thus referring to (\ref{gmu}), we have
\begin{equation}\label{sumtg}
\forall x\in I_{\varepsilon}: \ \tilde g_{j}(x)\ =\ \sum_{i=1}^{L} \tilde{g}_{j,i}(x)\ = \ \tilde{g}_{j,1}(x)+o(\tau^n_j)\alpha^n_1(x).
\end{equation}
Taking the limit as $x$ approaches 1 yields
\begin{equation}\label{gj1}
\tilde{g}_{j}(1) = \tilde{g}_{j,1}(1)+o(\tau^n_j)\alpha^n_1(1),\ \forall \ x\in I_{\varepsilon}.
\end{equation}
Denoting $ \tilde g_{j,1}(x)$ by $F_j(x)$\label{Fjx}, then
\begin{equation}\label{tg0}
F_0(x)=q_1(x) \alpha_1^n(x) x^{-\tilde \mu_n}
\ {\rm{and}}\
F_{j+1}(x)=(xF_{j}(x))'.
\end{equation}
Note that $q_1(x)$ and $\alpha_1(x)$ are $\ell$-times differentiable for any $\ell
\ge 1$(see Claim \ref{cla1}). Thus when $j=0$, we get
\begin{eqnarray}\label{tg1}
\nonumber
F_{1}(x)&=&(xF_{0}(x))'=\left(q_1(x) \alpha_1^n(x) x^{-\tilde \mu_n}\right)'\\
\nonumber &=&nxq_1(x)\alpha'_1(x) \alpha_1^{n-1}(x) x^{-\tilde \mu_n}-(\tilde \mu_n-1) q_1(x) \alpha_1^n(x) x^{-\tilde \mu_n}
+xq'_1(x) \alpha_1^n(x) x^{-\tilde \mu_n}\\
\nonumber &=&nxq_1(x)\alpha'_1(x) \alpha_1^{n-1}(x) x^{-\tilde \mu_n}-(Cn+d-1) q_1(x) \alpha_1^n(x) x^{-\tilde \mu_n}
+xq'_1(x) \alpha_1^n(x) x^{-\tilde \mu_n}\\
\nonumber &=&\alpha_1^n(x)x^{-\tilde \mu_n}\left[\left(\frac{x\alpha'_1(x)}{\alpha_1(x)}-C\right)q_1(x)n+(1-d)q_1(x)+xq'_1(x)\right]\\
&=&\alpha_1^n(x)x^{-\tilde \mu_n}\left[h(x)q_1(x)n+d'q_1(x)+xq'_1(x)\right],
\end{eqnarray}
where $h(x)$ and $d'$ are defined as
\begin{equation}\label{hd'}
h(x)=\frac{x\alpha'_1(x)}{\alpha_1(x)}-C\ {\rm{and}}\ d'=1-d=-\frac{q'_1(1)}{q_1(1)}
\end{equation}
(see (\ref{Cd}) for the definition of $d$). By (\ref{Cd}), we have
\begin{equation}\label{h0}
h(1)=0.
\end{equation}
Moreover, since $\alpha_1(x)$ is $\ell$-times differentiable at 1 and $\alpha_1(1)\neq 0$ (see Proposition \ref{diff}), we have
\begin{equation}\label{hdiff}
h(x)\ {\rm{is}}\ \ell{\rm{-times\ differentiable\ at\ 1\ for\ any}}\ \ell\ge 1.
\end{equation}

From (\ref{tg0}) and (\ref{tg1}), we observe that $F_{m}(x)$ can be written as a product of $\alpha_1^n(x)x^{-\tilde \mu_n}$ and a sum of other functions of $n$ and $x$. In fact, we have the following.

\begin{proposition}\label{propfim}
For any $m\ge 0$,

{\rm{(a)}} We have $F_{m}(x)$ is of the form
\begin{equation}\label{Fm}
F_{m}(x)=\alpha_1^n(x)x^{-\tilde \mu_n}\sum _{i=0}^m f_{i,m}(x)n^i,
\end{equation}
where the $f_{i,m}$'s are functions of $x$ and $\alpha_1(x)$ but independent of $n$.

{\rm{(b)}} The $f_{i,m}$'s are $\ell$-times differentiable at $x\in I_{\varepsilon}$ for any $\ell \ge 1$.

{\rm{(c)}} Define
\begin{equation}\label{fim0}
f_{i,m}(x)=0\ {\rm{if}} \ i>m \ {\rm{or}}\ i<0 \ {\rm{or}}\ m<0,
\end{equation}
then for $m>0$, we have the following recurrence relation:
\begin{equation}\label{frec}
f_{i,m}(x)=h(x)f_{i-1,m-1}(x)+d'f_{i,m-1}(x)+xf'_{i,m-1}(x).
\end{equation}
\end{proposition}

\begin{proof}
We proceed by induction on $m$. For $m=0$ and 1, (a) holds because of (\ref{tg0}) and (\ref{tg1}). Further,  (\ref{tg0}) and (\ref{tg1}) give the expressions of $f_{0,0}$, $f_{0,1}$ and $f_{1,1}$:
\begin{equation}\label{f00}
f_{0,0}(x)=q_1(x),f_{0,1}(x)=d'q_1(x)+xq'_1(x),f_{1,1}(x)=h(x)q_1(x).
\end{equation}
By Claim \ref{cla1} and (\ref{hdiff}), they are differentiable $\ell$-times at $x\in I_{\varepsilon}$ for any $\ell \ge 1$. Hence (b) holds for $m=0$ and 1. Finally, with (\ref{f00}), it is easy to verify that (c) holds for $m=0$ and 1.

If the statement holds for $m$, by (\ref{gmu}) we have
\begin{equation*}
\nonumber F_{m+1}(x)=\left[\alpha_1^n(x)x^{-\tilde \mu_n}\sum _{i=0}^m xf_{i,m}(x)n^i\right]'=\sum _{i=0}^m \left[\alpha_1^n(x)x^{-\tilde \mu_n}xf_{i,m}(x)n^i\right]'.
\end{equation*}
For convenience, we denote $h_i(x)=\alpha_1^n(x)x^{-\tilde \mu_n}xf_{i,m}(x)n^i$\label{hix} for $0\le i\le m$. Thus
\begin{equation}\label{gm+1sum}
F_{m+1}(x)=\sum _{i=0}^m h'_{i}(x).
\end{equation}
For each $0\le i\le m$, we have
\begin{eqnarray}\label{h'i}
\nonumber
h'_{i}(x)
&=&n^i\left[\alpha'_1(x)\alpha_1^{n-1}(x)x^{-\tilde \mu_n}xf_{i,m}-(\tilde \mu_n-1)\alpha_1^n(x)x^{-\tilde \mu_n}f_{i,m}(x)
+\alpha_1^n(x)x^{-\tilde \mu_n}xf'_{i,m}(x)\right]\\
\nonumber &=&n^i\alpha_1^n(x)x^{-\tilde \mu_n}\left[nf_{i,m}(x)\left(\alpha'_1(x)\alpha_1^{-1}(x)x-C\right)+(1-d)f_{i,m}(x)
+xf'_{i,m}(x)\right]\\
\nonumber &=&n^i\alpha_1^n(x)x^{-\tilde \mu_n}\left[nh(x)f_{i,m}(x)+d'f_{i,m}(x)+xf'_{i,m}(x)\right]\\
&=&\alpha_1^n(x)x^{-\tilde \mu_n}\left[n^{i+1}h(x)f_{i,m}(x)+n^i\left(d'f_{i,m}(x)+xf'_{i,m}(x)\right)\right]
\end{eqnarray}
(see (\ref{hd'}) for the definitions of $h(x)$ and $d'$). Plugging (\ref{h'i}) into (\ref{gm+1sum}) yields
\begin{eqnarray}\label{gm+1}
\nonumber F_{m+1}(x)&=&\alpha_1^n(x)x^{-\tilde \mu_n}\Big[n^{m+1}h(x)f_{m,m}(x)+ \sum _{i=1}^m n^{i}\left(h(x)f_{i-1,m}(x)+d'f_{i,m}(x)\right.\\
& &\left.\ + \ xf'_{i,m}(x)\right)+d'f_{0,m}(x)+xf'_{0,m}(x)\Big].
\end{eqnarray}
Hence (\ref{Fm}) holds for $m+1$ as desired.

For (b) and (c), from (\ref{gm+1}) we get
\begin{equation}\label{fmm}
 f_{m+1,m+1}(x)=h(x)f_{m,m}(x),
\end{equation}
\begin{equation}\label{fim}
f_{i,m+1}(x)=h(x)f_{i-1,m}(x)+d'f_{i,m}(x)+xf'_{i,m}(x),\ 1\le i\le m
\end{equation}
and
\begin{equation}\label{f0m}
f_{0,m+1}(x)=d'f_{0,m}(x)+xf'_{0,m}(x).
\end{equation}

By Definition (\ref{fim0}), we can combine (\ref{fmm}), (\ref{fim}) and (\ref{f0m}) into one recurrence relation (\ref{frec}) (with $m$ replaced by $m+1$). With this recurrence relation, (\ref{hdiff}) and the induction hypothesis of (b) for $m$, we see that (b) also holds for $m+1$. This completes the proof.
\end{proof}

\begin{proposition}\label{propmunm}
We have
\begin{equation}
\tilde \mu_n(m)=\frac{1}{q_1(1)}\sum _{i=0}^m f_{i,m}(1)n^i+o(\tau^n_m)\ {\rm{for\ some}}\ \tau_m\in (0,1).
\end{equation}
\end{proposition}

\begin{proof}
From (\ref{gmum}), (\ref{sumtg}), (\ref{Deltan}), the definition $F_m(x)=\tilde g_{m,1}(x)$ and Proposition \ref{propfim}, we obtain
\begin{eqnarray*}
\nonumber \tilde \mu_n(m)&=&\frac{\tilde g_m(1)}{\Delta_n}=\frac{\tilde g_{m,1}(1)+o(\tau^n_m)\alpha^n_1(1)}{\Delta_n}=\frac{\tilde F_{m}(1)+o(\tau^n_m)\alpha^n_1(1)}{\Delta_n}\\
\nonumber &=&\frac{[\sum _{i=0}^m f_{i,m}(1)n^i+o(\tau^n_m)]\alpha^n_1(1)}{\left[q_1(1)+o(\gamma^n_0)\right]\alpha^n_1(1)} =\frac{1}{q_1(1)}\sum _{i=0}^m f_{i,m}(1)n^i+o(\tau^n_m).
\end{eqnarray*}
\end{proof}

From Proposition \ref{propmunm}, we see that the main term of $\tilde \mu_n(m)$ only depends on $q_1(1)$ and the $f_{i,m}(1)$'s. Note that to prove (\ref{gauss2}), it suffices to find the main term of $\tilde \mu_n(m)$.  Thus the problem reduces to finding the $f_{i,m}(1)$'s. We first calculate the variance, namely $\tilde \mu_n(2)$.
\begin{proposition}\label{propmu2}
The variance of $K_n-\tilde \mu_n$
\begin{equation}\label{mun2}
\tilde \mu_n(2)=h'(1)n+q''_1(1)+o(\tau^n_2)
\end{equation}
with $h'(1)\neq 0$, $q''_1(1)$ and $\tau_2\in (0,1)$ constant depending on only $L$ and the $c_i$'s.
\end{proposition}
With the estimation (\ref{estimatemunm}), it follows immediately that the variance of $K_n$ is of order $n$.
\begin{theorem}\label{thm:varKn}
The variance of $K_n$
\begin{equation}
\mu_n(2)=h'(1)n+q''_1(1)+o(\tau'^n_2)
\end{equation}
with $h'(1)\neq 0$, $q''_1(1)$ and $\tau_2\in (0,1)$ constant depending on only $L$ and the $c_i$'s.
\end{theorem}

\begin{proof}[Proof of Proposition \ref{propmu2}]
If $m=2$, by (\ref{fmm}) and (\ref{h0}) we get
$f_{2,2}(1)=h(1)f_{1,1}(1)=0$.
Applying (\ref{frec}) to $(i,m)=(1,2)$ and plugging in (\ref{f00}) yields
\begin{eqnarray*}
\nonumber f_{1,2}(x)&=&h(x)f_{0,1}(x)+d'f_{1,1}(x)+xf'_{1,1}(x)\\
&=&h(x)f_{0,1}(x)+d'h(x)q_1(x)+xh(x)q'_1(x)+xh'(x)q_1(x).
\end{eqnarray*}
Setting $x=1$ and using $h(1)=0$ (see (\ref{h0})) yields
\begin{equation*}
f_{1,2}(1)=h(1)f_{0,1}(1)+d'h(1)q_1(1)+h(1)q'_1(1)+h'(1)q_1(1)
=h'(1)q_1(1).
\end{equation*}

Using (\ref{f0m}) and (\ref{frec}), we can find $f_{0,2}(x)$ as follows.
\begin{eqnarray*}
f_{0,2}(x)&=&d'f_{0,1}(x)+xf'_{0,1}(x)
=
d'^2q_1(x)+d'xq'_1(x)+d'xq_1(x)+xq'_1(x)+x^2q''_1(x).
\end{eqnarray*}
Setting $x=1$ and substituting $d'$ by $-\frac{q'_1(1)}{q_1(1)}$ (see (\ref{hd'})) yields
\begin{equation*}
f_{0,2}(1)=q''_1(1).
\end{equation*}

Combining the above results with Proposition \ref{propmunm} gives (\ref{mun2}). Thus it remains to show that $h'(1)\neq 0$. We can derive a formula of $h'(x)$ in terms of $y_1(x)$ by Definition (\ref{hd'}), (\ref{formulaC}) and (\ref{yi'}), and then prove that $h'(1)\neq 0$ by contradiction (see Appendix \ref{ah1}).
\end{proof}

From Propositions \ref{propmunm} and \ref{propmu2}, we see that (\ref{gauss2}) (which is what we need to show to finish the proof of Theorem \ref{Gauss}) is equivalent to
\begin{equation}\label{2u-10}
f_{i,2u-1}(1)=0,\ i\ge u,
\end{equation}
\begin{equation}\label{2u0}
f_{i,2u}(1)=0,\ i> u,
\end{equation}
and
\begin{equation}\label{u2u}
f_{u,2u}(1)=(2u-1)!!q_1(1)\left(h'(1)\right)^u.
\end{equation}

For convenience, we denote
\begin{equation*}
t^{(\ell)}_{i,m}=f^{(\ell)}_{i,m}(1),\ \ell\ge 0.
\end{equation*}
\label{tim}Note that if $\ell=0$, then the definition is just $t_{i,m}=f_{i,m}(1)$.

\begin{proposition}\label{proptl}
For any $0\le m<2i$ and $\ell\ge 0$, we have
\begin{equation}\label{tl0}
t^{(\ell)}_{i,m-\ell}=f^{(\ell)}_{i,m-\ell}(1)=0.
\end{equation}
\end{proposition}

\begin{proof}
If $\ell>m$ or $i>m-\ell$, according to Definition (\ref{fim0}), we have $f_{i,m-\ell}(x)=0$. Thus $f^{(\ell)}_{i,m-\ell}(x)=0$ and (\ref{tl0}) follows. Therefore, it suffices to prove for $0\le \ell \le m<2i$ and $i\le m-\ell$, i.e., \begin{equation}\label{con}
0\le \ell \le m-i<i.
\end{equation}

We proceed by induction on $m$. If $m=0$, then there is no $i$ that satisfies (\ref{con}). Thus the statement holds. If $m=1$, the only choice for $i$ and $\ell$ that satisfies (\ref{con}) is $i=1$ and $\ell =0$. By (\ref{f00}) and (\ref{h0}), we get $t^{(\ell)}_{i,m-\ell}=t_{1,1}=f_{1,1}(1)=h(1)q_1(1)=0$. Thus the statement holds for $m=1$. Assume that the statement holds for any $m'<m$ $(m\ge 2)$. For any $(i,m,\ell)$ that satisfies (\ref{con}) and $1\le j\le \ell$, we have \begin{equation*}
2(i-1)=2i-2>m-2\ge m-1-j,
\end{equation*}
thus we can apply the induction hypothesis (\ref{tl0}) to $(i-1,m-1-j,\ell-j)$, $(i,m-1,\ell)$ and $(i,m-1-\ell+j,j)$ with $1\le j\le \ell$ and obtain
\begin{equation}\label{eq11}
f_{i-1,m-1-\ell}^{(\ell-j)}(1)=f^{(\ell)}_{i,m-1-\ell}(1)=f^{(j)}_{i,m-1-\ell}(1)=0.
\end{equation}

Taking the $\ell^{\textsuperscript{th}}$ derivative of both sides of (\ref{frec}), we get
\begin{eqnarray*}
\nonumber f^{(\ell)}_{i,m-\ell}(x)&=&h(x)f_{i-1,m-1-\ell}^{(\ell)}(x)+\sum_{j=1}^{\ell} {\ell \choose j} h^{(j)}(x)f_{i-1,m-1-\ell}^{(\ell-j)}(x)\\
&  &+d'f^{(\ell)}_{i,m-1-\ell}(x)+xf^{(\ell+1)}_{i,m-1-\ell}(x)+\sum_{j=1}^{\ell}f^{(j)}_{i,m-1-\ell}(x).
\end{eqnarray*}
Setting $x=1$ and using (\ref{eq11}) and (\ref{h0}) yields
\begin{equation}\label{eq12}
f^{(\ell)}_{i,m-\ell}(1)=f^{(\ell+1)}_{i,m-1-\ell}(1),\ {\rm{i.e.,}}\ t^{(\ell)}_{i,m-\ell}=t^{(\ell+1)}_{i,m-1-\ell}.
\end{equation}
Applying (\ref{eq12}) to $\ell=0,1,\dots,m$, we get
\begin{equation*}
t^{(0)}_{i,m}=t^{(1)}_{i,m-1}=t^{(2)}_{i,m-2}=\cdots =t^{(m)}_{i,0}=t^{(m+1)}_{i,-1}=0,
\end{equation*}
where the last step follows from (\ref{fim0}).

Thus the statement holds for $m$ as well. This completes the proof.
\end{proof}

\begin{corollary}\label{coru}
For any $u \ge 1$, we have (\ref{2u-10}) and (\ref{2u0}), i.e.,
\begin{equation}
t_{i,2u-1}=0,\ i\ge u \ {\rm{and}}\ t_{i,2u}=0,\ i>u.
\end{equation}
\end{corollary}

\begin{proof}
Applying Proposition \ref{proptl} with $(i,m,\ell)=(i,2u-1,0)$ $(i\ge u)$ and $(i,m,\ell)=(i,2u-1,0)$ $(i>u)$.
\end{proof}

Thus it remains to show (\ref{u2u}).

\begin{proposition}\label{proptu2u}
For any $u\ge 1$ we have

{\rm{(a)}}$f_{u,u+v}(x)$ with $0\le v\le u$ is of the form
\begin{equation}\label{tuu+v}
f_{u,u+v}(x)=r_{u,v}q_1(x)x^v h^{u-v}(x)\left(h'(x)\right)^v+s_{u,v}(x)h^{u+1-v}(x),
\end{equation}
where $r_{u,v}$\label{ruv} is a constant determined by $u$ and $v$, $s_{u,v}(x)$ is a polynomial of the $h^{(\ell)}(x)$'s and the $q_1^{(\ell)}(x)$'s $(\ell \ge 0)$ with coefficients polynomials of $x$.

{\rm{(b)}} $r_{u,0}=1$ and
\begin{equation}\label{rrec}
r_{u,v}=r_{u-1,v}+(u-v+1)r_{u,v-1},\ r_{u,u}=r_{u,u-1},\ 1\le v< u.
\end{equation}

{\rm{(c)}}

\vspace{-0.35in}

\begin{equation}\label{ruu}
r_{u,u}=(2u-1)!!.
\end{equation}
\end{proposition}

\begin{proof}
We proceed by induction on $u+v$.

By (\ref{f00}) and (\ref{fmm}), we get
\begin{equation*}
f_{u,u}(x)=q_1(x)h^{u}(x),\ u\ge 1.
\end{equation*}
Hence (a) holds for $v=0$ and $r_{u,0}=1$.

Since the only $(u,v)$ with $u+v=1$ and $0\le v\le u$ is $(0,1)$, (a) holds for $u+v=1$. Assume that (a) holds for $u+v\le t$ $(t\ge 1)$. We will simultaneously prove (a) and (b). If $u+v=t+1$, we have shown that the statement holds for $v=0$. For $1\le v\le u$, we have three cases: $v=1$, $1<v<u$ and $1<v=u$.

When $1\le v< u$, applying (\ref{frec}) to $(i,m,\ell)=(u,u+v,0)$ and using the induction hypothesis for $(u-1,v),(u,v-1)$, we get
\begin{eqnarray}\label{eq20}
\label{eq24} & &f_{u,u+v}(x)=h(x)f_{u-1,u+v-1}+d'f_{u,u+v-1}+xf'_{u,u+v-1}\\
\nonumber &=&h(x)\left[r_{u-1,v} q_1(x)x^vh^{u-1-v}(x)\left(h'(x)\right)^v+s_{u-1,v}(x)h^{u-v}(x)\right]\\
\nonumber & &+d'\left[r_{u,v-1}q_1(x)x^{v-1}h^{u-v+1}(x)\left(h'(x)\right)^{v-1}+s_{u,v-1}(x)h^{u+2-v}(x)\right]\\
\nonumber & &+x\left[r_{u,v-1}q_1(x)x^{v-1}h^{u-v+1}(x)\left(h'(x)\right)^{v-1}+s_{u,v-1}(x)h^{u+2-v}(x)\right]'\\
\nonumber &=&r_{u-1,v}q_1(x)x^v h^{u-v}(x)\left(h'(x)\right)^v+[s_{u-1,v}(x)\\
\nonumber & &+d'r_{u,v-1}q_1(x)x^{v-1}\left(h'(x)\right)^{v-1}+d's_{u,v-1}(x)h(x)]h^{u+1-v}(x)\\
\nonumber & &+x\left[r_{u,v-1}q_1(x)x^{v-1}h^{u-v+1}(x)\left(h'(x)\right)^{v-1}+
s_{u,v-1}(x)h^{u+2-v}(x)\right]'.
\end{eqnarray}
Denote the last line of (\ref{eq20}) by $W$.

\vspace{0.1in}

\noindent{\bf{Case 1.}} $v=1$. We have

\begin{eqnarray*}
\nonumber W&=&x\left[r_{u,v-1}q'_1(x)h^{u-v+1}(x)+(u-v+1)r_{u,v-1}q_1(x)h'(x)h^{u-v}(x)\right.\\
\nonumber & &\left.+(u+2-v)s_{u,v-1}(x)h'(x)h^{u+1-v}(x)\right]\\
\nonumber &=&x\left[r_{u,v-1}q'_1(x)+(u+2-v)s_{u,v-1}(x)h'(x)\right]h^{u-v+1}(x)\\
& &+x(u-v+1)r_{u,v-1}q_1(x)h'(x)h^{u-v}(x).
\end{eqnarray*}

Noting that $v=1$, thus the above equation can be written as
\begin{eqnarray*}
\nonumber W&=&x\left[r_{u,v-1}q'_1(x)+(u+2-v)s_{u,v-1}(x)h'(x)\right]h^{u-v+1}(x)\\
& &+(u-v+1)r_{u,v-1}q_1(x)x^vh^{u-v}(x)\left(h'(x)\right)^v.
\end{eqnarray*}

Plugging this into (\ref{eq20}) yields
\begin{eqnarray*}
\nonumber
f_{u,u+v}(x)
&=&
r_{u-1,v}q_1(x)x^v h^{u-v}(x)\left(h'(x)\right)^v+[s_{u-1,v}(x)\\
\nonumber & &+d'r_{u,v-1}q_1(x)x^{v-1}\left(h'(x)\right)^{v-1}+d's_{u,v-1}(x)h(x)]h^{u+1-v}(x)\\
\nonumber & &+x\left[r_{u,v-1}q'_1(x)+(u+2-v)s_{u,v-1}(x)h'(x)\right]h^{u-v+1}(x)\\
\nonumber & &+(u-v+1)r_{u,v-1}q_1(x)x^vh^{u-v}(x)\left(h'(x)\right)^v\\
\nonumber &=&[r_{u-1,v}+(u-v+1)r_{u,v-1}]q_1(x)x^v h^{u-v}(x)\left(h'(x)\right)^v+[s_{u-1,v}(x)\\
\nonumber & &+d'r_{u,v-1}q_1(x)x^{v-1}\left(h'(x)\right)^{v-1}+d's_{u,v-1}(x)h(x)+xr_{u,v-1}q'_1(x)\\
& &+x(u+2-v)s_{u,v-1}(x)h'(x)]h^{u-v+1}(x).
\end{eqnarray*}
Hence $f_{u,u+v}(x)$ is of the form (\ref{tuu+v}) and (\ref{rrec})
holds.

\vspace{0.1in}

\noindent{\bf{Case 2.}} $1<v<u$. We have
\begin{eqnarray}
\nonumber W&=&(u-v+1)r_{u,v-1}q_1(x)x^{v}h^{u-v}(x)\left(h'(x)\right)^{v}\\
\nonumber & &+[r_{u,v-1}q'_1(x)x^{v}+(v-1)r_{u,v-1}q_1(x)x^{v-1}\left(h'(x)\right)^{v-1}\\
\nonumber & &+(v-1)r_{u,v-1}q_1(x)x^{v}\left(h'(x)\right)^{v-2}h''(x)\\
\nonumber & &+(u+2-v)xs_{u,v-1}(x)h'(x)]h^{u+1-v}(x)
\end{eqnarray}
Plugging this into (\ref{eq20}) yields
\begin{eqnarray}
\nonumber
f_{u,u+v}(x)
&=&[r_{u-1,v}+(u-v+1)r_{u,v-1}]q_1(x)x^v h^{u-v}(x)\left(h'(x)\right)^v+[s_{u-1,v}(x)\\
\nonumber & &+d'r_{u,v-1}q_1(x)x^{v-1}\left(h'(x)\right)^{v-1}+d's_{u,v-1}(x)h(x)+r_{u,v-1}q'_1(x)x^{v}\\
\nonumber & &+(v-1)r_{u,v-1}q_1(x)x^{v-1}\left(h'(x)\right)^{v-2}(h'(x)+xh''(x))\\
\nonumber & &+(u+2-v)xs_{u,v-1}(x)h'(x)]h^{u+1-v}(x).
\end{eqnarray}
Hence $f_{u,u+v}(x)$ is of the form (\ref{tuu+v}) and (\ref{rrec}) holds in this case too.

\vspace{0.1in}

\noindent{\bf{Case 3.}} $1<v=u.$ Thus $u\ge 2$. From the recurrence relation (\ref{frec}) and the initial condition (\ref{f00}), we see that each $f_{i,m}$ is a polynomial of the $h^{(\ell)}(x)$'s and the $q_1^{(\ell)}(x)$'s $(\ell \ge 0)$ with coefficients polynomials of $x$.
By (\ref{eq24}) and the induction hypothesis (\ref{tuu+v}) for $(u,v)=(u,u-1)$, we get
\begin{eqnarray}
\nonumber & &f_{u,u+v}(x)=f_{u,2u-1}(x)=h(x)f_{u-1,2u-1}+d'f_{u,2u-1}+xf'_{u,2u-1}\\
\nonumber &=&h(x)f_{u-1,2u-1}+r_{u,u-1}q_1(x)x^{u-1} h(x)\left(h'(x)\right)^{u-1}+s_{u,u-1}(x)h^2(x)\\
\nonumber & &+x[r_{u,u-1}q_1(x)x^{u-1} h(x)\left(h'(x)\right)^{u-1}+s_{u,u-1}(x)h^2(x)]'\\
\nonumber &=&r_{u,u-1}q_1(x)x^u\left(h'(x)\right)^{u}+[f_{u-1,2u-1}+r_{u,u-1}q_1(x)x^{u-1}\left(h'(x)\right)^{u-1}
+s_{u,u-1}(x)h(x)\\
\nonumber & &+r_{u,u-1}q'_1(x)x^{u}\left(h'(x)\right)^{u-1}
+ (u-1)r_{u,u-1}q_1(x)x^{u-1}\left(h'(x)\right)^{u-2}(h'(x)+xh''(x))\\
\nonumber & &+xs'_{u,u-1}(x)h(x)+2xs_{u,u-1}(x)h'(x)]h(x).
\end{eqnarray}
Hence $f_{u,u+v}(x)$ is of the form (\ref{tuu+v}) and (\ref{rrec}) holds in this case, completing the proof of (a) and (b).

We use generating functions to prove (c). The proof of (c) is an immediate consequence of Lemma \ref{lemTv} (see Remark \ref{rek:proofpropc} for the details).
\end{proof}

\begin{lemma}\label{lemTv}
Define
\begin{equation}\label{Tvdef}
T_v(x)=\sum_{u=v}^{\infty} r_{u,v}x^{u-v},\ v\ge 0.
\end{equation}
Then we have

{\rm{(a)}}\begin{equation}\label{Tv}
T_v(x)=\frac{T'_{v-1}(x)}{1-x},\ v\ge 1.
\end{equation}

{\rm{(b)}}
\begin{equation}\label{2v-1}
T_0(x)=\frac{1}{1-x}\ {\rm{and}}\
T_v(x)=\frac{(2v-1)!!}{(1-x)^{2v+1}},\ v\ge 1.
\end{equation}
\end{lemma}

\begin{proof}
(a) According to Definition (\ref{Tvdef}),
\begin{eqnarray}
\nonumber (1-x)T_v(x)&=&\sum_{u=v}^{\infty} r_{u,v}x^{u-v}-\sum_{u=v}^{\infty} r_{u,v}x^{u-v+1}
= \sum_{u=v}^{\infty} r_{u,v}x^{u-v}-\sum_{u=v+1}^{\infty} r_{u-1,v}x^{u-v}\\
\nonumber &=& r_{v,v}+\sum_{u=v+1}^{\infty} (r_{u,v}-r_{u-1,v})x^{u-v}.
\end{eqnarray}

By the recurrence relation (\ref{rrec}), we get
\begin{equation*}
r_{u,v}-r_{u-1,v}=(u-v+1)r_{u,v-1}\ {\rm{for}}\ u\ge v+1,\ {\rm{and}}\ r_{v-1,v}=r_{v,v}.
\end{equation*}
Thus
\begin{eqnarray}\label{tv}
\nonumber (1-x)T_v(x)&=&r_{v,v}+\sum_{u=v+1}^{\infty}(u-v+1)r_{u,v-1}x^{u-v}
= r_{v-1,v}+\sum_{u=v+1}^{\infty}(u-v+1)r_{u,v-1}x^{u-v}\\
&=& \sum_{u=v}^{\infty}(u-v+1)r_{u,v-1}x^{u-v}.
\end{eqnarray}
On the other hand, taking the derivative of both sides of Definition (\ref{Tvdef}), we see that $T'_{v-1}(x)$ also equals (\ref{tv}). Therefore (\ref{Tv}) holds.

(b) Since $r_{u,0}=1$ (see Proposition \ref{proptu2u}(b)), we have
\begin{equation*}
T_0(x)=\sum_{u=0}^{\infty} r_{u,0}x^u=\sum_{u=0}^{\infty}x^u=\frac{1}{1-x}.
\end{equation*}

Applying (a) to $v=1$, we get
\begin{equation*}
T_1(x)=\frac{T'_0(x)}{1-x}=\frac{1}{1-x}\left(\frac{1}{1-x}\right)'=\frac{1}{(1-x)^3}.
\end{equation*}
Thus (\ref{2v-1}) holds for $v=1$.

Assume that (\ref{2v-1}) holds for $v-1$ $(v\ge 2)$. It follows from (a) and the induction hypothesis that
\begin{eqnarray*}
\nonumber T_v(x)=\frac{T'_{v-1}(x)}{1-x}=\frac{1}{1-x}\left(\frac{(2v-3)!!}{(1-x)^{2v-1}}\right)'=\frac{(2v-1)!!}{(1-x)^{2v+1}}.
\end{eqnarray*}
Hence (\ref{2v-1}) holds for $v$ and therefore for any $v\ge 1$.
\end{proof}

\begin{rek}\label{rek:proofpropc}
The proof of part (c) of Proposition \ref{proptu2u} is immediate, as any $u\ge 1$,
\begin{equation*}
r_{u,u}=T_u(0)=(2u-1)!!
\end{equation*}
by Definition (\ref{Tvdef}) and Lemma \ref{lemTv}.
\end{rek}

Setting $v=u$ and $x=1$ in Proposition \ref{proptu2u}(a) and using (\ref{h0}) and (\ref{ruu}), we get
\begin{equation*}
f_{u,2u}(1)=r_{u,u}q_1(1)\left(h'(1)\right)^u=(2u-1)!!q_1(1)\left(h'(1)\right)^u,
\end{equation*}
as desired, completing the proof of Theorem \ref{Gauss}.

\section{Far-difference Rrepresentation}\label{hannah}
In this section, we apply the generating function approach to study the distributions of the numbers of positive and negative summands in the far-difference representation of integers in $(S_n, S_{n+1}]$ (see Definition \ref{far}). We prove that as $n\to\infty$ these two random variables converge to being a bivariate Gaussian with a computable, negative correlation. We do not need to prove that a generalization of Zeckendorf's theorem holds for far-difference representations, as this was done by Alpert \cite{Al} (see Theorem \ref{thmhan}).

\subsection{Generating Function of the Probability Density}

Let $p_{n,k,l}$\label{p_{n,k,l}} $(n>0)$ be the number of far-difference representations of integers in $(S_{n-1},S_n]$ with $k$ positive summands and $l$ negative summands. We have the following formula for the generating function $\hat{\mathscr{G}}(x,y,z)=\sum_{n>0,k>0,l\ge 0}p_{n,k,l}x^ky^lz^n$\label{hatmathscrG}.

\begin{theorem}
We have
\begin{equation}\label{Gxyz}
\hat{\mathscr{G}}(x,y,z)=\frac{xz+xyz^4}{1-z-(x+y)z^4-xyz^6-xyz^7}.
\end{equation}
\end{theorem}

\begin{proof} We first derive the recurrence relation
 \begin{equation}\label{prec}
p_{n,k,l}=p_{n-1,k,l}+p_{n-4,k-1,l}+p_{n-3,l,k-1},\ n\ge 5
\end{equation}
by a combinatorial approach. Next we want to get the generating function by the same technique as for $\mathscr{G}(x,y)$ in Section \ref{sec:genfnprobden}. To achieve that, we need to have a recurrence relation with all terms of form $p_{n-n0,k-k_0,l-l_0}$ with $n_0,k_0$ and $l_0$ constant. We solve this by using the proceeding recurrence relation with repeated substitutions.

\vspace{0.1in}

Let us prove (\ref{prec}) first. Clearly, $p_{n,k,l}=0$ if $k\le 0$ or $l<0$. For every far-difference representation $N=\sum_{j=1}^{m} a_jF_{i_j}\in [S_{n-1}+1,S_n]$, $N':=\sum_{j=2}^{m} a_jF_{i_j}$ is also a far-difference representation. Theorem \ref{thmhan} states that $i_1=n$ and $a_1=1$, therefore $N'\in [S_{n-1}+1-F_n,S_n-F_n]$. Since \begin{eqnarray}\label{Sn-3}
\nonumber F_n-S_{n-1}-S_{n-3}&=&F_n-F_{n-1}-F_{n-3}-F_{n-5}\cdots = F_{n-2}-F_{n-3}-F_{n-5}-\cdots\\
 &=& F_{n-4}-F_{n-5}-\cdots = \cdots (=F_3-F_2)=F_2-F_1=1,
\end{eqnarray}
we get $S_{n-1}+1-F_n=-S_{n-3}$. Thus $p_{n,k,l}$ is the number of far-difference representations of integers in $[-S_{n-3},S_{n-4}]$ with $k-1$ positive summands and $l$ negative summands.

Let $n\ge 5$. We have two cases: $(k-1,l)\neq (0,0)$ and $(k-1,l)= (0,0)$.

\begin{case1}
$(k-1,l)= (0,0)$.
\end{case1}
Since $F_n-S_{n-1}-S_{n-3}=1$ by (\ref{Sn-3}), we have $F_{n-1}<S_{n-1}<F_n$ for all $n>1$. Hence there is exactly one Fibonacci number in $[S_{n-1}+1,S_n]$ for all $n>1$. Thus $p_{n,1,0}=p_{n-1,1,0}=1$. Further, for $n\ge 5$, we have $p_{n-4,0,0}=p_{n-3,0,0}=0$, then (\ref{prec}) follows.

\begin{case2}
$(k-1,l)\neq (0,0)$.
\end{case2}
Then $N'=N-a_1F_{i_1}\neq 0$. Let $N(J,k,l)$ be the number of far-difference representations of integers in the interval $J$ with $k$ positive summands and $l$ negative summands. Thus \begin{eqnarray}\label{pnkl}
\nonumber p_{n,k,l}&=&N((0,S_{n-4}],k-1,l)+N([-S_{n-3},0),k-1,l)\\
\nonumber &=&N((0,S_{n-4}],k-1,l)+N((0,S_{n-3}],l,k-1)\\
&=&\sum_{i=1}^{n-4}p_{i,k-1,l}+\sum_{i=1}^{n-3}p_{i,l,k-1}.
\end{eqnarray}
For $n\ge 5$, replacing $n$ with $n-1$ yields
\begin{equation}\label{pn-1kl}
p_{n-1,k,l}=\sum_{i=1}^{n-5}p_{i,k-1,l}+\sum_{i=1}^{n-4}p_{i,l,k-1}.
\end{equation}
Subtracting (\ref{pn-1kl}) from (\ref{pnkl}), we get (\ref{prec}).

\vspace{0.1in}

Let $n\ge 9$. Replacing $(n,k,l)$ in (\ref{prec}) with $(n-3,l,k-1)$ gives
\begin{equation}\label{pnlk4}
p_{n-3,l,k-1}=p_{n-4,l,k-1}+p_{n-7,l-1,k-1}+p_{n-6,k-1,l-1},\ n\ge 8.
\end{equation}
Rearranging the terms of (\ref{prec}), we obtain
\begin{equation}\label{pnlk}
p_{n-3,l,k-1}=p_{n,k,l}-p_{n-1,k,l}-p_{n-4,k-1,l},\ n\ge 5.
\end{equation}
Replacing $(n,k,l)$ in (\ref{prec}) with $(n-1,k,l)$ and $(n-4,k,l-1)$ (since $n\ge 9$, $n-1>n-4\ge 5$, thus (\ref{pnlk}) applies to $n-1$ and $n-4$), we get
\begin{equation}\label{pnlk2}
p_{n-4,l,k-1}=p_{n-1,k,l}-p_{n-2,k,l}-p_{n-5,k-1,l}
\end{equation}
and
\begin{equation}\label{pnlk3}
p_{n-7,l-1,k-1}=p_{n-4,k,l-1}-p_{n-5,k,l-1}-p_{n-8,k-1,l-1}.
\end{equation}
Plugging (\ref{pnlk4}), (\ref{pnlk2}) and (\ref{pnlk3}) into (\ref{prec}) yields
\begin{eqnarray}\label{klrec}
\nonumber p_{n,k,l}&=&2p_{n-1,k,l}-p_{n-2,k,l}+p_{n-4,k-1,l}+p_{n-4,k,l-1}-p_{n-5,k-1,l}\\
& &-p_{n-5,k,l-1}+p_{n-6,k-1,l-1}-p_{n-8,k-1,l-1},\ n\ge 9.
\end{eqnarray}

Multiplying both sides of (\ref{klrec}) by $x^ky^lz^n$, we get
\begin{eqnarray*}
\nonumber p_{n,k,l}x^ky^lz^n&=&2zp_{n-1,k,l}x^ky^lz^{n-1}-z^2p_{n-2,k,l}x^ky^lz^{n-2}
+xz^4p_{n-4,k-1,l}x^{k-1}y^lz^{n-4}\\
\nonumber & &+yz^4p_{n-4,k,l-1}x^{k-1}y^lz^{n-4}
-xz^5p_{n-5,k-1,l}x^{k-1}y^lz^{n-5}\\
\nonumber & &-yz^5p_{n-5,k,l-1}x^ky^{l-1}z^{n-5}
+xyz^6p_{n-6,k-1,l-1}x^{k-1}y^{l-1}z^{n-6}\\
& &-xyz^8p_{n-8,k-1,l-1}x^{k-1}y^{l-1}z^{n-8}.
\end{eqnarray*}
Summing both sides over $n\ge 9$ and recalling that $p_{n,k,l}=0$ if $k\ge 0$ or $l<0$, we obtain
\begin{eqnarray}\label{eq13}
\nonumber
\hat{\mathscr{G}}(x,y,z)
&=&2z\hat{\mathscr{G}}(x,y,z)-2\sum_{1<n\le 8} p_{n-1,k,l}x^ky^lz^n-z^2\hat{\mathscr{G}}(x,y,z)\\
\nonumber & &+\sum_{2<n\le 8}p_{n-2,k,l}x^ky^lz^n+xz^4\hat{\mathscr{G}}(x,y,z)-\sum_{4<n\le 8}p_{n-4,k-1,l}x^ky^lz^n\\
\nonumber & &+yz^4\hat{\mathscr{G}}(x,y,z)-\sum_{4<n\le 8}p_{n-4,k,l-1}x^ky^lz^n-xz^5\hat{\mathscr{G}}(x,y,z)\\
\nonumber & &+\sum_{5<n\le 8}p_{n-5,k-1,l}x^ky^lz^n-yz^5\hat{\mathscr{G}}(x,y,z)+\sum_{5<n\le 8}p_{n-5,k,l-1}x^ky^lz^n\\
\nonumber & &+xyz^6\hat{\mathscr{G}}(x,y,z)-\sum_{6<n\le 8}p_{n-6,k-1,l-1}x^ky^lz^n-xyz^8\hat{\mathscr{G}}(x,y,z)\\
\nonumber &=&\left(2z-z^2+xz^4+yz^4-xz^5-yz^5+xyz^6-xyz^8\right)\hat{\mathscr{G}}(x,y,z)\\
\nonumber & &-2\sum_{1<n\le 8} p_{n-1,k,l}x^ky^lz^n+\sum_{2<n\le 8}p_{n-2,k,l}x^ky^lz^n
-\sum_{4<n\le 8}p_{n-4,k-1,l}x^ky^lz^n\\
\nonumber & &-\sum_{4<n\le 8}p_{n-4,k,l-1}x^ky^lz^n
+\sum_{5<n\le 8}p_{n-5,k-1,l}x^ky^lz^n+\sum_{5<n\le 8}p_{n-5,k,l-1}x^ky^lz^n\\
& &-\sum_{6<n\le 8}p_{n-6,k-1,l-1}x^ky^lz^n.
\end{eqnarray}
We calculated all $p_{n,k,l}$'s for $n\le 8$ and found that the only terms in the right-hand side of (\ref{eq13}) that are not canceled are $xz$, $-xz^2$, $xyz^4$ and $-xyz^5$, therefore
\begin{eqnarray}
\nonumber \hat{\mathscr{G}}(x,y,z)&=&\frac{x(z-z^2)+xy(z^4-z^5)}{1-\left(2z-z^2+xz^4+yz^4-xz^5-yz^5+xyz^6-xyz^8\right)}\\
&=&\frac{xz+xyz^4}{1-z-(x+y)z^4-xyz^6-xyz^7}.
\end{eqnarray}
\end{proof}

\subsection{Lekkerkerker's Theorem and Gaussian Behavior}\label{GaussianF}\label{subsubsecak+bl}

To show that $\mathcal{K}_n$ and $\mathcal{L}_n$ are bivariate Gaussian, it suffices to prove the Gaussian behavior of $a\mathcal{K}_n+b\mathcal{L}_n$ for any $a$, $b$ with $(a,b)\neq (0,0)$\label{ab}. Note that the coefficient of $z^n$ in $\hat{\mathscr{G}}(x,y,z)$ is $\sum_{k>0,l\ge 0}p_{n,k,l}x^ky^l$. Setting $(x,y)=(w^a,w^b)$ and using differentiating identities will give the moments of $a\mathcal{K}_n+b\mathcal{L}_n$.

We first prove the following generalized Lekkerkerker's Theorem and Gaussian behavior for $a\mathcal{K}_n+b\mathcal{L}_n$. This suffices to deduce Theorem \ref{thm:lekgaussfardiff} as ${\rm{cov}}(\mathcal{K}_n,\mathcal{L}_n) =\frac14{\rm{var}}(\mathcal{K}_n+\mathcal{L}_n) -\frac14{\rm{var}}(\mathcal{K}_n-\mathcal{L}_n)$.

\begin{theorem}\label{thm:aK+bL}
For any real numbers $(a,b)\neq (0,0)$, the mean of $a\mathcal{K}_n+b\mathcal{L}_n$ is
\begin{equation}
\frac{a+b}{10}n+\frac{371-113\sqrt{5}}{40}\ a+\frac{361-123\sqrt{5}}{40}\ b+o(\hat{\gamma}^n_{a,b})\ {\rm{for\ some}}\ \hat{\gamma}_{a,b}\in (0,1),
\end{equation}
and the variance of $a\mathcal{K}_n+b\mathcal{L}_n$ is
\begin{equation}\label{varianceaK+bL}
\frac{\sqrt{5}-1}{200}\left[10\left (a^2+b^2\right)-\frac{20-\sqrt{5}}{5}(a+b)^2\right]n+q_{a,b}+o(\hat \tau_{a,b}^n)\ {\rm{for\ some}}\ \hat \tau_{a,b}\in (0,1),
\end{equation}
with $q_{a,b}$ a constant depending on only $a$ and $b$; further, standardization of $a\mathcal{K}_n+b\mathcal{L}_n$ converges weakly to a Gaussian as $n\rightarrow \infty$; in other words, $\mathcal{K}_n$ and $\mathcal{L}_n$ are bivariate Gaussian as $n\rightarrow \infty$.
\end{theorem}

Let $\hat{A}(z)$ be the denominator of $\hat{\mathscr{G}}(x,y,z)$, i.e.,
\begin{equation}\label{hatA}
\hat{A}(z)=1-z-(x+y)z^4+xyz^6+xyz^7
\end{equation}
Clearly, 0 is not a root of $\hat{A}(z)$. When $x=y=1$, we have
\begin{equation}\label{decA}
\hat{A}(z)=1-z-2z^4-z^6-z^7=-(z^2+z-1)(z^2
+1)(z^3+1).
\end{equation}
Thus $\hat{A}(z)$ has no multiple roots; moreover, except $\frac{\sqrt{5}-1}{2}$, any other root $z$ of $\hat{A}(z)$ satisfies $|z|\le 1$. Note that in both cases $x=1$ and $y=1$, the coefficients of $\hat{A}(z)$ are polynomials in one variable and hence continuous, thus the roots of $\hat{A}(z)$ are continuous (see \cite{US} or Appendix A of \cite{MW}).

To evaluate the moments of $a\mathcal{K}_n+b\mathcal{L}_n$, we set $(x,y)=(w^a,w^b)$ and let $\hat{A}_w(z)$\label{hatAw} be the corresponding $\hat{A}(z)$, namely
\begin{equation*}
\hat{A}_w(z)=1-z-(w^a+w^b)z^4-w^{a+b}z^6-w^{a+b}z^7.
\end{equation*}

We have the following proposition similarly to Proposition \ref{proppar} (see Appendix \ref{apropAw} for the proof).

\begin{proposition}\label{propAw}
There exists $\varepsilon\in (0,1)$ such that for any $w\in I_{\varepsilon}=(1-\varepsilon,1+\varepsilon)$:

{\rm{(a)}} $\hat{A}_w(z)$ has exactly 7 roots but no multiple roots.

{\rm{(b)}} There exists a root $e_1(w)$ such that $|e_1(w)|<1$ and $|e_1(w)|<|e_i(w)|$, $1< i\le 7$.

{\rm{(c)}} Each root $e_i(w)$ $(1\le i\le 7)$ is continuous and $\ell$-times differentiable for any $\ell\ge 1$, and \begin{equation}\label{e'}
 e'_i(w)= -\frac{\left(aw^{a-1}+bw^{b-1}\right)e^4_i(w) +(a+b)w^{a+b-1}[e^6_i(w)+e^7_i(w)]}{1+4(w^a+w^b)e^3_i(x)+6w^{a+b}e^5_i(w)+7w^{a+b}e^6_i(w)}
\end{equation}

{\rm{(d)}}

\vspace{-0.4in}

\begin{equation}\label{parw}
\frac{1}{\hat{A}_w(z)}=-\frac{1}{w^{a+b}}\sum_{i=1}^7\frac{1}{(z-e_i(w))\prod_{j\neq i}\left(e_j(w)-e_i(w)\right)}.
\end{equation}
\end{proposition}

Let us return to the proof of Theorem \ref{thm:aK+bL}.

\begin{proof}[Proof of Theorem \ref{thm:aK+bL}]
Assume $w\in I_{\varepsilon}$. Combining (\ref{Gxyz}) and Proposition \ref{propAw}(d), we get
\begin{equation*}
\hat{\mathscr{G}}(w^a,w^b,z) =-(z+w^bz^4)\sum_{i=1}^7\frac{1}{w^b(z-e_i(w))\prod_{j\neq i}\left(e_j(w)-e_i(w)\right)}.
\end{equation*}
Denote $\hat{g}(w)$ the coefficient of $z^n$ in $\hat{\mathscr{G}}(w^a,w^b,z)$, i.e.,
\begin{equation*}
\hat{g}(w)=\sum_{k>0,l\ge 0}p_{n,k,l}w^{ak+bl},
\end{equation*}
then
\begin{eqnarray*}
\nonumber \hat{g}(w)&=&\langle z^{n-4} \rangle\sum_{i=1}^7\frac{1}{(1-\frac{z}{e_i(w)})e_i(w)\prod_{j\neq i}\left(e_j(w)-e_i(w)\right)}\\
\nonumber & &+\langle z^{n-1} \rangle\sum_{i=1}^7\frac{1}{w^b(1-\frac{z}{e_i(w)})e_i(w)\prod_{j\neq i}\left(e_j(w)-e_i(w)\right)}\\
\nonumber &=&\sum_{i=1}^7\frac{1}{e^{n-3}_i(w)\prod_{j\neq i}\left(e_j(w)-e_i(w)\right)}
+\sum_{i=1}^7\frac{1}{w^b e^{n}_i(w)\prod_{j\neq i}\left(e_j(w)-e_i(w)\right)}\\
&=&\sum_{i=1}^7\frac{w^{-b}+e^3_i(w)}{e^{n}_i(w)\prod_{j\neq i}\left(e_j(w)-e_i(w)\right)}.
\end{eqnarray*}
Let\label{hatqw}
\begin{equation*}
\hat{q}_i(w)=\frac{w^{-b}+e^3_i(w)}{w\prod_{j\neq i}\left(e_j(w)-e_i(w)\right)}.
\end{equation*}
Then $\hat{g}(w)=\sum_{i=1}^7 w\hat{q}_i(w)/e^{n}_i(w)$. Since $e_i(w)$ is $\ell$-times differentiable for any $\ell$, so is $\hat{q}_i(w)$.

Similarly to Theorem \ref{thm:genlekkerkerker} with (\ref{Cd}) and Theorem \ref{thm:varKn} with (\ref{hd'}), we have
\begin{equation}\label{EaK+bL}
\mathbb{E}[a\mathcal{K}_n+b\mathcal{L}_n]=\hat C_{a,b}n+\hat d_{a,b} + o(\hat{\gamma}^n_{a,b}) \ {\mbox{and}}\ {\mbox{var}}(a\mathcal{K}_n+b\mathcal{L}_n)=\hat h'_{a,b}(1)n+\hat q''_1(1)+o(\hat \tau_{a,b}^n)
\end{equation}
with
$$\hat C_{a,b}=-e'_1(1)/e_1(1),\ \hat d_{a,b}=1+\frac{\hat q'_1(1)}{\hat q_1(1)},\ \hat h_{a,b}(w)=-\frac{we'_1(w)}{e_1(w)}-\hat C_{a,b}$$
and constants $\hat{\gamma}_{a,b},\hat \tau_{a,b}\in(0,1)$ and  $\hat q''_1(1)$ depending on only $a$ and $b$.

Setting $w=1$ in (\ref{e'}) and using $e_1(1)=\Phi$ (with $\Phi = (\sqrt{5}-1)/2$), we get $\hat C_{a,b}=-e'_1(1)/e_1(1)=(a+b)/10$. It is more difficult to calculate $\hat d_{a,b}$ but still tractable. We show that
$$\hat d_{a,b}=\frac{371-113\sqrt{5}}{40}\ a+\frac{361-123\sqrt{5}}{40}\ b.$$


Recall from (\ref{hatqw}) that
\begin{equation}
\hat{q}_{1}(w)=\frac{w^{-b}+e^3_i(w)}{w\prod_{j\neq 1}\left(e_j(w)-e_1(w)\right)}.
\end{equation}
Let
\begin{equation}\label{E}
\hat E(w)=\prod_{j\neq 1}\left(e_j(w)-e_1(w)\right),
\end{equation}
then
\begin{equation*}
\hat q_1(w)=\frac{w^{-b}+e^3_1(w)}{w\hat E(w)}.
\end{equation*}
Thus
\begin{eqnarray*}
\nonumber & &\hat d_{a,b}=1+\frac{\hat q'_{1}(1)}{\hat q_{1}(1)}
=1+\frac{[(w^{-b}+e^3_1(w))'w\hat E(w)-(w\hat E(w))'(w^{-b}+e^3_1(w)))]/(w\hat E(w))^2}{(w^{-b}+e^3_1(w))/(w\hat E(w))}\\
&=&1+\frac{(w^{-b}+e^3_1(w))'}{w^{-b}+ e^3_1(w)}-\frac{(w\hat E(w))'}{w\hat E(w)}=1+\frac{-bw^{-b-1}+3e^2_1(w) e'_1(w)}{w^{-b}+e^3_1(w)}-\frac{\hat E(w)+w\hat E'(w)}{w\hat E(w)}.
\end{eqnarray*}
Setting $x=1$ and using $e_1(1)=\Phi$ and $e'_1(1)=-(a+b)\Phi/10$, we get
\begin{equation}\label{hatd}
\hat d_{a,b} =\frac{-b-\frac{3}{10}(a+b)\Phi^3}{1+\Phi^3}-\frac{\hat E'(1)}{\hat E(1)}= -\frac{\sqrt{5}+1}{4}b-\frac{9-3\sqrt{5}}{40}(a+b)-\frac{\hat E'(1)}{\hat E(1)}.
\end{equation}

Thus it remains to evaluate $\hat E(1)$ and $\hat E'(1)$. Consider $\hat{A}_{w}(e'+e_1(w))$:
\begin{equation}\label{eq17}
\hat{A}_{w}(e'+e_1(w))=1-e'-e_1(w)-(w^a+w^b)(e'+e_1(w))^4-w^{a+b}(e'+e_1(w))^6-w^{a+b}(e'+e_1(w))^7.
\end{equation}
On the other hand, we have
\begin{equation}\label{eq18}
\hat{A}_{w}(e'+e_1(w))=-w^{a+b}\prod_{j\neq 1}(e'+e_1(w)-e_j(w)).
\end{equation}
Comparing the coefficients of $e'$ in (\ref{eq17}) and (\ref{eq18}) gives
\begin{equation*}
w^{a+b}\prod_{j\neq 1}(e_1(w)-e_j(w))=1+4(w^a+w^b)e^3_1(w)+6w^{a+b}e^5_1(w)+7w^{a+b}e^6_1(w).
\end{equation*}
Thus
\begin{equation}\label{Ew}
\hat E(w) =\prod_{j\neq 1}(e_1(w)-e_j(w)) =w^{-(a+b)}+4(w^{-b}+w^{-a})e^3_1(w)+6e^5_1(w)+7e^6_1(w).
\end{equation}
Setting $w=1$, we get
\begin{equation*}
\hat E(1) =1+8\Phi^3+6\Phi^5+7\Phi^6=10\Phi^2.
\end{equation*}
Differentiating both sides of (\ref{Ew}) yields
\begin{eqnarray*}
\nonumber \hat E'(x)
=-(a+b)w^{-(a+b+1)}-4\left(aw^{-a-1}+bw^{-b-1}\right)e^3_1(w)+30e^4_1(w)e'_1(w)+42e^5_1(w)e'_1(w).
\end{eqnarray*}
Setting $x=1$ and plugging in $e_1(1)=\Phi$ and $e'_1(1)=-(a+b)\Phi/10$ yields
\begin{eqnarray*}
\nonumber \hat E'(1)
=-(a+b)-4(a+b)\Phi^3+30\Phi^4\frac{(a+b)}{10}\ \Phi+42\Phi^5\frac{(a+b)}{10}\ \Phi.
\end{eqnarray*}
Thus
\begin{equation}\label{eq21}
\frac{\hat E'(1)}{\hat E(1)}= \frac{29\sqrt{5}-95}{10}(a+b).
\end{equation}
Plugging (\ref{eq21}) into (\ref{hatd}) yields
\begin{equation}\label{valued+}
\hat d_{a,b}=\frac{371-113\sqrt{5}}{40}\ a+\frac{361-123\sqrt{5}}{40}\ b.
\end{equation}


For $\hat h'_{a,b}(1)$, we derive a formula for $\hat h'_{a,b}(w)$ in terms of $e_1(w)$ by using (\ref{e'}). Then by $e_1(1)=\Phi$ we get
 \begin{equation}\label{hab}
\hat h'_{a,b}(1)=\frac{\sqrt{5}-1}{200}\left[10\left (a^2+b^2\right)-\frac{20-\sqrt{5}}{5}(a+b)^2\right]
\end{equation}
We verify that it is nonzero (details can be found in Appendix \ref{ahab}), thus similarly to the proof of Theorem \ref{Gauss}, we have $a\mathcal{K}_n+b\mathcal{L}_n$ converges to a Gaussian as $n\rightarrow \infty$.
\end{proof}

Applying Theorem \ref{thm:aK+bL} to the special cases $(a,b)=(1,0)$ and $(0,1)$, we obtain the following results.
\begin{theorem}\label{thm:KL}
The expected values and variances of $K_n$ and $L_n$ are
\begin{eqnarray*}
\mathbb{E}[K_n]=\frac{1}{10}n+\frac{371-113\sqrt{5}}{40}+o(\hat \gamma^n_{1,0}),& &
{\rm{var}}(K_n)=\frac{29\sqrt{5}-25}{1000}n+O(1),\\
\mathbb{E}(L_n)=\frac{1}{10}n+\frac{361-123\sqrt{5}}{40}+o(\hat \gamma^n_{0,1}),& &
{\rm{var}}(L_n)=\frac{15+21\sqrt{5}}{1000}n+O(1).
\end{eqnarray*}
Additionally, we have
\begin{eqnarray*}
\mathbb{E}[K_n]-\mathbb{E}[L_n]=\frac{1+\sqrt{5}}{4}+o(\hat \gamma'^n)=\frac{\varphi}{2}+o(\hat \gamma'^n)\approx 0.809016994\ {\rm{for\ some}}\ \hat \gamma'\in (0,1).
\end{eqnarray*}
In words, on average there are approximately 0.809 more positive terms than negative terms in the far-difference representation.
 \end{theorem}

Applying Theorem \ref{thm:aK+bL} to $a=b=1$, we get
\begin{equation}\label{vark+l}
{\rm{var}}(\mathcal{K}_n+\mathcal{L}_n)=\frac{2\sqrt{5}}{125}n+O(1),\ {\rm{and}}\ {\rm{var}}(\mathcal{K}_n-\mathcal{L}_n)=\frac{\sqrt{5}-1}{10}n+O(1).
\end{equation}
Hence
\begin{eqnarray}\label{Cov}
\nonumber {\rm{cov}}(\mathcal{K}_n,\mathcal{L}_n) &=&\frac{{\rm{var}}(\mathcal{K}_n+\mathcal{L}_n) -{\rm{var}}(\mathcal{K}_n-\mathcal{L}_n)}{4}\\ \nonumber &=&\frac{25-21\sqrt{5}}{1000}n+O(1)\approx -0.0219574275n+O(1).
\end{eqnarray}
With Theorem \ref{thm:KL} and (\ref{Cov}), we compute the correlation between $\mathcal{K}_n$ and $\mathcal{L}_n$:
\begin{eqnarray*}
\nonumber {\rm{corr}}(\mathcal{K}_n,\mathcal{L}_n) &=&\frac{{\rm{cov}}(\mathcal{K}_n,\mathcal{L}_n)}{\sqrt{{\rm{var}}(\mathcal{K}_n){\rm{var}}(\mathcal{L}_n)}}
=
\frac{\frac{25-21\sqrt{5}}{1000}n+O(1)}{\sqrt{\left(\frac{29\sqrt{5}-25}{1000}n+O(1)\right)\left(\frac{29\sqrt{5}-25}{1000}n+O(1)\right)}}\\
\nonumber &=&\frac{\frac{25-21\sqrt{5}}{1000}n+O(1)}{\frac{29\sqrt{5}-25}{1000}n+O(1)}=\frac{25-21\sqrt{5}}{29\sqrt{5}-25}+o(1)\\
\nonumber &=&\frac{10\sqrt{5}-121}{179} +o(1)\approx -0.551057655+o(1).
\end{eqnarray*}

Since ${\rm{var}}(\mathcal{K}_n)$ and ${\rm{var}}(\mathcal{L}_n)$ are of size $n$ and have the same coefficients of $n$, we have
\begin{eqnarray*}
\nonumber & &{\rm{cov}}(\mathcal{K}_n+\mathcal{L}_n,\mathcal{K}_n-\mathcal{L}_n) \\
\nonumber &=&E\left[\left(\mathcal{K}_n-E[\mathcal{K}_n]+(\mathcal{L}_n-E[\mathcal{L}_n])\right)\left(\mathcal{K}_n-E[\mathcal{K}_n]-(\mathcal{L}_n-E[\mathcal{L}_n])\right)\right]\\
\nonumber &=&E[(\mathcal{K}_n-E[\mathcal{K}_n])^2-(l-E[\mathcal{L}_n])^2]={\rm{var}}(\mathcal{K}_n)-{\rm{var}}(\mathcal{L}_n)\\
&=&O(1).
\end{eqnarray*}
Further, we have the values of ${\rm{var}}(\mathcal{K}_n+\mathcal{L}_n)$ and ${\rm{var}}(\mathcal{K}_n-\mathcal{L}_n)$ from (\ref{vark+l}) and (\ref{vark+l}), thus
\begin{eqnarray*}
\nonumber {\rm{corr}}(\mathcal{K}_n+\mathcal{L}_n,\mathcal{K}_n
-\mathcal{L}_n) &=&\frac{{\rm{cov}}(\mathcal{K}_n+\mathcal{L}_n,\mathcal{K}_n-\mathcal{L}_n)}{\sqrt{{\rm{var}}(\mathcal{K}_n+\mathcal{L}_n){\rm{var}}(\mathcal{K}_n-\mathcal{L}_n)}}\\
\nonumber &=&\frac{O(1)}{\sqrt{\left(\frac{2\sqrt{5}}{125}n+O(1)\right)\left(\frac{\sqrt{5}-1}{10}n+O(1)\right)}}\\
&=&o(1).
\end{eqnarray*}
Since $\mathcal{K}_n$ and $\mathcal{L}_n$ are bivariate Gaussian, $\mathcal{K}_n+\mathcal{L}_n$ and $\mathcal{K}_n-\mathcal{L}_n$ are independent as $n\rightarrow \infty$.


\section{Conclusion and Future Research}\label{sec:conclusion}

Our combinatorial perspective has extended previous work, allowing us to prove Gaussian behavior for the number of summands for a large class of expansions in terms of solutions to linear recurrence relations. This is just the first of many questions one can ask. Others, which we hope to return to at a later date, include:

\begin{enumerate}
\item Are there similar results for linearly recursive sequences with arbitrary integer coefficients (i.e., negative coefficients are allowed in the defining relation, which is different than allowing negative summands)?

\item What happens if we consider sequences where either uniqueness of representation fails, or some numbers are not representable? In particular, what is true for a `generic' number?

\item Lekkerkerker's theorem, and the Gaussian extension, are for the behavior in intervals $[F_n, F_{n+1})$.
Do the limits exist if we consider other intervals, say
$[F_n+g_1(F_n), F_n + g_2(F_n))$ for some functions $g_1$ and $g_2$? If yes, what must be
true about the growth rates of $g_1$ and $g_2$?

\item For the generalized recurrence relations, what happens if instead of looking at $\sum_{i=1}^n a_i$ we study $\sum_{i=1}^n \min(1,a_i)$? In other words, we only care about how many distinct $H_i$'s occur in the decomposition.

\item What can we say about the distribution of the largest gap between summands in generalized Zeckendorf decomposition? Appropriately normalized, how does the distribution of gaps between the summands behave?

\end{enumerate}

The last question has been solved in some cases by Beckwith and Miller \cite{BM}. They prove

\begin{thm}[Base $B$ Gap Distribution]
For base $B$ decompositions, as $n\to\infty$ the probability of a gap of length 0 between summands for numbers in $[B^n, B^{n+1})$ tends to $\frac{(B-1)(B-2)}{B^2}$, and for gaps of length $k \geq 1$ to $\frac{(B-1)(3B-2)}{B^2} B^{-k}$.
\end{thm}

Note if $B \ge 3$ the density is a sum of a point mass at the origin and a geometric random variable.

\begin{thm}[Zeckendorf Gap Distribution]
For Zeckendorf decompositions, for integers in $[F_n, F_{n+1})$ the probability of a gap of length $k \ge 2$ tends to $\frac{\varphi(\varphi-1)}{\varphi^{k}}$ for $k \ge 2$, with $\varphi = \frac{1+ \sqrt{5}}{2}$ the golden mean.
\end{thm}



\appendix



\section{No Multiple Roots for $x\in I_{\epsilon}$}\label{amult}
Assume that $L\ge 2$. We first show that there exists $x>0$ such that $A(y)$ has no multiple roots.

\begin{lemma}\label{norm}
For any $n\ge 1$ and positive real numbers $a_0\le a_1\le \cdots \le a_n$ but not all equal, any root $z$ of $P(x)=a_0+a_1 x+\cdots + a_n x^n$ satisfies $|z|<1$.
\end{lemma}

\begin{proof}
Let $z$ be a root of $P(x)$, then $z$ is also a root of $(1-x)P(x)$. Thus
\begin{equation*}
a_0+(a_1-a_0)z+(a_2-a_1)z^2+\cdots +(a_n-a_{n-1})z^n-a_n z^n=0.
\end{equation*}
If $|z|\ge 1$, then we get
\begin{align*}
|a_n z^n|&= |a_0+(a_1-a_0)z+(a_2-a_1)z^2+\cdots +(a_n-a_{n-1})z^n|\\
&\le|a_0|+|(a_1-a_0)z|+|(a_2-a_1)z^2|+\cdots +|(a_n-a_{n-1})z^n|\\
&=a_0+(a_1-a_0)|z|+(a_2-a_1)|z|^2+\cdots +(a_n-a_{n-1})|z|^n\\
&\le a_0+(a_1-a_0)|z|^n+(a_2-a_1)|z|^n+\cdots +(a_n-a_{n-1})|z|^n\\
&=a_n |z|^n = |a_n z^n|.
\end{align*}
Hence all of the equalities are achieved, i.e., $|z|=1$ and $(a_1-a_0)z$, $(a_2-a_1)z^2$, $\dots$, $(a_n-a_{n-1})z^n$ are real and nonnegative since $a_0$ is real and positive.

Since the $a_i$'s are not all equal, there exists an $i$ such that $a_{i+1}>a_i$. As $(a_{i+1}-a_{i})z^{i+1})z$ is real and nonnegative, so is $z$. Therefore, $P(z)=a_0+a_1 z+\cdots + a_n z^n\ge a_0>0$, contradiction.
\end{proof}

\begin{lemma}\label{norm2}
Let $f_0(x)=1-x-x^2-\cdots-x^n$ with $n\ge 2$, then
{\rm{(a)}} $f_0(x)$ has a unique positive real root $r_0$, $0<r_0<1$ and $r_0$ is not a multiple root of $f_0(x)$.
{\rm{(b)}} Any root $z\neq r_0$ of $f_0(x)$ satisfies $|z|>1$.
\end{lemma}

\begin{proof}
(a) Since $f_0(x)$ is decreasing on $(0,\infty)$ and $f(0)=1>0>f(1)$, $Q(x)$ has a unique positive real root $r$ and $0<r<1$.

Since $f_0'(x)=-1-2x-\cdots-nx^{n-1}$ and $r>0$, $f_0'(r)<0$. Therefore $r$ is not a multiple root of $f_0(x)$.

(b) Note that $f_0(0)\neq 0$, thus $0$ is not a root of $f_0(x)$. Let
\begin{equation*}
f(x)=x^nf_0 \left(\frac{1}{x}\right)=x^n-x^{n-1}-\cdots-x-1,
\end{equation*}
then it suffices to show that any root $z\neq r$ of $f(x)$ satisfies $|z|<1$ where $r=1/r_0$.

Since $r$ is a root of $f(x)$, $f(x)$ can be factored as
\begin{eqnarray}\label{f}
f(x)&=&(x-r)(d_0x^{n-1}+d_1 x^{n-2}+\cdots +d_{n-2} x+d_{n-1})\\
\nonumber &=&x^n+\sum_{i=1}^{n-1}(d_i-rd_{i-1})x_{n-i}-rd_{n-1},
\end{eqnarray}
where $d_0=1$.

Comparing the coefficients of $x_{n-i}$ of both sides, we get $d_i-rd_{i-1}=-1$, i.e.,
\begin{equation}\label{eq8}
d_i=rd_{i-1}-1,\ 1\le i \le n-1.
\end{equation}
Using $d_0=1$ and applying (\ref{eq8}) repeatedly, we get
\begin{equation*}
d_i=r^i-r^{i-1}-r^{i-2}-\cdots -1,\ 1\le i \le n-1.
\end{equation*}
Since $f(r)=0$, for $1\le i\le n-1$,
\begin{equation*}
d_i=r^i-r^{i-1}-r^{i-2}-\cdots -1=\frac{1}{r^{n-i}}(r^{n-i-1}+r^{n-i-2}+\cdots +1)>0,
\end{equation*}
and for $1\le i\le n-2$,
\begin{eqnarray*}
d_i>\frac{1}{r^{n-i}}(r^{n-i-1}+r^{n-i-2}+\cdots + r)=\frac{1}{r^{n-i-1}}(r^{n-i-2}+r^{n-i-3}+\cdots + 1)=d_{i+1}.
\end{eqnarray*}
Hence $d_1>d_2>\cdots >d_{n-1}>0$.

Since $f_0(r)=0$, we have
\begin{equation*}
r^n=r^{n-1}+r^{n-2} +\cdots +1=\frac{r^n-1}{r-1},
\end{equation*}
which yields
\begin{equation*}
r^n (r-1)\le (r^n-1)<r^n.
\end{equation*}
Hence $r-1<1$ and therefore $d_1=r-1<1=d_0$.

Let $P(x)=d_0x^{n-1}+d_1 x^{n-2}+\cdots +d_{n-2} x+d_{n-1}$, then $f(x)=(x-r)P(x)$ (see (\ref{f})). Applying Lemma \ref{norm} to $P(x)$, we see that $|z|<1$ for any root $z$ of $P(x)$, i.e., any root $z$ of $f(x)$ such that $z\neq r$.
\end{proof}

\begin{lemma}\label{coprime}
Let $Q(x)=A(1)=1-x-\cdots - x^{s_L-1}$ and $$R(x)\ = \ A'(1) \ = \ -\sum_{m=0}^{L-1}\sum_{j=s_m}^{s_{m+1}-1}(m+1)x^j,$$ then $R(x)$ and $Q(x)$ are coprime (see (\ref{Ayy}) for the definition of $A(y)$).
\end{lemma}

\begin{proof}
Let $n=s_L-1\ge c_1+c_L-1\ge 1$. If $n=1$, then $c_1=c_L=1$ and the other $c_i$'s are zero. Thus $Q(x)=-x$ and $R(x)=-1-Lx$ are coprime.

Assume that $n\ge 2$. We prove by contradiction. Assume that $Q(x)$ and $R(x)$ are not coprime. Let $D(x)=\sum_{i=0}^{l}a_ix^i$ be a greatest common divisor of $Q(x)$ and $R(x)$ with $l,a_l>0$. Let $Q(x)=D(x)Q_1(x)$, where $Q_1(x)=\sum_{j=0}^{t}b_jx^j\in \mathbb{Z}[x]$. Noting that the leading coefficient and the constant term of $Q(x)$ are -1 and 1, respectively, we get $a_l=1$, $b_t=-1$ and $a_0=b_0\in \{\pm 1\}$.

Let the $z_i$'s be the roots of $D(x)$; they are also roots of $Q(x)$ and $R(x)$. Applying Lemma \ref{norm} to $R(x)$, we see that any root of $R(x)$ has norm smaller than 1. Hence we have $|z_i|<1$ for all $i$. On the other hand, by Lemma \ref{norm2} to $Q(x)$, any root of $Q(x)$ except one (the unique positive root) has norm greater than 1.   Therefore $D(x)$ only has one root $z_1$, which is the unique positive root of $Q(x)$. This implies that $D(x)$ is of degree 1. Since $Q(x)$ is of degree $n\ge 2$ and $Q(x)=D(x)Q_1(x)$, $Q_1(x)$ is of degree at least 1. Since any root other than $z_1(x)$ of $Q(x)$ is a root of $Q_1(x)$ and thus has norm greater than 1, the norm of the product of roots of $Q_1(x)$ should be greater than 1; however by Vieta's Formula, the norm of the product is $|b_0/b_t|=1$, contradiction.
\end{proof}

\begin{lemma}\label{nomulti}
There are only finitely many $x>0$ such that $A(y)$ has multiple roots. As a consequence, there exists $\epsilon \in (0,1)$ such that for any $x\in I_{\epsilon}$, $A(y)$ has no multiple roots.
\end{lemma}

\begin{proof}
If $x>0$, then $A(y)$ is of degree $s_L-1$ in terms of $y$. We proved in Lemma \ref{coprime} that
\begin{equation*}
A(1)=\mathscr{A}(x,1)\ {\rm{and}}\ A'(1)=\left.\frac{d}{dy} \mathscr{A}(x,y)\right|_{y=1}
\end{equation*}
are coprime, hence $\mathscr{A} (x,y)$ and $\frac{d}{dy} \mathscr{A}(x,y)$ are coprime (see (\ref{A}) for the definition of $\mathscr{A} (x,y)$).

Now, we regard $\mathscr{A} (x,y)$ and $\frac{d}{dy} \mathscr{A}(x,y)$ as polynomials $A(y)$ and $A'(y)$ of $y$ with coefficients polynomials of $x$. We use the Euclidean algorithm to compute the great common divisor of $A(y)$ and $A'(y)$. In each step, the quotient and remainder are (fractional) polynomials of $x$. If we get a fractional polynomial, there are finitely many $x$'s such that the denominator is zero. We exclude these values from the current \emph{admissible} set of $x$ and continue (the admissible set was $\{x>0\}$ at the beginning).

Since $A(y)$ and $A'(y)$ are coprime, the Euclidean algorithm terminates in a constant polynomial in terms of $y$ (if not then we would have found a common divisor of $A(y)$ and $A'(y)$ of degree at least 1 in $y$ and coefficients fractional polynomials in $x$).

We exclude from the current admissible set the roots of the numerator and the denominator of this fractional polynomial.

In the above procedure, at each step we exclude finitely many values from the current admissible set. Since there are at most $s_L$ steps, we exclude finitely many values in total. For any $x$ in the last admissible set, $A(y)$ has no multiple roots. Hence there are finitely many $x\in \mathbb{R}$ such that $A(y)$ has multiple roots.
\end{proof}


\section{Differentiability Results}

\subsection{Differentiability of the Roots}\label{adiff}

\begin{proof}[Proof of Proposition \ref{diff}]
For fixed positive $x$ and a small increment $\Delta x>0$, letting $z_i(x)=y_i(x+\Delta x)$ $(1\le i \le L)$, we have
\begin{equation}\label{x}
1-\sum_{m=0}^{L-1} \sum_{j=s_m}^{s_{m+1}-1} x^j y_i^{m+1}(x)=0
\end{equation}
and
\begin{equation}\label{dx}
1-\sum_{m=0}^{L-1} \sum_{j=s_m}^{s_{m+1}-1} (x+\Delta x)^j z_i^{m+1}(x)=0.
\end{equation}
Subtracting (\ref{dx}) from (\ref{x}), we get
\begin{equation*}
\sum_{m=0}^{L-1} \sum_{j=s_m}^{s_{m+1}-1} \left((x+\Delta x)^j z_i^{m+1} -x^j y_i^{m+1}(x)\right)=0.
\end{equation*}
The left-hand side can be written as
\begin{equation*}
\sum_{m=0}^{L-1} \sum_{j=s_m}^{s_{m+1}-1} \left(z_i^{m+1}(x)\left((x+\Delta x)^j -x^j\right) + x^j \left(z_i^{m+1}(x)-y_i^{m+1}(x)\right)\right),
\end{equation*}
thus
\begin{eqnarray}\label{eq2}
 \sum_{m=0}^{L-1} \sum_{j=s_m}^{s_{m+1}-1} x^j \left(z_i^{m+1}(x)-y_i^{m+1}(x)\right)
=-\sum_{m=0}^{L-1} \sum_{j=s'_m}^{s'_{m+1}-1}z_i^{m+1}(x)\left((x+\Delta x)^j -x^j\right).
\end{eqnarray}

Since
\begin{equation*}z_i^{m+1}(x)-y_i^{m+1}(x)=(z_i(x)-y_i(x))\sum_{l=0}^{m} z_i^l(x)y_i^{m-l}(x)
 \end{equation*}
 and
 \begin{equation*}
 (x+\Delta x)^j -x^j=\Delta x \sum_{t=0}^{j-1} (x+\Delta x)^t x^{j-1-t},
  \end{equation*}
  (\ref{eq2}) can be written as
  \begin{eqnarray}\label{eq3}
& & (z_i(x)-y_i(x))\sum_{m=0}^{L-1} \sum_{j=s_m}^{s_{m+1}-1} x^j \sum_{l=0}^{m} z_i^l(x)y_i^{m-l}(x)\nonumber\\
&=& \ \ \ -\Delta x \sum_{m=0}^{L-1} \sum_{j=s'_m}^{s'_{m+1}-1}z_i^{m+1}(x)\sum_{t=0}^{j-1} (x+\Delta x)^t x^{j-1-t}.
\end{eqnarray}
The coefficient of $z_i(x)-y_i(x)$ on the left-hand side is
 \begin{equation}\label{left}
\sum_{m=0}^{L-1} \sum_{j=s_m}^{s_{m+1}-1} x^j \sum_{l=0}^{m} z_i^l(x)y_i^{m-l}(x),
  \end{equation}
which is nonzero for all but finitely many $z_i(x)$ (to see this, regard (\ref{left}) as a polynomial of $z_i(x)$) and hence nonzero for all but finitely many $\Delta x$ (regard (\ref{dx}) as polynomial of $\Delta x$). Therefore, there exists $\widetilde{\epsilon}>0$ such that for any $\Delta x\in (0,\widetilde{\epsilon})$,  (\ref{left}) is not zero. Thus  we can write (\ref{eq3}) as
 \begin{equation}\label{eq4}
\frac{z_i(x)-y_i(x)}{\Delta x}=-\frac{\sum_{m=0}^{L-1} \sum_{j=s'_m}^{s'_{m+1}-1}z_i^{m+1}(x)\sum_{t=0}^{j-1} (x+\Delta x)^t x^{j-1-t}}{\sum_{m=0}^{L-1} \sum_{j=s_m}^{s_{m+1}-1} x^j \sum_{l=0}^{m} z_i^l(x)y_i^{m-l}(x)}.
  \end{equation}

The differentiability of $y_i(x)$ follows from showing the limit of the right-hand side of (\ref{eq4}) exists. Recall that $y_i(x)$ is continuous, so it suffices to verify that the denominator of the limit of (\ref{eq4}) as $\Delta x \rightarrow 0$ is nonzero.

The limit of the denominator is
 \begin{eqnarray*}
\mathscr{R}_i(x)&:=&\sum_{m=0}^{L-1} \sum_{j=s_m}^{s_{m+1}-1} x^j \sum_{l=0}^{m} y_i^l(x)y_i^{m-l}(x)=\sum_{m=0}^{L-1} \sum_{j=s_m}^{s_{m+1}-1} (m+1)x^j  y_i^m(x)\\
&=&-A'(y_i(x)),
 \end{eqnarray*}
which is not zero as $y_i(x)$ is not a multiple root of $A(y)$. Since $y_i(x)$ is continuous, $\mathscr{R}_i(x)$ is continuous. Thus there exists $\epsilon\in (0,\widetilde{\epsilon})$ such that for any $\widetilde{x}\in (x-\epsilon,x+\epsilon)$,
\begin{equation}\label{non0}
\mathscr{R}_i(\widetilde{x})\neq 0.
\end{equation}
As the denominator is non-zero in $(x-\epsilon,x+\epsilon)$, we can take the limits of both sides of (\ref{eq4}), yielding
\begin{equation}
y'_i(x)=-\frac{\sum_{m=0}^{L-1} \sum_{j=s'_m}^{s'_{m+1}-1}jy_i^{m+1}(x)x^{j-1}}{\sum_{m=0}^{L-1} \sum_{j=s_m}^{s_{m+1}-1} (m+1)x^j  y_i^m(x)}.
\end{equation}

By repeated use of the quotient rule and the differentiability of $y_i(x)$, we see that $y_i^{(\ell)} (x)$ exists, and further is of the form
\begin{equation}\label{lder}
y_i^{(\ell)} (x)=\frac{\mathscr{P}_{\ell}(y_i(x))}{\mathscr{Q}^{2{\ell}-1}(y_i(x))},
\end{equation}
where $\mathscr{P}_{\ell}$ and $\mathscr{Q}$ are polynomials with coefficients polynomials of $x$, and
 \begin{equation*}
\mathscr{Q}(y_i(x))=\sum_{m=0}^{L-1} \sum_{j=s_m}^{s_{m+1}-1} (m+1)x^j  y_i^m(x)=\mathscr{R}_i(x).
\end{equation*}
Note that $\mathscr{Q}(y_i(x))=\mathscr{R}_i(x)\neq 0$ by (\ref{non0}).
\end{proof}


\subsection{Differentiability of the $\alpha_i(x)$'s and the $q_i(x)$'s}\label{acla1}

\begin{proof}[Proof of Claim \ref{cla1}]
For any $\ell \ge 1$, by Proposition \ref{diff} there is an $\varepsilon > 0$ such that each $y_i(x)$ $(i>1)$ is $\ell$-times differentiable at any $x\in I_{\varepsilon}$ $=$ $(1-\varepsilon,1+\varepsilon)\backslash \{1\}$ and $y_1(x)$ is $\ell$-times differentiable at any $x\in (1-\varepsilon,1+\varepsilon)$. Further, $y_i(x)\neq 0$ for any $i$ and $x>0$ as $A(0)=1\neq 0$ (see (\ref{Ayy}) for the definition of $A(y)$), thus for each $i>1$ we have $\alpha_i(x)=(y_i(x))^{-1}$ is $\ell$-times differentiable at any $x\in I_{\varepsilon}$ and $\alpha_1(x)=(y_1(x))^{-1}$ is $\ell$-times differentiable at any $x\in (1-\varepsilon,1+\varepsilon)$.

By Definition (\ref{qi}), the denominator and the numerator of $q_i(x)$ are
\begin{equation*}
{\sum_{j=s_{L-1}+1}^{s_L}x^j\prod_{j\neq i}\left(y_j(x)-y_i(x)\right)},\ \ \ {\rm and}\ \ \ \sum_{m=1}^{L} b_m(x) y^{m}_i(x),
\end{equation*}
which are $\ell$-times differentiable at $x\in I_{\varepsilon}$ since each $y_j(x)$ is $\ell$-times differentiable at $x\in I_{\varepsilon}$. (Recall from Definitions (\ref{B}) and (\ref{By}) that the $b_m(x)$'s are polynomials of $x$.) Further, since the denominator is nonzero when $x\in I_{\varepsilon}$, $q_i(x)$ is $\ell$-times differentiable at $x\in I_{\varepsilon}$.

Let
\begin{equation}\label{eq32}
E_i(x)=\prod_{j\neq i}\left(y_j(x)-y_i(x)\right).
\end{equation}
Then the denominator of $q_1(x)$ is $x^{s_L}y_1(x)E_1(x)$, which is nonzero when $x=1$. Since $\sum_{j=s_{L-1}+1}^{s_L}x^j$ and $y_1(x)$ are $\ell$-times differentiable at 1, it suffices to show that $E_1(x)$ is $\ell$-times differentiable at 1.
Letting $y=y'+y_1(x)$ in (\ref{Ayy}), we get
\begin{equation*}
A(y)=1-\sum_{m=0}^{L-1} \sum_{j=s_m}^{s_{m+1}-1} x^j (y'+y_1(x))^{m+1}.
\end{equation*}
On the other hand, we have
\begin{eqnarray}\label{eq33}
\nonumber A(y)&=&-\sum_{j=s_{L-1}+1}^{s_L}x^j\prod_{j=1}^L\left(y-y_j(x)\right)\\ \nonumber &=&-\sum_{j=s_{L-1}+1}^{s_L}x^j\prod_{j=1}^L\left(y'+y_1(x)-y_j(x)\right)\\
&=&-\sum_{j=s_{L-1}+1}^{s_L}x^jy'\prod_{j\neq 1}\left(y'+y_1(x)-y_j(x)\right).
\end{eqnarray}
Comparing the coefficients of $y$ in (\ref{eq32}) and (\ref{eq33}) yields
\begin{equation*}
-\sum_{m=0}^{L-1} \sum_{j=s_m}^{s_{m+1}-1} x^j (m+1)y^{m}_1(x)=-\sum_{j=s_{L-1}+1}^{s_L}x^j(-1)^{L-1}E_1(x).
\end{equation*}
Hence
\begin{equation}\label{EE}
E_1(x)=\frac{\sum_{m=0}^{L-1} \sum_{j=s_m}^{s_{m+1}-1} x^j (m+1)y^{m}_1(x)}{\sum_{j=s_{L-1}+1}^{s_L}x^j(-1)^{L-1}}.
\end{equation}
Since $y_1(x)$ is $\ell$-times differentiable at 1 and the denominator in (\ref{EE}) is nonzero as $x>0$, $E_1(x)$ is $\ell$-times differentiable at 1.
\end{proof}


\section{Main Term of $g^{(\ell)}(x)$}\label{amain}

\begin{proof}[Proof of Claim \ref{claimo1}]
We first give an outline of the proof before jumping into the details.

We proceed by expressing $q_i(x)$ in terms of the $\alpha_i(x)$'s and showing that
\begin{equation}\label{eqsum}
\sum_{i=2}^{L}xq_i(x) \alpha_i^n(x)=-\sum_{m=1}^{L} b_m(x)\sum_{i=2}^{L}\frac{\alpha_i^{n-L+2-m}(x)}{\prod_{j\neq i}(\alpha_j(x)-\alpha_i(x))}.
\end{equation}
Then it reduces to proving that
\begin{equation}\label{o1}
\frac{d^{\ell}}{dx^{\ell}}\sum_{i=2}^{L}\frac{\alpha_i^{n}(x)}{\prod_{j\neq i}(\alpha_j(x)-\alpha_i(x))}=o(\gamma^n_{\ell})\alpha_1^n(x)
\end{equation}
for some $\gamma_{\ell}\in (0,1)$. In fact, if this is true, then we can replace $n$ by $n-L+2-m$ for $1\le m\le L$. Since the $b_m(x)$'s are bounded on $I_{\varepsilon}$ and $L$ is fixed, it follows from (\ref{eqsum}) that $\sum_{i=2}^{L} xq_i(x)\alpha^n_i(x)$ is of the form $o(\gamma^n_{\ell})\alpha^n_1(x)$).

Let
\begin{equation*}
\mathcal{P}(x)=\sum_{i=2}^{L}\frac{\alpha_i^{n}(x)}{\prod_{j\neq i}(\alpha_j(x)-\alpha_i(x))}.
\end{equation*}
We show that $\mathcal{P}^{(\ell)}(x)$ can be written as a fraction satisfying the following and then Claim \ref{claimo1} follows from (\ref{alpha}).

\begin{enumerate}
\item The numerator is of form $\sum_{i} \mathcal{P}_{i,\ell}(x)\prod_{j=2}^{L} \alpha^{i_j}_j(x)$, where there are at most $O(n^{N_\ell})$ summands and $\sum_{j=2}^{L}i_j\le n+M_{\ell}$ with $N_{\ell}$ and $M_{\ell}$ independent of $n$ and the ${\mathcal{P}}_{i,\ell}(x)$'s are polynomials (independent of $n$) of $\alpha_1(x),\dots,\alpha_L(x), \mathcal{E}^{(l)}_1(x)$ $(1\le l\le \ell)$ and $x$.

\item The denominator is a function of $x$ such that it is well-defined and bounded and nonzero on $(1-\varepsilon,1+\varepsilon)$.
\end{enumerate}

\vspace{0.1in}

Now we prove (\ref{eqsum}). From Definition \ref{qi}, (\ref{eq32}) and $\alpha_i(x)=1/y_i(x)$, we get
\begin{equation*}
q_i(x)=\frac{\sum_{m=1}^{L} b_m(x) y^{m}_i(x)}{\sum_{j=s_{L-1}+1}^{s_L}x^jE_i(x)}=\sum_{m=1}^{L} \frac{b_m(x) }{\sum_{j=s_{L-1}+1}^{s_L}x^jE_i(x)\alpha^{m}_i(x)},
\end{equation*}
where
\begin{eqnarray}\label{eq51}
\nonumber E_i(x)&=&\prod_{j\neq i}(y_j(x)-y_i(x))=\prod_{j\neq i}\left[\frac{1}{\alpha_j(x)}-\frac{1}{\alpha_i(x)}\right]\\
\nonumber &=&\frac{\prod_{j\neq i}(\alpha_i(x)-\alpha_j(x))}{\alpha^{L-1}_i(x)\prod_{j\neq i}\alpha_j(x)}=\frac{\prod_{j\neq i}(\alpha_i(x)-\alpha_j(x))}{\alpha^{L-2}_i(x)\prod_{j=1}^{L}\alpha_j(x)}\\
\nonumber &=&\frac{(-1)^{L-1}\prod_{j\neq i}(\alpha_j(x)-\alpha_i(x))}{\alpha^{L-2}_i(x)(-1)^L\sum_{j=s_{L-1}}^{s_L-1}x^j}\\
&=&-\frac{\prod_{j\neq i}(\alpha_j(x)-\alpha_i(x))}{\alpha^{L-2}_i(x)\sum_{j=s_{L-1}}^{s_L-1}x^j}
\end{eqnarray}
by Vieta's Formula (relating the coefficients of a polynomial to its roots). Thus
\begin{equation*}
q_i(x)=-\sum_{m=1}^{L} \frac{b_m(x)}{x\alpha^{L-2+m}_i(x)}\prod_{j\neq i}\frac{1}{\alpha_j(x)-\alpha_i(x)},
\end{equation*}
and  (\ref{eqsum}) follows.

Next we look at the $\mathcal{P}^{(\ell)}(x)$'s. Note that $\mathcal{P}$ is a symmetric function of $\alpha_2(x),\dots,\alpha_L(x)$. For $1<i_0<j_0$, we have
\begin{eqnarray*}
\nonumber & & (\alpha_{i_0}(x)-\alpha_{j_0}(x))\mathcal{P}(x)\\
\nonumber  &=&\sum_{i\neq1, i_0,j_0}\frac{\alpha_i^{n}(x)(\alpha_{i_0}(x)-\alpha_{j_0}(x))}{\prod_{j\neq i}(\alpha_j(x)-\alpha_i(x))} -\frac{\alpha_{i_0}^{n}(x)}{\prod_{j\neq i_0,j_0}(\alpha_j(x)-\alpha_{i_0}(x))}\\
& & + \frac{\alpha_{j_0}^{n}(x)}{\prod_{j\neq i_0,j_0}(\alpha_j(x)-\alpha_{j_0}(x))},
\end{eqnarray*}
which equals zero if $\alpha_{i_0}(x)=\alpha_{j_0}(x)$. Hence the polynomial
\begin{equation}\label{poly}
\prod_{1\le i<j\le L}(\alpha_j(x)-\alpha_i(x))\mathcal{P}(x)
\end{equation}
of $\alpha_1(x),\dots,\alpha_L(x)$ is divisible by $\alpha_{i_0}(x)-\alpha_{j_0}(x)$ for any $1<i_0<j_0$. Therefore
\begin{equation}\label{poly1}
\prod_{j\neq 1}(\alpha_j(x)-\alpha_1(x))\mathcal{P}(x)
\end{equation}
is a polynomial of $\alpha_1(x),\dots,\alpha_L(x)$.

Since (\ref{poly}) is homogeneous of order $n-(L-1)+\frac{1}{2}(L-1)L$, the polynomial in (\ref{poly1}) is homogeneous of order $n-(L-1)+\frac{1}{2}(L-1)L-\frac{1}{2}(L-2)(L-1)=n$. Furthermore, note that (\ref{poly}) is a sum of $O(1)$ terms with each summand a product of $\alpha^n_i(x)$ $(i>1)$ and a polynomial of $\alpha_1(x),\dots,\alpha_L(x)$ independent of $n$. We can divide the summands into $O(1)$ pairs with each pair of the form $\tilde {\mathcal{P}}(x)(\alpha^{l}_{i_0}(x)-\alpha^{l}_{j_0}(x))$ where $\tilde {\mathcal{P}}(x)$ is a polynomial of $\alpha_1(x),\dots,\alpha_L(x)$ independent of $n$ and $l\le n$. Dividing each pair by $\alpha^{l}_{i_0}(x)-\alpha^{l}_{j_0}(x)$, we get
\begin{equation*}
\frac{\tilde {\mathcal{P}}(x)(\alpha^{l}_{i_0}(x)-\alpha^{l}_{j_0}(x))}{\alpha_{i_0}(x)-\alpha_{j_0}(x)}=\tilde {\mathcal{P}}(x)\sum_{t=0}^{l} \alpha^t_{i_0}(x)\alpha^{l-t}_{j_0}(x),
\end{equation*}
which is a sum of $O(n)$ terms with each summand a product of at most $n$ element (with multiplicity) from $\{\alpha_i(x)\}_{i>1}$ and a polynomial of $\alpha_1(x),\dots,\alpha_L(x)$ independent of $n$, hence dividing (\ref{poly}) by $\alpha_{i_0}(x)-\alpha_{j_0}(x)$ yields a sum of $O(n)$ terms with each summand a product of at most $n$ element (with multiplicity) from $\{\alpha_i(x)\}_{i>1}$ and a polynomial of $\alpha_1(x),\dots,\alpha_L(x)$ independent of $n$.

Repeating this procedure, namely dividing (\ref{poly}) by $\alpha_{i_0}(x)-\alpha_{j_0}(x)$ for all $1<i_0<j_0$,w e get a sum of $O(n^{N_0})$ terms with each term a product of at most $n$ element (with multiplicity) from $\{\alpha_i(x)\}_{i>1}$ and a polynomial of $\alpha_1(x),\dots,\alpha_L(x)$ independent of $n$, where $N_0$ is determined by $L$ and independent of $n$, namely
\begin{equation}\label{calP}
\mathcal{P}(x) = \frac{\sum_{i} \mathcal{P}_{i,0}(x)\prod_{j=2}^{L} \alpha^{i_j}_j(x)}{\prod_{j\neq 1}(\alpha_j(x)-\alpha_1(x))},
\end{equation}
where $\sum_{j=2}^{L}i_j\le n$ and the $\mathcal{P}_i(x)$'s are polynomials of $\alpha_1(x),\dots,\alpha_L(x)$ independent of $n$. Since the denominator of $\mathcal{P}(x)$ is continuous, nonzero and well-defined at $x=1$, the claim in the case $\ell=0$ follows by Proposition \ref{proppar}.

Let
\begin{equation}\label{calE}
\mathcal{E}_i(x)=\prod_{j\neq i}(\alpha_j(x)-\alpha_i(x)).
\end{equation}
Plugging Definition (\ref{calE}) with $i=1$ into (\ref{calP}), we get
\begin{equation*}
\mathcal{P}(x) = \frac{1}{\mathcal{E}_1(x)}\sum_{i} \mathcal{P}_{i,0}(x)\prod_{j=2}^{L} \alpha^{i_j}_j(x).
\end{equation*}
Thus
\begin{equation}\label{calP'}
\mathcal{P}'(x)=\left[\frac{1}{\mathcal{E}_1(x)}\right]'\sum_{i} \mathcal{P}_{i,0}(x)\prod_{j=2}^{L} \alpha^{i_j}_j(x)+\frac{1}{\mathcal{E}_1(x)}\left[\sum_{i} \mathcal{P}_{i,0}(x)\prod_{j=2}^{L} \alpha^{i_j}_j(x)\right]'.
\end{equation}

By (\ref{eq51}), we get
\begin{equation*}
\mathcal{E}_i(x)=-\alpha_i^{L-2}(x)\sum_{j=s_{L-1}}^{s_L-1}x^jE_i(x).
\end{equation*}
Plugging in (\ref{EE}) with the index $1$ replaced by $i$ yields
\begin{equation*}
\mathcal{E}_i(x)=\frac{(-1)^{L}\alpha_i^{L-2}(x)}{x}\sum_{m=0}^{L-1} \sum_{j=s_m}^{s_{m+1}-1}(m+1) x^j y^{m}_i(x).
\end{equation*}
Since $\alpha_i(x)$ and $y_i(x)$ are $\ell'$-times differentiable at $x\in I_{\varepsilon}$ for all $i$ and at $x=1$ for $i=1$ for all $\ell'$, so is $\mathcal{E}_i(x)$.

Note from (\ref{yi'}) that
\begin{equation*}
\sum_{m=0}^{L-1} \sum_{j=s_m}^{s_{m+1}-1} (m+1)x^j  y_i^m(x)=-\frac{1}{y'_i(x)}{\sum_{m=0}^{L-1} \sum_{j=s'_m}^{s'_{m+1}-1}jy_i^{m+1}(x)x^{j-1}},
\end{equation*}
thus
\begin{eqnarray}\label{calEi}
\nonumber \mathcal{E}_i(x)&=&\frac{(-1)^{L-1}\alpha_i^{L-2}(x)}{xy'_i(x)} {\sum_{m=0}^{L-1} \sum_{j=s'_m}^{s'_{m+1}-1}jy_i^{m+1}(x)x^{j-1}}\\
\nonumber &=&\frac{(-1)^{L}\alpha_i^{L}(x)}{x\alpha'_i(x)} {\sum_{m=0}^{L-1} \sum_{j=s'_m}^{s'_{m+1}-1}j\alpha_i^{-m-1}(x)x^{j-1}}\\
&=&\frac{(-1)^{L}}{x\alpha'_i(x)} {\sum_{m=0}^{L-1} \sum_{j=s'_m}^{s'_{m+1}-1}j\alpha_i^{L-m-1}(x)x^{j-1}}.
\end{eqnarray}
Therefore
\begin{equation}\label{alpha'}
\alpha'_i(x)=\frac{(-1)^{L}}{x\mathcal{E}_i(x)} {\sum_{m=0}^{L-1} \sum_{j=s'_m}^{s'_{m+1}-1}j\alpha_i^{L-m-1}(x)x^{j-1}}.
\end{equation}

Note that $\left[\sum_{i} \mathcal{P}_{i,0}(x)\prod_{j=2}^{L} \alpha^{i_j}_j(x)\right]'$ is a sum of $O(n^{N'_1})$ terms with each summand a product of $\alpha'_t(x)\prod_{j=2}^{L} \alpha^{i_j}_j(x)$ and a polynomial of $\alpha_1(x),\dots,\alpha_L(x)$ independent of $n$, where $N'_1$ is also independent of $n$, $t>1$ and $\sum_{j=2}^{L}i_j\le n$. By (\ref{alpha'}), each summand is of the form
\begin{equation*}
\frac{(-1)^{L}}{x\mathcal{E}_t(x)} {\sum_{m=0}^{L-1} \sum_{j'=s'_m}^{s'_{m+1}-1}j'\alpha_t^{L-m-1}(x)x^{j'-1}}\prod_{j=2}^{L} \alpha^{i_j}_j(x).
\end{equation*}

Since $\mathcal{P}(x)$ is symmetric with respect to $\alpha_2(x),\alpha_3(x),\dots,\alpha_L(x)$, so is $\sum_{i} \mathcal{P}_{i,0}(x)\prod_{j=2}^{L} \alpha^{i_j}_j(x)$ and its derivative. Thus, by the same approach as in the case $\ell=0$, we can prove that
\begin{equation*}
\left[\sum_{i} \mathcal{P}_{i,0}(x)\prod_{j=2}^{L} \alpha^{i_j}_j(x)\right]' =\frac{1}{x\mathcal{E}_1(x)} \sum_{i'} \hat{\mathcal{P}}_{i',1}(x)\prod_{j=2}^{L} \alpha^{i'_j}_j(x),
\end{equation*}
where there are at most $O(n^{N''_1})$ summands and $\sum_{j=2}^{L}i'_j\le n+M'_1$ with $N''_1$ and $M'_1$ independent of $n$ and the $\mathcal{P}_{i',1}(x)$'s are polynomials of $\alpha_1(x),\dots,\alpha_L(x)$ and $x$ that are also independent of $n$.

Using this result and (\ref{calP'}), we obtain
\begin{equation*}
\mathcal{P}'(x) = \frac{\sum_{i'} \mathcal{P}_{i',1}(x)\prod_{j=2}^{L} \alpha^{i_j'}_j(x)}{x\mathcal{E}^2_1(x)},
\end{equation*}
where there are at most $O(n^{N_1})$ summands and $\sum_{j=2}^{L}i'_j\le n+M_1$ with $N_1$ and $M_1$ independent of $n$ and the ${\mathcal{P}}_{i',1}(x)$'s are polynomials of $\alpha_1(x),\dots,\alpha_L(x), \mathcal{E}_1(x), \mathcal{E}'_1(x)$ and $x$ that are also independent of $n$. Since the denominator of $\mathcal{P}'(x)$, namely $x\mathcal{E}^2_1(x)$ is continuous, well-defined and nonzero at $x=1$, the claim in the case $\ell=1$ then follows by Proposition \ref{proppar}.

By induction and the same approach, we can show that for each $\ell$, we have
\begin{equation*}
\mathcal{P}^{(\ell)}(x) = \frac{\sum_{i} \mathcal{P}_{i,\ell}(x)\prod_{j=2}^{L} \alpha^{i_j}_j(x)}{x^{2^{\ell-1}}\mathcal{E}^{2^{\ell}}_1(x)},
\end{equation*}
where there are at most $O(n^{N_\ell})$ summands and $\sum_{j=2}^{L}i_j\le n+M_{\ell}$ with $N_{\ell}$ and $M_{\ell}$ independent of $n$ and the ${\mathcal{P}}_{i,\ell}(x)$'s are polynomials of $\alpha_1(x),\dots,\alpha_L(x), \mathcal{E}^{(l)}_1(x)$ $(1\le l\le \ell)$ and $x$ that are also independent of $n$. Since the denominator of $\mathcal{P}^{(\ell)}(x)$, namely $x^{2^{\ell-1}}\mathcal{E}^{2^{\ell}}_1(x)$ is continuous, well-defined and nonzero at $x=1$, the claim then follows by (\ref{alpha}).
\end{proof}


\section{Upper and Lower Bound for $C$}\label{aCbd}
In the Generalized Lekkerkerker Theorem (Theorem \ref{thm:genlekkerkerker}) we proved the mean $\mu_n$ of $K_n$ satisfies $\mu_n = C n + d + o(\gamma_1^n)$; we now give some bounds on $C$.\\

\begin{lem}\label{lem:boundsforCinLekkerkerker} We have \be \min\left\{\frac{c_1-1}{2},\  \frac{c_1-2}{L}+1\right\} \ \le \ C \le \ \frac{(2L-1)c_1 - 1}{2L} \ < \ c_1. \ee \end{lem}

\begin{proof} If $L=1$ then $C=\frac{1}{2}(s_0+s_1-1)=\frac{c_1-1}{2}$.

If $L\ge 2$, for each $m\in \{0,1,\dots,L-1\}$ we have
\begin{eqnarray}\label{upper}
\nonumber \frac{\frac{1}{2}(s_m+s_{m+1}-1)}{m+1}&\le& \frac{mc_1+(m+1)c_1-1}{2(m+1)}=c_1-\frac{c_1+1}{2(m+1)}\\
&\le& c_1-\frac{c_1+1}{2L}=\frac{(2L-1)c_1-1}{2L}<c_1.
\end{eqnarray}
Note that when $L=1$, $\frac{(2L-1)c_1-1}{2L}=\frac{c_1-1}{2}$, hence (\ref{upper}) holds in this case as well. Thus we get an upper bound for $C$:
\begin{equation*}
C\le \frac{(2L-1)c_1-1}{2L}<c_1.
\end{equation*}

If $m=0$, then
\begin{equation*}
\frac{\frac{1}{2}(s_m+s_{m+1}-1)}{m+1}=\frac{c_1+m-1+c_1+m-1}{2(m+1)}=\frac{c_1-1}{2}.
\end{equation*}

If $m\ge 1$ and $c_1\ge 2$, then
\begin{eqnarray*}
\frac{s_m+s_{m+1}-1}{2(m+1)}\ge \frac{c_1+m-1+c_1+m-1}{2(m+1)}=\frac{c_1-2}{m+1}+1\ge\frac{c_1-2}{L}+1.
\end{eqnarray*}
Thus
\begin{equation}\label{lower}
C\ge \min \left\{\frac{c_1-1}{2},\  \frac{c_1-2}{L}+1\right\}.
\end{equation}
Note that when $c_1=1$, the right-hand side of (\ref{lower}) is 0, and when $L=1$, the right-hand side of (\ref{lower}) is $\min \{\frac{1}{2}(c_1-1),\  c_1-1\}=\frac{1}{2}(c_1-1)$. Thus (\ref{lower}) gives a lower bound for $C$ for all $L$.
\end{proof}


\section{Needed results for Far-Difference Representations}\label{apropAw}\label{ahab}

\subsection{Proof that $h'(1)\neq 0$}\label{ah1}

In this section we prove $h'(1) \neq 0$. This is a key ingredient in the proof of Gaussian behavior in Section \ref{sec:GaussianBehavior}, as this tells us that the variance grows like $n$. If $h'(1) =0$ we would be in the absurd situation where the variance of $K_n$ is bounded independent of $n$; unfortunately, all elementary approaches to derive a contradiction have failed, and thus we must resort to the arguments below.

\begin{proof}
\emph{Case 1: $L = 1$:} When $L=1$, we have $c_1>1$ (see the assumption of Theorem \ref{thm:genZeckendorf}) and $\alpha_1(x)=1+x+x^2+\cdots +x^{c_1-1}$. Thus
\begin{equation*}
\alpha'_1(x)=1+2x+3x^2+\cdots +(c_1-1)x^{c_1-2}
\end{equation*}
and
\begin{equation*}
\twocase{\alpha''_1(x)\ =\ }{2\cdot 1+3\cdot 2x+\cdots +(c_1-1)(c_1-2)x^{c_1-3}}{if $c_1>2$}{0}{if $c_1=2$.}
\end{equation*}
Setting $x=1$ gives
\begin{equation}\label{L=1}
\alpha_1(1)=c_1,\ \ \ \alpha'_1(1)=\frac{c_1(c_1-1)}{2},\ \ \ \alpha''_1(1)=\frac{c_1(c_1-1)(c_1-2)}{3}.
\end{equation}
By Definition (\ref{hd'}), we get
\begin{equation*}
h'(x)=\left(\frac{x\alpha'_1(x)}{\alpha_1(x)}-C\right)'=\frac{\alpha_1(x)\left(\alpha'_1(x)+x\alpha''_1(x)\right)-x\left(\alpha'_1(x)\right)^2}{\alpha_1^2(x)}.
\end{equation*}
Setting $x=1$ yields
\begin{eqnarray*}
\alpha_1^2(1)h'(1)=\alpha_1(1)\left(\alpha'_1(1)+\alpha''_1(1)\right)-\left(\alpha'_1(1)\right)^2
=\frac{c_1^2(c_1-1)(c_1+1)}{12}.
\end{eqnarray*}
Combining this with (\ref{L=1}), we get $h'(1)=(c_1-1)(c_1+1)/12=(c_1^2-1)/12\neq 0$. Note that we can interpret this as the variance of uniform random variables on $\{0,\dots,c_1-1\}$ (see also Footnote 2, p. \pageref{footnote:L=1} for the case of $L=1$).

\emph{Case 2: $L = 2$:} We prove by contradiction for $L\ge 2$. Assuming $h'(1)=0$, we will show that $0=-y'_1(1)/y_1(1)=c_1/2$ and thus deduce a contradiction.

From (\ref{formulaC}) we get
\begin{equation*}
h(x)=\frac{x\alpha'_1(x)}{\alpha_1(x)}-C=-\frac{xy'_1(x)}{y_1(x)}-C.
\end{equation*}
Thus
\begin{equation*}
h'(x)=\left(-\frac{xy'_1(x)}{y_1(x)}\right)'.
\end{equation*}
Plugging in (\ref{yi'}) yields
\begin{equation*}
h'(x)=\left(\frac{\sum_{m=0}^{L-1}\sum_{j=s_m}^{s_{m+1}-1}jx^j y^m_1(x)}{\sum_{m=0}^{L-1}\sum_{j=s_m}^{s_{m+1}-1}(m+1)x^j y^m_1(x)}\right)'.
\end{equation*}
Under the assumption that $h'(1)=0$, we find
\begin{eqnarray*}
\nonumber & &\left(\sum_{m=0}^{L-1}\sum_{j=s_m}^{s_{m+1}-1}j1^j y^m_1(1)\right)'\sum_{m=0}^{L-1}\sum_{j=s_m}^{s_{m+1}-1}(m+1)1^j y^m_1(1)\\
&=&\left(\sum_{m=0}^{L-1}\sum_{j=s_m}^{s_{m+1}-1}(m+1)1^j y^m_1(1)\right)'\sum_{m=0}^{L-1}\sum_{j=s_m}^{s_{m+1}-1}j1^j y^m_1(1),
\end{eqnarray*}
which is equivalent to
\begin{eqnarray}
\label{eq10}& &\frac{\sum_{m=0}^{L-1}\sum_{j=s_m}^{s_{m+1}-1}j y^m_1(1)}{\sum_{m=0}^{L-1}\sum_{j=s_m}^{s_{m+1}-1}(m+1)y^m_1(1)}\\
\nonumber &=&\frac{\sum_{m=0}^{L-1}\sum_{j=s_m}^{s_{m+1}-1}\left(j^2 1^{j-1} y^m_1(1)+mj1^j y^{m-1}_1(1)y'_1(1)\right)}{\sum_{m=0}^{L-1}\sum_{j=s_m}^{s_{m+1}-1}\left((m+1)j1^{j-1} y^m_1(1)+m(m+1)1^j y^{m-1}_1(1)y'_1(1)\right)}.
\end{eqnarray}

From (\ref{formC}), we see that (\ref{eq10}) is exactly $-(y'_1(1))/(y_1(1))$, thus
\begin{eqnarray*}
\nonumber & &y'_1(1)\sum_{m=0}^{L-1}\sum_{j=s_m}^{s_{m+1}-1}\left((m+1)j y^m_1(1)+m(m+1)y^{m-1}_1(1)y'_1(1)\right)\\
&+&y_1(1)\sum_{m=0}^{L-1}\sum_{j=s_m}^{s_{m+1}-1}\left(j^2 y^m_1(1)+mjy^{m-1}_1(1)y'_1(1)\right)=0.
\end{eqnarray*}
Rearranging the terms, we get
\begin{equation*}
\sum_{m=0}^{L-1}\sum_{j=s_m}^{s_{m+1}-1}y^{m-1}_1(1)[j^2y^2_1(1)+(2m+1)jy_1(1)y'_1(1)+m(m+1)\left(y'_1(1)\right)^2]=0.
\end{equation*}
Adding $\sum_{m=0}^{L-1}\sum_{j=s_m}^{s_{m+1}-1} y^{m-1}_1(1)[jy_1(1)y'_1(1)+(m+1)\left(y'_1(1)\right)^2]$ to both sides yields
\begin{eqnarray}
\label{sq} & &\sum_{m=0}^{L-1}\sum_{j=s_m}^{s_{m+1}-1} y^{m-1}_1(1)[j^2y^2_1(1)+(2m+2)jy_1(1)y'_1(1)+(m+1)^2\left(y'_1(1)\right)^2]\\
\nonumber &=&\sum_{m=0}^{L-1}\sum_{j=s_m}^{s_{m+1}-1} y^{m-1}_1(1)[jy_1(1)y'_1(1)+(m+1)\left(y'_1(1)\right)^2]\\
\nonumber &=&y'_1(1)\sum_{m=0}^{L-1}\sum_{j=s_m}^{s_{m+1}-1} [jy^{m}_1(1)+(m+1)y^{m-1}_1(1)y'_1(1)]\\
\nonumber &=&y'_1(1)\left[\sum_{m=0}^{L-1}\sum_{j=s_m}^{s_{m+1}-1} jy^{m}_1(1)+ \frac{y'_1(1)}{y_1(1)}\sum_{m=0}^{L-1}\sum_{j=s_m}^{s_{m+1}-1} (m+1)y^{m}_1(1)\right]\\
\nonumber &=&0
\end{eqnarray}
by (\ref{formC}).

On the other hand, we can rewrite (\ref{sq}) as
\begin{equation*}
\sum_{m=0}^{L-1}\sum_{j=s_m}^{s_{m+1}-1} y^{m-1}_1(1)[jy_1(1)+(m+1)y'_1(1)]^2.
\end{equation*}
Since $y_1(1)>0$, each $jy_1(1)+(m+1)y'_1(1)$ should be 0. Therefore
\begin{equation*}
\forall m\in[0,L-1] \ \ \ {\rm and}\ \ \  \forall j\in[s_m,s_{m+1}-1]:\ \frac{j}{m+1}=-\frac{y'_1(1)}{y_1(1)}.
\end{equation*}
Letting $m=0,\ j=0$ and $m=1,\ j=s_1$ (since $L\ge 2$, $m$ can be 1), we get
\begin{equation*}
\frac{0}{1}=-\frac{y'_1(1)}{y_1(1)}=\frac{s_1}{2}=\frac{c_1}{2},
\end{equation*}
contradiction. Hence $h'(1)\neq 0$.
\end{proof}

\subsection{Proof that $\mu_n(m)=\tilde \mu_n(m)+o(\beta^n_m)$}\label{amun}

In Theorem \ref{thm:genlekkerkerker} we proved that $\mu_n = Cn+d + o(\gamma_1^n)$ and we set $\tilde \mu_n=Cn+d$ ($C$ and $d$ are defined in (\ref{Cd})). Thus $\mu_n=\tilde \mu_n+o(\gamma_1^n)$. We defined
$\tilde \mu_n(m)=\sum_{k}  p_{n,k}(k-\tilde \mu_n)^m/{\Delta_n}$. In this section we prove the following.

\begin{lem} For any $m$, we have $\mu_n(m)=\tilde \mu_n(m)+o(\beta^n_m)$ for some $\beta_m\in (0,1)$.
\end{lem}

\begin{proof}
In the argument below, we will need an upper bound for the number of summands an $N \in [H_n, H_{n+1})$ can have. Let $c=\max\{c_1,c_2,\dots,c_L\}$. As there are $n$ generalized Fibonacci numbers and each one can be taken at most $c$ times, the maximum number of summands such an $H$ can have is $c n$. It is important to note that while the trivial estimate as to the \emph{number} of distinct choices of summands is $c_1^n$, the trivial upper bound for the number of \emph{summands} is $c n$, which is linear and not exponential in $n$.

Since \begin{eqnarray*}
\mu_n(m)& \ = \ & \sum_{k} \frac{p_{n,k}(k-\mu_n)^m}{\Delta_n}=\sum_{k} {\rm{Prob}}(n,k)(k- \mu_n)^m \nonumber\\ \tilde \mu_n(m)& \ = \ & \sum_{k} \frac{p_{n,k}(k-\tilde \mu_n)^m}{\Delta_n} \ = \ \sum_{k} {\rm{Prob}}(n,k)(k-\tilde \mu_n)^m
\nonumber\\ \mu_n & = & \tilde \mu_n + o(\gamma_1^n)\ {\rm{by\ Theorem\ \ref{thm:genlekkerkerker}}},
\end{eqnarray*}
we have
\begin{eqnarray*}
\left|\mu_n(m)-\tilde \mu_n(m)\right| &=&\left|\sum_{k} {\rm{Prob}}(n,k)(k-\tilde \mu_n + o(\gamma_1^n))^m-\sum_{k} {\rm{Prob}}(n,k)(k-\tilde \mu_n)^m\right|\\
\nonumber &=&\left|o(\gamma_1^n)\sum_{k} {\rm{Prob}}(n,k)\sum_{i=1}^m \ncr{m}{i}(k-\tilde \mu_n)^{m-i} o^{i-1}(\gamma_1^n)\right|\\
\nonumber &\ll& \left|o(\gamma_1^n)\sum_{k} {\rm{Prob}}(n,k) (k+\tilde \mu_n+1)^m  \right| \nonumber\\ & \le & \left|o(\gamma_1^n)(cn+Cn+d+2011)^m\sum_k {\rm Prob}(n,k) \right|\\
\nonumber &\le& \left|o(\gamma_1^n)(C+c+|d|+2011)^m n^m \cdot 1\right|\\
\nonumber &=& o(\beta_m^n)
\end{eqnarray*}
for some $\beta_m\in (0,1)$.
\end{proof}


\subsection{Proof of Proposition \ref{propAw}}
\begin{proof}[Proof of Proposition \ref{propAw}]
Since the roots of $\hat{A}_w(z)$ are continuous and (a), (b) hold for $x=1$, they also hold for a sufficiently small neighborhood $I_{\varepsilon}$ of 1.

For (c), since $e_i(w)$ is a root of $\hat{A}_w(z)$, we have
\begin{equation}\label{hatAiw}
0=1-e_i(w)-(w^a+w^b)e^4_i(w)-w^{a+b}e^6_i(w)-w^{a+b}e^7_i(w).
\end{equation}
For a small increment $\Delta w$, we have
\begin{align}\label{hatAidw}
\nonumber 0=&\ 1-e_i(w+\Delta w)-[(w+\Delta w)^a+(w+\Delta w)^b]e^4_i(w+\Delta w)\\
& -(w+\Delta w)^{a+b}e^6_i(w+\Delta w)-(w+\Delta w)^{a+b}e^7_i(w+\Delta w).
\end{align}
Subtracting (\ref{hatAidw}) from (\ref{hatAiw}) yields
\begin{align}\label{eq16}
\nonumber 0\ =\ &\ e_i(w+\Delta w)-e_i(w)+(w^a+w^b)[e^4_i(w+\Delta w)-e^4_i(w)]\\
\nonumber & \ + \ [(w+\Delta w)^a+(w+\Delta w)^b-w^a-w^b]e^4_i(w+\Delta w)\\
\nonumber & \ + \ w^{a+b}[e^6_i(w+\Delta w)-e^6_i(w)]+[(w+\Delta w)^{a+b}+w^{a+b}]e^6_i(w+\Delta w)\\
\nonumber &\ + \ w^{a+b}[e^7_i(w+\Delta w)-e^7_i(w)]+[(w+\Delta w)^{a+b}-w^{a+b}] e^7_i(w+\Delta w)\\
\nonumber =&\ \  [e_i(w+\Delta w)-e_i(w)]\left[1+(w^a+w^b)\sum_{j=0}^{3}e^j_i(w+\Delta w)e^{3-j}_i(w)\right.\\
\nonumber &\left. \ + \ w^{a+b}\sum_{j=0}^{5}e^j_i(w+\Delta w)e^{5-j}_i(w)+w^{a+b}\sum_{j=0}^{6}e^j_i(w+\Delta w)e^{6-j}_i(w)\right]\\
\nonumber &\ +\ \Delta w \left[\left(\frac{(w+\Delta w)^a-w^a}{\Delta w}+\frac{(w+\Delta w)^b-w^b}{\Delta w}\right)e^4_i(w+\Delta w)\right.\\
&\left.\ + \ \frac{(w+\Delta w)^{a+b}-w^{a+b}}{\Delta w}\left(e^6_i(w+\Delta w)+e^7_i(w+\Delta w)\right)\right].
\end{align}
Since $e_i(w)$ is continuous, the coefficient of $[e_i(w+\Delta w)-e_i(w)]$ converges as $\Delta w \rightarrow 0$ and its limit is
\begin{equation*}
1+4(w^a+w^b)e^3_i(x)+6w^{a+b}e^5_i(w)+7w^{a+b}e^6_i(w),
\end{equation*}
which is exactly $-\hat{A}'_w(z)$ (with respect to $z$) at $e_i(w)$ and therefore nonzero since $\hat{A}_w(z)$ has no multiple roots. Since $w^a$, $w^b$ and $w^{a+b}$ are differentiable at $w=1$, the coefficient of $\Delta w$ in (\ref{eq16}) also converges as $\Delta w \rightarrow 0$ and its limit is
\begin{equation*}
\left(aw^{a-1}+bw^{b-1}\right)e^4_i(w)+(a+b)w^{a+b-1}[e^6_i(w)+e^7_i(w)].
\end{equation*}
Thus $e'_i(w)$ exists and
\begin{eqnarray}
 e'_i(w) & \ = \ & \lim_{\Delta w \to 0} \frac{e_i(w+\Delta w)-e_i(w)}{\Delta w}\nonumber\\
&= & -\frac{\left(aw^{a-1}+bw^{b-1}\right)e^4_i(w) +(a+b)w^{a+b-1}[e^6_i(w)+e^7_i(w)]}{1+4(w^a+w^b)e^3_i(x)+6w^{a+b}e^5_i(w)+7w^{a+b}e^6_i(w)}.
\end{eqnarray}
Since the denominator of $e'_i(w)$ is not zero, by the same approach in Proposition \ref{diff}, we can show that $e_i(w)$ is $\ell$-times differentiable for any $\ell\ge 1$.

Finally, with (a), Part (d) can be shown in the exactly same way as in Proposition \ref{proppar}(b).
\end{proof}


\subsection{Proof that $h'_{a,b}(1)\neq 0$}

Analogously to Appendix \ref{ah1}, this is important in the proof of the Gaussian behavior in Section \ref{subsubsecak+bl}, as this tells us that the variances grows like $n$.

\begin{proof}
By (\ref{e'}), we have
\begin{equation}\label{e'e}
\frac{we'_1(w)}{e_1(w)}= -\frac{\left(aw^{a}+bw^{b}\right)e^3_1(w)+(a+b)w^{a+b}[e^5_1(w)+e^6_1(w)]} {1+4(w^a+w^b)e^3_1(w)+6w^{a+b}e^5_1(w)+7w^{a+b}e^6_1(w)}.
\end{equation}
Thus
\begin{eqnarray}\label{h'ab}
\nonumber & &\hat h'_{a,b}(w)\\
\nonumber &=&\left[\frac{\left(aw^{a}+bw^{b}\right)e^3_1(w) +(a+b)w^{a+b}(e^5_1(w)+e^6_1(w))} {1+4(w^a+w^b)e^3_1(w)+6w^{a+b}e^5_1(w)+7w^{a+b}e^6_1(w)}\right]'\\
\nonumber &=&\left[\left[\left(aw^{a}+bw^{b}\right)e^3_1(w) +(a+b)w^{a+b}(e^5_1(w)+e^6_1(w))\right]'\right.\\
\nonumber & &\cdot \left[1+4(w^a+w^b)e^3_1(w)+6w^{a+b}e^5_1(w)+7w^{a+b}e^6_1(w)\right]\\
\nonumber & &-[\left(aw^{a}+bw^{b}\right)e^3_1(w)+(a+b)w^{a+b}\left(e^5_1(w)+e^6_1(w)\right)]\\
\nonumber & &\left.\cdot  \left[1+4(w^a+w^b)e^3_1(w)+w^{a+b}\left(6e^5_1(w)+7e^6_1(w)\right)\right]'\right]\\
& &\cdot \left[1+4(w^a+w^b)e^3_1(w)+w^{a+b}\left(6e^5_1(w)+7e^6_1(w)\right)\right]^{-2}.
\end{eqnarray}
Setting $w=1$ in (\ref{e'e}) and using $e_1(1)=\Phi$, we get
\begin{equation*}
\frac{e'_1(1)}{e_1(1)}= -\frac{(a+b)(\Phi^3+\Phi^5+\Phi^6)} {1+8\Phi^3+6\Phi^5+7\Phi^6}=-\frac{a+b}{10}.
\end{equation*}
Thus
\begin{equation}\label{e1'}
e'_1(1)\ = \ -\frac{a+b}{10}\ \Phi.
\end{equation}
Plugging $e_1(1)=\Phi$ and (\ref{e1'}) into (\ref{h'ab}) with $w=1$ yields
\begin{eqnarray}
\nonumber \hat h'_{a,b}(1)&=&\left[\Phi^5\left[10\left (a^2+b^2\right)+(a+b)^2\left(-3+10\Phi-5\Phi^2-6\Phi^3\right)\right]\right.\\
\nonumber & &\left.-\Phi^5(a+b)^2\left(1.6+3\Phi^2+2.8\Phi^3\right)\right]/(100\Phi^4)\\
&=&\frac{\sqrt{5}-1}{200}\left[10\left (a^2+b^2\right)-\frac{20-\sqrt{5}}{5}(a+b)^2\right]
\end{eqnarray}
Since $\frac{20-\sqrt{5}}{5}<4$ and $a^2+b^2>0$, we have
\begin{equation*}
\frac{20-\sqrt{5}}{5}\left(a^2+b^2\right)<4(a+b)^2\le 8\left (a^2+b^2\right)<10\left(a^2+b^2\right).
\end{equation*}
Hence $\hat h_{a,b}'(1)\neq 0$.
\end{proof}


\section{Notations and Definitions}\label{sec:notationdef}

\noindent We list the various notations and terminology in the paper, followed by the page number of its first occurrence or definition.\\ \

$a$, p. \pageref{ab}: a real number.

$a_i$, p. \pageref{a_i}: the $i^{\textsuperscript{th}}$ coefficient of a legal decomposition.

$A_i$, p. \pageref{A_i}: the corresponding random variable of $a_i$.

$A(y)$, p. \pageref{G}: $\mathscr{A}(x,y)$ as polynomial of $y$.

$\hat{A}(z)$, p. \pageref{hatA}: the denominator of $\hat{\mathscr{G}}(x,y,z)$.

$\hat{A}_w(z)$, p. \pageref{hatAw}: $\hat{A}(z)$ when $x=w^a$ and $y=w^b$.

$\mathcal{A}(y)$,  p. \pageref{calA}: $y^L A(1/y)$.

$\mathscr{A}(x,y)$, p. \pageref{A}.

$\alpha_i(x)$, p. \pageref{alphaix}: $\left(y_i(x)\right)^{-1}$.

$b$, p. \pageref{ab}: a real number.

$b_i(x)$, p. \pageref{By}: polynomials of $x$.

$B(y)$, p. \pageref{G}: $\mathscr{B}(x,y)$ as polynomial of $y$.

$\mathscr{B}(x,y)$, p. \pageref{B}.

$\beta_m$, p. \pageref{betam}: some constant in $(0,1)$ indicating the decaying rate.

$c_i$, p. \pageref{c_i}: the $i^{\textsuperscript{th}}$ coefficient of a linear recurrence relation.

$C$, p. \pageref{thm:genlekkerkerker}: a constant, the coefficient of $n$ in the generalized Lekkerkerker's Theorem.

$\hat C_{a,b}$, p. \pageref{EaK+bL}: $-e'_1(1)/e_1(1)$.

$d$, p. \pageref{thm:genlekkerkerker}: the constant term in the generalized Lekkerkerker's Theorem.

$d'$, p. \pageref{hd'}: $1-d$.

$D(L,M)$, p. \pageref{DLM}: the parenthesized part in (\ref{n6}).

$\Delta_n$, p. \pageref{index:Deltan}:  $H_{n+1}-H_n$.

$\mathcal{D}_n$, p. \pageref{mathcalD_n}: set of legal decompositions with $H_n$ the largest term.

$e_i(w)$, p. \pageref{propAw}: root of $\hat{A}_w(z)$.

$\hat E(x)$, p. \pageref{E}: $\prod_{j\neq 1}\left(z_j(x)-z_1(x)\right)$.

$\mathbb{E}[X]$: the expected value of random variable $X$.

$\epsilon$, p. \pageref{epsilon}: a number in $(0,1)$.

$\varepsilon$, p. \pageref{varepsilon}: a number in $(0,\epsilon)$.

\textit{Far-difference representation}, p. \pageref{far-difference representation}.

$f_{i,m}(x)$, p. \pageref{fimx}, p. \pageref{propfim}.

$F_j(x)$, p. \pageref{Fjx}: $\tilde g_{j,1}(x)$.

$F_n$, p. \pageref{F_n}: the $n^{\textsuperscript{th}}$ Fibonacci number with $F_1=1$ and $F_2=2$.

$g(x)$, p. \pageref{gx}: $\sum_{k>0} p_{n,k}x^k$.

$g_i(x)$, p. \pageref{gix}: $xq_i(x) \alpha_i^n(x)$.

$g_{j,i}(x)$, p. \pageref{gji}: $\frac{q_i(x) \alpha_i^n(x)}{x^{\tilde \mu_n}}$.

$\tilde{g}_m(x)$, p. \pageref{gmu}: $g(x)/x^{\tilde{\mu}_n+1}$.

$G(y)$, p. \pageref{G}: $\mathscr{G}(x,y)$ as polynomial of $y$.

$\mathscr{G}(x,y)$, p. \pageref{scrGfib}, p. \pageref{propG}: $\sum_{n,k>0} p_{n,k}x^ky^n$.

 $\hat{\mathscr{G}}(x,y,z)$, p. \pageref{hatmathscrG}: $\sum_{n>0,k>0,l\ge 0}p_{n,k,l}x^ky^lz^n$.

$\gamma_{\ell}$, p. \pageref{thm:genlekkerkerker}, p. \pageref{claimo1}: some constant in $(0,1)$ indicating the
decaying rate.

\textit{Good recurrence relation}, p. \pageref{def:goodrecurrence}.

$h(x)$, p. \pageref{hd'}: $x\alpha'_1(x)/\alpha(x)-C$.

$h_i(x)$, p. \pageref{hix}: $\alpha_1^n(x)x^{-\tilde \mu_n}xf_{i,m}(x)n^i$.

$\hat h'_{a,b}$, p. \pageref{EaK+bL}: $-we'_1(w)/e_1(w)-\hat C_{a,b}$.

$H_n$, p. \pageref{def:goodrecurrence}: a Positive Linear Recurrence Sequence.

$k$-\textit{summand decomposition}, \pageref{page:ksummand}

$L$, p. \pageref{c_i}: the order of the recurrence relation.

\textit{Number of summands}, p. \pageref{summands}.

\textit{Legal decomposition/sequence}, p. \pageref{legal}.

$K_n$, p. \pageref{K_n}: the corresponding random variable of $k$ for integers in $[H_n,H_{n+1})$.

$\mathcal{K}_n$, p. \pageref{mathcalK_n}: the corresponding random variable denoting the number of positive summands.

$\mathcal{L}_n$, p. \pageref{mathcalK_n}: the corresponding random variable denoting the number of negative summands.

$M$, p. \pageref{M}: $s_L$.

$\mu_n$, p. \pageref{thm:genlekkerkerker}: the mean of $K_n$.

$\mu_n(m)$, p. \pageref{munm}: the $m{\textsuperscript{th}}$ moment of $K_n-\mu_n$.

$\tilde{\mu_n}(m)$, p. \pageref{tildemunm}: the $m${\textsuperscript{th}} moment of $K_n-(Cn+d)$.

$\varphi$, p. \pageref{varphi}: the golden mean $(\sqrt{5}+1)/2$.

$p_{n,k}$, p. \pageref{pnk}: the number of integers in $[H_n, H_{n+1})$ with $k$-summand legal decomposition.

$p_{n,k,l}$, p. \pageref{p_{n,k,l}}: the number of far-difference representations of integers in $(S_{n-1}, S_n]$ with $k$

\ \ \ \ \  positive summands and $l$ negative summands.

\textit{Positive Linear Recurrence Sequence (PLRS)}, p. \pageref{def:goodrecurrence}.

$q_i(x)$, p. \pageref{qi}.

$\hat{q}(w)$,  p. \pageref{hatqw}.

$r_{u,v}$, p. \pageref{ruv}: constant determined by $u$ and $v$.

$s_m,s'_m$, p. \pageref{defs_m}: partial sum of $c_i$'s.

$s_{u,v}(x)$, p. \pageref{ruv}: function of $x$.

$S_n$, p. \pageref{S_n}: $\sum_{0<n-4i\le n} F_{n-4i}$ for positive $n$ and 0 otherwise.

$\mathcal{S}_n$, p. \pageref{mathcalS_n}: the set of integers in $[H_n, H_{n+1})$.

$\sigma_n$, p. \pageref{sigma_n}: the standard deviation of $K_n$.

$t_{i,m},\ t^{(\ell)}_{i,m}$, p. \pageref{tim}: $f_{i,m}(1)$, $f^{(\ell)}_{i,m}(1)$.

$T_v(x)$, p. \pageref{Tvdef}: $\sum_{u=v}^{\infty} r_{u,v}x^{u-v}$.

$\tau_m$, p. \pageref{taum}: some constant in $(0,1)$ indicating the decaying rate.

$w$, p. \pageref{ab}.

$y_i(x)$, p. \pageref{yixfib}, p. \pageref{yix}: the roots of $A(y)$.

$\langle y^n\rangle G(y)$, \pageref{angle}: the coefficient of $y^n$ in $G(y)$.

\textit{Zeckendorf decomposition}, p. \pageref{Zeckendorfdec}.

$(2m-1)!!$, p. \pageref{!!}: the double factorial, $(2m-1)(2m-3)\cdots 1$.


\ \\

\end{document}